\documentclass{amsart}

\usepackage[utf8x]{inputenc}
\usepackage[english]{babel}

\usepackage{cite}
\usepackage{amsmath}
\usepackage{amsfonts}
\usepackage{amssymb}
\usepackage{textcomp}
\usepackage[lined,ruled,linesnumbered,norelsize]{algorithm2e}
\usepackage{booktabs}
\usepackage{url}

\usepackage[margin=2.90cm]{geometry}
\usepackage{listings}
\usepackage{graphicx}
\usepackage{xcolor}

\newcommand{\forcebold}[1]{\boldsymbol{#1}}
\newcommand{\Seg}[1]{\operatorname{Seg}(#1)}
\newcommand{\Var}[1]{\mathcal{#1}}
\newcommand{\Sec}[2]{\sigma_{#1}(#2)}
\newcommand{\SecZ}[2]{\sigma^0_{#1}(#2)}
\newcommand{\tensor}[1]{\mathfrak{#1}}
\newcommand{\vect}[1]{\mathbf{#1}}
\newcommand{\sten}[3]{\vect{#1}_{#2}^{#3}}
\newcommand{\Tang}[2]{\mathrm{T}_{#1} {#2}}

\newcommand{\cspan}{\operatorname{span}}

\newcommand{\diag}{\operatorname{diag}}

\newcommand{\vecc}[1]{\operatorname{vec}(#1)}

\newcommand{\C}{\mathbb{C}}
\newcommand{\R}{\mathbb{R}}
\newcommand{\F}{\mathbb{F}}

\newcommand{\refthm}[1]{{Theorem \ref{#1}}}
\newcommand{\reflem}[1]{{Lemma \ref{#1}}}
\newcommand{\refeqn}[1]{{(\ref{#1})}}
\newcommand{\refdef}[1]{{Definition \ref{#1}}}

\newcommand{\refalg}[1]{Algorithm \ref{#1}}
\newcommand{\refsec}[1]{{Section \ref{#1}}}

\newcommand{\refcor}[1]{{Corollary \ref{#1}}}
\newcommand{\reffig}[1]{{Figure \ref{#1}}}

\newcommand{\refprop}[1]{{Proposition \ref{#1}}}

\newcommand{\refapp}[1]{{Appendix \ref{#1}}}
\newcommand{\refconj}[1]{{Conjecture \ref{#1}}}

\newtheorem{theorem}{Theorem}
\newtheorem{lemma}[theorem]{Lemma}
\newtheorem{proposition}[theorem]{Proposition}
\newtheorem{corollary}[theorem]{Corollary}

\newtheorem{remark}[theorem]{Remark}
\newtheorem{conjecture}{Conjecture}
\theoremstyle{definition}
\newtheorem{definition}{Definition}

\begin{document}

\title{A condition number for the tensor rank decomposition}
\author{Nick Vannieuwenhoven}
\address{KU Leuven, Department of Computer Science, B-3001 Leuven, Belgium}
\email{nick.vannieuwenhoven@cs.kuleuven.be}
\thanks{The author's research was supported by a Postdoctoral Fellowship of the Research Foundation--Flanders (FWO)}


\begin{abstract}
The tensor rank decomposition problem consists of recovering the unique set of parameters representing a robustly identifiable low-rank tensor when the coordinate representation of the tensor is presented as input.
A condition number for this problem measuring the sensitivity of the parameters to an infinitesimal change to the tensor is introduced and analyzed. It is demonstrated that the absolute condition number coincides with the inverse of the least singular value of Terracini's matrix. Several basic properties of this condition number are investigated.
\end{abstract} 

\maketitle

\allowdisplaybreaks

%
%
\section{Introduction}
A tensor $\tensor{A} \in \F^{n_1} \otimes \cdots \otimes \F^{n_d}$, where $\F = \C$ or $\R$, and $n_k \ge 2$, $k=1,\ldots,d$, lives in the tensor product of vector spaces. It can be represented as a $d$-array after choosing bases. If the coordinate representation of $\tensor{A}$ is given with respect to the standard tensor basis, then we can explicitly write
\begin{align*} 
 \tensor{A} = \sum_{i_1=1}^{n_1} \sum_{i_2=1}^{n_2} \cdots \sum_{i_d=1}^{n_d} a_{i_1, i_2, \ldots, i_d} \sten{e}{i_1}{1} \otimes \sten{e}{i_2}{2} \otimes \cdots \otimes \sten{e}{i_d}{d},
\end{align*}
where $\sten{e}{i}{k}$ is the $i$th standard basis vector of $\F^{n_k}$.
Tensors of the form $\vect{a}^1 \otimes \vect{a}^2 \otimes \cdots \otimes \vect{a}^d$ are called rank-$1$ tensors. As in the $2$-factor case, i.e., matrices, the above form is, in general, not the minimal way of writing a tensor as a linear combination of rank-$1$ tensors. This leads to a different type of tensor decomposition: 
\begin{align}\label{eqn_cpd}
\tensor{A} = \sum_{i=1}^r \sten{a}{i}{1} \otimes \sten{a}{i}{2} \otimes \cdots \otimes \sten{a}{i}{d}, \quad\mbox{where } \sten{a}{i}{k} \in \F^{n_k}.
\end{align}
If the number of terms $r$ is minimal, then this decomposition called a \emph{tensor rank decomposition} \cite{Hitchcock1927a,Hitchcock1927}. The number $r$ is then called the \emph{rank} of the tensor. In the literature, decomposition \refeqn{eqn_cpd} is sometimes also called Candecomp \cite{Carroll1970}, Parafac \cite{Harshman1970}, Canonical Polyadic decomposition, or CP decomposition. Here, we will call \refeqn{eqn_cpd}, with $r$ not necessarily minimal, an $r$-term decomposition, and if $r$ is minimal, then we call it a rank-$r$ decomposition. 
One key property of the latter---but not of the former---is that a rank-$r$ tensor with $d \ge 3$ often admits an \emph{essentially unique} decomposition \cite{Kruskal1977,COV2013,DdL2015}: the rank-$1$ tensors appearing in decomposition \refeqn{eqn_cpd} are uniquely determined. 
For instance, the well-known Kruskal criterion \cite{Kruskal1977} states that \refeqn{eqn_cpd} is the essentially unique decomposition if 
\[
 r \le \frac{1}{2}( k_1 + k_2 + k_3 - 2 ),
\]
where $k_i$ is the \emph{Kruskal rank} of $A_i = \{ \sten{a}{j}{i} \}_{j=1}^r$, i.e., the integer $k_i$ such that every subset of $A_i$ of cardinality $k_i$ forms a linearly independent set.

\subsection{Motivation}
The essential uniqueness of a tensor rank decomposition turns out to be very useful in several applications, such as in the parameter identification problem in latent variable models \cite{Allman2009,Anandkumar2014b}. Models whose parameters may be inferred with this decomposition (and symmetrical variants) were recently surveyed in a unified tensor-based framework by Anandkumar, Ge, Hsu, Kakade, and Telgarsky \cite{Anandkumar2014b}; they include exchangeable single topic models, hidden Markov models, and Gaussian mixture models. The blind source separation problem in signal processing is another example where the uniqueness of the (symmetric) tensor rank decomposition of a cumulant tensor is key in identifying the underlying statistically independent signals \cite{Comon1994}. A crucial application of the essential uniqueness of the tensor rank decomposition is found in the chemometrics literature, in particular in fluorescence spectroscopy where its use was introduced by Appellof and Davidson \cite{Appellof1981}. 
Since I will treat this application as a prototypical use of the tensor rank decomposition in parameter identification problems---such as all aforementioned applications---I recall the discussion in \cite[Section 2.3]{TomasiPhD}. For a sufficiently diluted solution, some type of molecules called fluorophores will absorb a certain amount of light at the $i$th wavelength and re-emit a fraction of the absorbed energy as light at the $j$th wavelength. Mathematically, $x_{i,j} = \chi \lambda_i \mu_j$, where $x_{i,j}$ is the intensity of the light emitted at the $j$th wavelength when the fluorophore is excited by light at the $i$th wavelength, $\lambda_i$ models the fraction of light absorbed at the $i$th wavelength, $\mu_j$ models the fraction of light emitted at the $j$th wavelength, and $\chi$ is some chemical constant proportional to the concentration of the fluorophore in the solution. When $r$ fluorophores occur in a solution and the interactions among them are negligible, then their signal is additive: $x_{i,j} \approx \sum_{\ell=1}^r \chi_\ell \lambda_{i,\ell} \mu_{j,\ell}$. If several of such mixtures with varying concentrations of fluorophores are simultaneously analyzed, the model for the intensity of the emitted light is a tensor rank decomposition:
\(
 x_{i,j,k} \approx \sum_{\ell=1}^r \chi_{k,\ell} \lambda_{i,\ell} \mu_{j,\ell};
\)
the model is only approximately valid, as it rest on some assumptions such as a sufficiently diluted solution and no interactions among the fluorophores. As stated in \cite[Section 2.3]{TomasiPhD}, ``the uniqueness property of [tensor rank decompositions] guarantees that, if [the above model] is adequate, [...] the estimated loading vectors for the $\ell$th component are readily interpretable as concentration profiles [$\forcebold{\chi}_\ell$] and emission and excitation spectra ([$\forcebold{\mu}_\ell$] and [$\forcebold{\lambda}_\ell$] respectively) of the $\ell$th fluorophore.'' This allows for an identification of the fluorophores present in a chemical mixture of unknown composition.
While the tensor $\tensor{X}$ may be obtained from physical measurements, such a process is always subject to measurement errors because of the finite precision of the measuring equipment. Consequently, only an approximate low-rank decomposition may be anticipated.

The foregoing application highlights that one rarely works with the true low-rank tensor, however one still desires to interpret the uniquely defined rank-$1$ tensors appearing in the true decomposition. One desires knowing, or at least estimating, how much the individual rank-$1$ terms in the decomposition computed from the approximate data differ from the true underlying rank-$1$ tensors. Such an analysis will be the topic of the paper. In order to state the main question more precisely, one needs to specify more carefully which problem we are trying to solve.
In the literature, the parameters of an $r$-term decomposition are often succinctly represented using \emph{factor matrices}. For the present purpose, dealing with a particular vectorization of the factor matrices is more convenient. 
Define the bijections
\begin{align*}
 \operatorname{vec}: \F^{n_1} \times \cdots \times \F^{n_d} &\to \F^{\Sigma+d} &
 \operatorname{vecr}: (\F^{n_1} \times \cdots \times \F^{n_d})^{\times r} &\to \F^{r(\Sigma+d)} \\
(\sten{a}{i}{1}, \sten{a}{i}{2}, \ldots, \sten{a}{i}{d}) &\mapsto
\begin{bmatrix} 
\sten{a}{i}{1} \\ 
\sten{a}{i}{2} \\
\vdots \\ 
\sten{a}{i}{d} 
\end{bmatrix}
, \quad\text{and} &
 ( b_1, b_2, \ldots, b_r) &\mapsto 
\begin{bmatrix}
\vecc{b_1} \\
\vecc{b_2} \\
\vdots \\
\vecc{b_r}
\end{bmatrix},
\end{align*}
where $\Sigma = \sum_{k=1}^d (n_k - 1)$. Let $\Pi = \prod_{k=1}^d n_k$.
With these definitions in place, we can introduce the ``tensor computation'' function
\begin{align*}
 f: \F^{r(\Sigma+d)} &\to \SecZ{r}{\Var{S}_\F} \subset \F^{\Pi} \\
\vect{p} = \operatorname{vecr}(b_1,b_2,\ldots,b_r) &\mapsto \sum_{i=1}^r \sten{a}{i}{1} \otimes \sten{a}{i}{2} \otimes \cdots \otimes \sten{a}{i}{d},
\end{align*}
which takes the \emph{vectorized factor matrices} $\vect{p}$ as input and sends them to the tensor they represent.
Now we can formally state that the \emph{tensor rank decomposition problem} with as input the rank-$r$ tensor $\tensor{A} \in \F^{n_1} \otimes \cdots \otimes \F^{n_d} \cong \F^\Pi$ consists of finding a vector $\vect{x} \in \F^{r(\Sigma+d)}$ in the preimage of $f$ at $\tensor{A}$, so that $\tensor{A} = f(\vect{x})$. We adopt the convenient notation
\[
 f^{\dagger}(\tensor{A}) = \{ \vect{x}\in\F^{r(\Sigma+d)} \;|\; f(\vect{x}) = \tensor{A} \}
\]
for denoting the fiber of $f$ at $\tensor{A}$. Since a tensor rank decomposition is not unique in the strict sense, one should be wary that $f$ is \emph{never} an injective function. Hence, $f^{\dagger}$ is always a \emph{set-valued map} in the sense of Dontchev and Rockafellar \cite{DR2009}. If $\tensor{A}$ does not admit an exact rank-$r$ decomposition, then $f^\dagger(\tensor{A}) = \emptyset$.

The goal of this paper is understanding the relationship between the essentially unique decompositions $f^{\dagger}(\tensor{A})$ and $f^{\dagger}(\tensor{A}')$ when $\|\tensor{A} - \tensor{A}'\| \le \epsilon$ and both $\tensor{A}$ and $\tensor{A}'$ are $r$-identifiable tensors through a structured perturbation analysis of the inverse problem $\tensor{A} = f(\vect{x})$. 
Specifically, I will derive a \emph{condition number} of the tensor rank decomposition problem that captures the first-order behavior of aforementioned relationship.

A condition number of a function $g: \F^m \to \F^n$ is an elementary, well-entrenched property in numerical analysis \cite{BC2013,Wilkinson1965,MPT,Higham1996} that measures the sensitivity of $g$ to infinitesimal changes in the input. In this paper, I consider the \emph{structured} condition number \cite{BC2013}, which differs only from the usual definition in restricting the domain of $g$ to some connected domain $\Var{D} \subset \F^m$; for instance, $\Var{D}$ could be the set of rank-$r$ tensors. Formally, the structured condition number measures the maximum relative change at the output of the function when infinitesimally perturbing the input $\vect{x} \in \Var{D}$ to another $\vect{x}+\Delta\vect{x} \in \Var{D}$: 
\begin{align}\label{eqn_general_condition_number}
\kappa_{\alpha,\beta}(\vect{x}) = \lim_{\epsilon\to0} \max_{\substack{\|\Delta\vect{x}\|\le\epsilon, \\ \vect{x}+\Delta\vect{x} \in \Var{D} }} \frac{ \| g(\vect{x}+\Delta\vect{x}) - g(\vect{x}) \|_\alpha / \| g(\vect{x}) \|_\alpha }{ \| \Delta\vect{x} \|_\beta / \|\vect{x}\|_\beta },
\end{align}
where $\|\cdot\|_\alpha$ and $\|\cdot\|_\beta$ are norms on $\F^n$ and $\F^m$ respectively, and $\vect{x} \in \Var{D} \subset \F^m$. If $\Var{D} = \F^{m}$ then the usual condition number is obtained. 

The main application of condition numbers consists of obtaining an asymptotically sharp \emph{a posteriori} bound on the relative forward error, by grace of the observation 
\begin{align}\label{eqn_forward_error_estimate}
 \frac{\| g(\vect{x}+\Delta\vect{x}) - g(\vect{x}) \|_\alpha}{\| g(\vect{x}) \|_\alpha} \lesssim \kappa_{\alpha,\beta}(\vect{x}) \cdot \frac{ \| \Delta\vect{x} \|_\beta }{ \|\vect{x}\|_\beta },
\end{align}
where $\vect{x}, \vect{x}+\Delta\vect{x} \in \Var{D}$ and $\|\Delta\vect{x}\|_\beta$ should be sufficiently close to zero.
That is, \emph{the relative forward error is asymptotically bounded by the product of the relative backward error multiplied with the condition number}, which is a well-known rule of thumb in numerical analysis. 
This observation is particularly important in applications, because the forward error usually cannot be computed for a practical problem, as you typically would have to know the true solution $g(\vect{x}+\Delta\vect{x})$; here $\vect{x}$ would represent the observed input to problem $g$ which is typically corrupted in varying degrees by measurement, modeling, computation and representation errors, whereas $\vect{x}+\Delta\vect{x}$ would represent the true input. 
The key advantage of a computable condition number of a problem $g$ is that the forward error can be estimated from \emph{a posteriori} information, namely the backward error and the condition number. 
For example, in the aforementioned chemometrics application, the underlying physical model ensures that the theoretical tensor $\tensor{A}$ admits a unique rank decomposition of low rank as in \refeqn{eqn_cpd}, say, with vectorized factor matrices $\vect{x}+\Delta\vect{x}$: $\tensor{A} = f(\vect{x}+\Delta\vect{x})$. Each of the rank-$1$ terms uniquely corresponds with an individual chemical compound in the fluorescent mixture. In practice, the tensor $\widetilde{\tensor{A}}$ that one obtained by measuring the reflected light in an experiment contains measurement errors. This may be regarded as a perturbation of $\tensor{A}$ of small Frobenius norm, say, $\|\tensor{A}-\widetilde{\tensor{A}}\|_F \le \mu$. A low-rank approximate decomposition of $\widetilde{\tensor{A}}$ is then constructed using a numerical method for computing an $r$-term decomposition. This results in an approximation $\tensor{A} \approx \widetilde{\tensor{A}} \approx \widehat{\tensor{A}} = f(\vect{x})$ with factor matrices $\vect{x}$. If $\|\widehat{\tensor{A}} - \widetilde{\tensor{A}}\|_F = \nu$, then we know that the distance between the approximate decomposition $\widehat{\tensor{A}} = f(\vect{x})$ computed from the sampled data $\widetilde{\tensor{A}}$ and the true tensor corresponding to the physical phenomenon $\tensor{A}$ is at most $\mu + \nu$. If we compute the condition number $\kappa$ of the tensor rank decomposition $\widehat{\tensor{A}} = f(\vect{x})$, then it reveals an upper bound on the relative difference between the computed approximate factor matrices $\vect{x}$ and the true factor matrices $\vect{x}+\Delta\vect{x}$. If this relative forward error is small, then one can safely match the approximate rank-$1$ terms to one or more chemical compounds, as the true excitation spectra will have been well approximated. 
In several data-analytic applications where the factor matrices themselves are analyzed and linked to certain phenomena or interpretations, all too often the analysis is stopped when a small backward error $\mu+\nu$ is obtained without considering the well-known fact in numerical analysis that a small backward error must not entail a small forward error. A small backward error $\|\tensor{A}-\widehat{\tensor{A}}\|_F \le \mu + \nu$ entails a small forward error $\|\Delta\vect{x}\|$ only if the problem is \emph{well-conditioned}, i.e., if the condition number is small.

\subsection{Contributions}
The main contribution of this paper is a definition of a condition number for the rank decomposition problem $f^\dagger$. This condition number is well-defined for every rank decomposition that satisfies a technical condition called robust $r$-identifiability. This condition is generally satisfied, such as for a generic\footnote{A property that may be admitted by the elements of a set $S$ is called generic if all elements of $S \setminus N$ admit the property, where $N \subset S$ has measure zero in $S$ in some topology. Note that the notion of generality depends on the underlying topology. A property may, hence, be generic with respect to one topology, while it is not generic in another.} rank-$r$ decomposition in every tensor space that is generically $r$-identifiable---see \reflem{lem_robust_idf}; most spaces are conjectured to be of this kind, see, e.g., \cite{BCO2013,COV2013}. The prime result, \refthm{thm_condition}, may be stated informally as follows.

\begin{theorem} \label{thm_informal_theorem}
 Let $N = r(\Sigma+1) < \Pi$, $N' = \min\{r(\Sigma+d), \Pi\}$, and let $\tensor{A} = f(\vect{x})$ be a generic rank-$r$ tensor in $\F^{n_1 \times n_2 \times \cdots \times n_d}$. If $J = [\partial f_i / \partial x_j]_{i,j}$ denotes the Jacobian matrix of $f$ evaluated at $\vect{x}$ and $\varsigma_1 \ge \varsigma_2 \ge \cdots \ge \varsigma_N \ge \varsigma_{N+1} = \cdots = \varsigma_{N'} = 0$ are its singular values, then $\kappa_A = \varsigma_N^{-1}$ is an absolute condition number of the tensor rank decomposition problem $\tensor{A} = f(\vect{x})$.
\end{theorem}

The proof of this theorem is rather technical, and it should become clear to the reader that several of the technical issues could have been overcome if a slightly different definition of a condition number was investigated---in particular, the reader may notice that Part I of the proof of \refthm{thm_condition} contains in it the nuggets of a definition of a (set of) alternative condition numbers. The technical issues are mathematically easy to overcome by reducing the number of parameters in the vectorized factor matrices to the (expected) dimension (as algebraic variety) of the set of rank-$r$ tensors. This amounts to removing the so-called \emph{scaling indeterminacies} that are inherent in the definition of $f$, resulting in a Jacobian matrix with linearly independent columns. The scaling indeterminacies can be removed in several ways, perhaps most straightforwardly by simply removing some of the variables (as in Part I of the proof of \refthm{thm_condition}). However, I stress that this approach leads to a condition number of a \emph{different} problem $\widehat{f}^\dagger$, simply because the function $\widehat{f}$ mapping the (minimal) set of parameters to a tensor is different from $f$. This approach introduces some challenges of its own, mainly related to the fact that each choice of elimination of variables leads to a different tensor rank decomposition problem, and hence to a different condition number. While it leads to very interesting questions such as determining which elimination of variables one should use to minimize the condition number in a given problem instance, this paper provides an answer to Nicaise's question:\footnote{Remark that Nicaise's question explicitly excludes taking the minimum or maximum over all possible modified tensor rank decomposition problems $\widehat{f}^\dagger$.}
\begin{center}\emph{
 Does $f^\dagger$ admit an intrinsic condition number, i.e., one that does not depend (even implicitly) on the choice of affine chart (or, equivalently, on the choice of elimination of variables)?
}\end{center}
This is the question that I believe \refthm{thm_condition} answers. I will not investigate the alternative set-of-condition-numbers approach in this paper.

Another way of looking at \refthm{thm_informal_theorem} is that it provides an answer to the following research question:
\begin{center}\emph{
Does the singular value $\varsigma_N$ of the Jacobian of $f$ admit any kind of geometrical interpretation?
}\end{center}
This Jacobian matrix lies at the heart of gradient-based optimization algorithms for the computation of (approximate) tensor rank decompositions, such as the algorithms in \cite{Paatero1997,Phan2013,Hayashi1982,Tomasi2005,Oseledets2006,Acar2011,Phan2013a,Sorber2013a,VMV2015}. That the inverse of this singular value admits an interpretation as a condition number is somewhat of a free lunch: one needs to construct the Jacobian matrix (or its Gram matrix) in every step of the optimization algorithm anyway, and at the final iteration one obtains an estimate of the conditioning of the computed decomposition with little extra effort. In addition, it was already observed in the chemometrics literature that this singular value seems to be related to the stability of the tensor rank decomposition. According to Tomasi \cite{TomasiPhD}, ``several diagnostic parameters have been used in the chemometrics literature to measure the difficulty of [tensor rank decomposition] problems, but many seem to fall short of giving consistent and concise information.'' Tomasi proceeds by listing the various criteria that have been proposed, among them ``the Jacobian's condition number and singular values'' that were studied in \cite{Paatero1999,Paatero2000,Tomasi2005,Tomasi2006}. This paper proves that the inverse of the least singular value of the Jacobian matrix has a clear, consistent and concise interpretation: \emph{it is a condition number} of the tensor rank decomposition problem.

\subsection{Related work}
Despite the importance of a perturbation theory for the tensor rank decomposition problem, heretofore few attempts have been undertaken to develop such a theory. 

As mentioned before, \cite[p.~31]{TomasiPhD} contains an overview of several criteria that have been proposed in the chemometrics literature for assessing the difficulty of the tensor rank decomposition problem. However, none of the cited references derives a condition number as such. It will be shown in \refthm{thm_condition} that the least singular value of the Jacobian matrix that was already proposed as a heuristic for checking stability is, in fact, a condition number, and, as such, admits the useful interpretation in \refeqn{eqn_forward_error_estimate}.

Robust identifiability of the parameters of a tensor rank decomposition was investigated by Bhaskara, Charikar, and Vijayaraghavan \cite{Bhaskara2014}. It is a necessary technical requirement for the definition of the condition number. In addition, it is shown in \cite{Bhaskara2014} how these robustness results relate to the accuracy with which the parameters of a tensor rank decompositions, i.e., the factor matrices, can be identified.

Perhaps the result closest in spirit to this work are the Cram\'er--Rao bounds (CRB) for the tensor rank decomposition problem that have been investigated in \cite{LS2001,KTP2011,TK2011,TPK2013}. This CRB measures the stability of a tensor rank decomposition in a statistical framework, wherein additive Gaussian noise is assumed to corrupt the factor matrices. The quantity of interest is the squared angular error between the true parameters and the estimated parameters \cite{TPK2013}. The Cram\'er--Rao lower bound (CRLB) for estimating the parameters $\vect{p}$ of the tensor rank decomposition of $\tensor{A} = f(\vect{p})$ is then defined \cite{TPK2013} as the inverse of the Fisher information matrix $\tfrac{1}{\sigma^2} J^T J$, where $J$ is as in \refthm{thm_informal_theorem} and $\sigma^2$ is the variance of the Gaussian noise. However, this matrix is not invertible, so in \cite{TPK2013} it is suggested to alter the tensor rank decomposition problem by removing some of the parameters. This implies that the CRLB depends on the particular elimination of the variables that is chosen; hence, it is not an \emph{intrinsic} measure of stability of the tensor rank decomposition problem as defined in this paper. 

As an intermezzo, \cite{TPK2013} states as an open problem whether ``stability of the [tensor rank decomposition] problem implies its essential uniqueness.'' The answer is an unqualified no. There are explicit counterexamples where the generic tensor is not $r$-identifiable, yet whereby the $N$th singular value of the Jacobian $J$ is nonzero, so that the tensor still has finitely many decompositions---for a conjecturally complete list of such tensors, see \cite[Theorem 1.1]{COV2013}. In these cases, the CRLB will be finite, but the tensor is not identifiable. Counterexamples are not limited to those appearing in \cite[Theorem 1.1]{COV2013}, however they expected to be intrinsically tied up with singular points of generically $r$-identifiable Segre varieties, as Example 4.2 in \cite{COV2013} shows.

\subsection{Outline}
The rest of this paper is structured as follows. In the next section, some necessary preliminaries from (semi-)algebraic geometry are recalled. The main object of study, Terracini's matrix, is introduced. \refsec{sec_fiber} analyzes the fiber $f^\dagger$. It is shown that the fiber is usually unique up to the action of some multiplicative group $\Var{T}$. The concept of robust $r$-identifiability is introduced as a weak technical assumption. Essentially it allows us to assume that a neighborhood of a point on the $r$-secant variety of the Segre variety exists wherein all points are smooth, of maximal dimension, and $r$-identifiable. \refsec{sec_distance} introduces a distance measure that ignores the action of the multiplicative group $\Var{T}$ acting on $f^\dagger$. Terracini's matrix has a nontrivial kernel originating from $\Var{T}$, however \refsec{sec_isl} proves that it may be ignored completely in an analysis of the condition number with respect to the proposed distance measure. This lemma is proved by constructing a kind of approximate Newton-type algorithm that replaces a vector in the kernel of Terracini's matrix with a vector in the orthogonal complement of the kernel of much smaller norm. Thereafter, the fact that the inverse of the least nonzero singular value of Terracini's matrix corresponds with an absolute condition number with respect to the distance measure of \refsec{sec_distance} is proved in \refsec{sec_condition}. \refsec{sec_norm_balanced} continues by defining a particular condition number and presents an algorithm for computing it; the basic properties of this condition number are further analyzed in \refsec{sec_basic_properties}. Numerical experiments illustrating the behavior of the proposed condition number are presented in \refsec{sec_numerical_experiments}. A discussion of the conclusions and future research directions concludes the paper in \refsec{sec_conclusions}.

\subsection*{Acknowledgements}
This paper benefited immeasurably from the fruitful discussions I had with B.~Jeuris, K.~Meerbergen, and G.~Ottaviani. I am particularly indebted to J.~Nicaise for inquiring about an intrinsic definition of the condition number. I thank P.~Breiding for detailed comments on an earlier version of this manuscript.

\section{Preliminaries} \label{sec_basic_results}
Several requisite results concerning the tensor rank decomposition are recalled with a particular focus on some basic concepts from algebraic geometry. While the varieties that we study admit the structure of projective varieties, we will present the results in affine space. The reason is that in the applications where we envisage that a condition number would be of interest, one commonly works in affine space. In addition, the definition of a condition number in this paper is inherently local, so that we may always restrict ourselves to affine varieties. Before proceeding, some notation is fixed.

\subsection*{Notation and conventions}
Varieties are typeset in a calligraphic uppercase font ($\Var{S}, \Var{V}$), tensors in an uppercase fraktur font ($\tensor{A})$, matrices in uppercase letters ($A, U, V$), vectors in boldface lowercase letters ($\mathbf{a}$, $\mathbf{x}$), and scalars, as well as points on varieties, in lowercase letters ($a, b, \lambda, p, q$). The symbol $\F$ denotes either the real field $\mathbb{R}$ or the complex field $\mathbb{C}$. For a matrix $A \in \F^{m \times n}$, $A^T$ is its transpose and $A^H$ is its conjugate transpose.
The Euclidean norm of $\vect{v} \in \F^{m}$ is $\|\vect{v}\| = \sqrt{\vect{v}^H \vect{v}}$.
A block diagonal matrix with diagonal blocks $A_1, \ldots, A_d$ is denoted by $\diag(A_1, \ldots, A_d)$. By $\cspan(A)$ the column span of the matrix $A$ is meant.

The prototypical tensor in this paper is the rank-$r$ tensor $\tensor{A} \in \F^{n_1 \times n_2 \times \cdots \times n_d}$. Hence, the scalar $d$ is used exclusively for the order of $\tensor{A}$, $r$ denotes its rank, the scalars $n_1$, $n_2$, $\ldots$, $n_d$ are reserved for $\tensor{A}$'s dimensions, and we also define 
\[
 \Sigma = \sum_{k=1}^d (n_k - 1), \quad
 N = r(\Sigma + 1), 
 \quad\text{and}\quad
 \Pi = \prod_{k=1}^d n_k,
\]
which are respectively the dimension of the variety of rank-$1$ tensors embedded in $F = \F^{n_1} \otimes \F^{n_2} \otimes \cdots \otimes \F^{n_d}$, the expected dimension of the smallest variety containing the tensors of rank $r$ in $F$, and the dimension of $F$.

\subsection{Segre varieties and secants}
The Segre variety over $\F = \C$ or $\R$, denoted by 
\[
\Var{S}_\F = \operatorname{Seg}(\F^{n_1} \times \F^{n_2} \times \cdots \times \F^{n_d}) \subset \F^{n_1} \otimes \F^{n_2} \otimes \cdots \otimes \F^{n_d} \cong \F^\Pi, 
\]
is the smooth, determinantal variety of rank-$1$ tensors. Its elements can be parameterized explicitly via the Segre map:
\begin{align*}
 \operatorname{Seg} : \F^{n_1} \times \F^{n_2} \times \cdots \times \F^{n_d} &\to \F^{n_1} \otimes \F^{n_2} \otimes \cdots \otimes \F^{n_d} \cong \F^\Pi \\
 (\vect{{a}}^1, \vect{{a}}^2, \ldots, \vect{{a}}^d) &\mapsto \vect{{a}}^1 \otimes \vect{{a}}^2 \otimes \cdots \otimes \vect{{a}}^d
\end{align*}
where $\otimes$ is the tensor product. For notational brevity, I will consider the Segre map that embeds into $\F^\Pi$; hence, in coordinates the tensor product may be realized as the Kronecker product.
Note that the $2$-factor Segre variety is the set of rank-$1$ matrices.
In this paper, it is often required to refer to the specific vectors $b_i = (\sten{a}{i}{1}, \sten{a}{i}{2}, \ldots, \sten{a}{i}{d})$ in the representation of a point $p_i = \operatorname{Seg}(b_i) \in \Var{S}_\F$ on a Segre variety $\Var{S}_\F$. We will call $b_i$ the \emph{representative} of $p_i$.

Let $r < \tfrac{\Pi}{\Sigma+1}$ be strictly \emph{subgeneric}, and define the set of tensors of $\F$-rank at most $r$ as
\[
 \SecZ{r}{\Var{S}_\F} = \{ p_1 + p_2 + \cdots + p_r \;|\; p_i \in \Var{S}_\F \} = \{ \tensor{A} \in \F^{n_1} \otimes \cdots \otimes \F^{n_d} \;|\; \operatorname{rank}_{\F}(\tensor{A}) \le r \}.
\]
Over the complex numbers, $\F = \C$, the closures of $\SecZ{r}{\Var{S}_\C}$ in both the Euclidean and Zariski topology agree \cite[Corollary 5.1.1.5]{Landsberg2012}, so that rank-$r$ tensors form a Euclidean-dense subset of the algebraic variety $\Sec{r}{\Var{S}_\C} = \overline{\SecZ{r}{\Var{S}_\C}}$, where the overline indicates the Zariski-closure. This projective variety is called the \emph{$r$-secant variety} of the Segre variety $\Var{S}_\C$. Over the real numbers, $\F = \R$, we could define the $r$-secant variety analogously by taking the closure in the Zariski topology, resulting in an algebraic variety over $\R$. Unfortunately, $\SecZ{r}{\Var{S}_\R}$ is not Euclidean-dense in $\overline{\SecZ{r}{\Var{S}_\R}}$ in general. This can seen, for example, in the occurrence of multiple \emph{typical real ranks} \cite{Blekherman2013, CO2012, BBO2015, Landsberg2012}, i.e., the fact that there may be several open sets in $\R^{n_1\times\cdots\times n_d}$ each with constant and different $\R$-rank. For this reason, it is more sensible to consider the closure of $\SecZ{r}{\Var{S}_\R}$ in the Euclidean topology, which results in a semi-algebraic set that we denote by $\Sec{r}{\Var{S}_\R}$, see, e.g., \cite[Theorem 6.2]{Silva2008}. Recall that a semi-algebraic set is the solution set of a system of polynomial equations and inequalities over $\R$---see \cite{BCR1998}.

The foregoing showed that for both $\F = \C$ and $\R$ the tensors of $\F$-rank equal to $r$ form a Euclidean-dense subset of $\Sec{r}{\Var{S}_\F}$. Taking the closure of $\SecZ{r}{\Var{S}_\F}$ in the Euclidean topology generally introduces new points, i.e., $\Var{R} = \Sec{r}{\Var{S}_\F} \setminus \SecZ{r}{\Var{S}_\F} \ne \emptyset$.  
This residual set $\Var{R}$ contains tensors of rank strictly larger than $r$, which can be approximated arbitrarily well by tensors of rank $r$. Such tensors are called border tensors \cite{Bini1980,Bini1980a}. 
They cause problems when trying to approximate them by a rank-$r$ tensor, which will be illustrated in \refsec{sec_illposed_example}.

\subsection{Dimension} \label{sec_dimension}
A fundamental property of manifolds, varieties and semi-algebraic sets is that at most points they locally resemble a linear space. 
One way to describe this local linear approximation to a manifold, variety or semi-algebraic set $\Var{V} \in \F^{\Pi}$ at a point $p \in \Var{V}$ is the \emph{Zariski tangent space}; see respectively \cite[p.~54]{Lee2013}, \cite[pp.~175--176]{GH1978} and \cite[Section 3.3]{BCR1998}. This vector space is defined as the span of all vectors in $\F^\Pi$ that are tangent to $\Var{V}$:
\[
 \Tang{p}{\Var{V}} = \cspan\bigl( \{ \tfrac{d}{dt}p (0)  \;|\; p(t) \subset \Var{V},\; t \in (0,1], p(0) = p \} \bigr).
\]
The dimension of a manifold or variety $\Var{V} \subset \F^{\Pi}$ may be defined geometrically as the number $n$ such that $\dim \Tang{p}{\Var{V}} = n$ for all $p$ in an open neighborhood $\Var{N}$ of $p_0 \in \Var{V}$. For a semi-algebraic set $\Var{V}$ the dimension is defined as the dimension of the Zariski-closure $\overline{\Var{V}}$ \cite[Proposition 2.8.2]{BCR1998}. 
The dimension of the Segre variety $\Var{S}_\F = \operatorname{Seg}(\F^{n_1} \times \cdots \times \F^{n_d}) \subset \F^{\Pi}$ is well-known \cite{Harris1992,Landsberg2012}:
\[
 \dim \Var{S}_\F = \Sigma + 1 = 1 + \sum_{k=1}^d (n_k - 1).
\]
The dimension of $\Sec{r}{\Var{S}_\F}$ is, in general, not known. Since $\SecZ{r}{\Var{S}_\F}$ is Euclidean-dense in $\Sec{r}{\Var{S}_\F}$, an upper bound is readily established by noting that $\Sec{r}{\Var{S}_\F}$ is the projection of 
\[
 \overline{ \bigl\{ \bigl((p_1,p_2,\ldots,p_r),p\bigr) \;|\; p=p_1+p_2+\ldots+p_r,\; p_1, \ldots, p_r \in \Var{S}_\F \} } \subset (\Var{S}_\F \times \cdots \times \Var{S}_\F) \times \F^{\Pi}
\]
onto the second factor, where the overline indicates the closure in the Euclidean topology.
Hence,
\(
 \dim \Sec{r}{\Var{S}_\F} \le \min\{ \Pi , r(\Sigma+1) \},
\)
by \cite[Proposition 2.8.6]{BCR1998} for $\F=\R$ and \cite[Theorem 11.12]{Harris1992} for $\F=\C$.
One expects that equality holds; if this is the case, then we say that the $r$-secant is \emph{nondefective}. If all secant varieties of a Segre variety $\Var{S}_\F$ are nondefective, then we say that this Segre variety is nondefective. 
The problem of determining the dimension of $\Sec{r}{\Var{S}_\F}$ has seen much progress during the last two decades, both theoretically \cite{AOP2009,BCO2013,Bocci2011,CGG2002,CGG2005c,CGG2007,CGG2008,CGG2011,ChCi2001,Chiantini2006,Chiantini2012,Gesmundo2012} and numerically \cite{Comon2009,COV2013,VVM2014}, yet the following is still an open problem: 

\begin{conjecture}[Abo, Ottaviani, and Peterson \cite{AOP2009}] \label{conj_aop}
Let $\Var{S}_\F = \operatorname{Seg}(\F^{n_1} \times \cdots \times \F^{n_d}) \subset \F^{\Pi}$, be a Segre variety with $d \ge 3$ and $n_1 \ge n_2 \ge \cdots \ge n_d \ge 2$. Then, \(\Var{S}_\F\) is nondefective, unless 
\begin{enumerate}
\item $(n_1, n_2, n_3) = (4,4,3)$, 
\item $(n_1, n_2, n_3) = (2k+1,2k+1,3)$ with $k \in \mathbb{N}$,
\item $(n_1, n_2, n_3, n_4) = (k+1, k+1, 2, 2)$ with $k \in \mathbb{N}$, or
\item $n_1 > 1 + \prod_{k=2}^d n_k - \sum_{k=2}^d (n_k-1)$, i.e., $\Var{S}_\F$ is unbalanced.
\end{enumerate}
\end{conjecture}

\begin{remark}
It is noted that \cite{AOP2009} only states the conjecture in the case of $\F=\C$ but the geometrical approach in \cite{VVM2014,COV2013} applies to $\R$-varieties as well, as it consisted of constructing a subset of the Zariski tangent space at a randomly chosen point of $\Sec{r}{\Var{S}_\mathbb{Z}}$ and $\Sec{r}{\Var{S}_\R}$ in \cite{COV2013} and \cite{VVM2014} respectively. By \cite[Proposition 2.8.2]{BCR1998} the dimension of the $\R$-variety $\overline{\Sec{r}{\Var{S}_\R}}$, where the overline denotes the closure in the Zariski topology, coincides with $\dim \Sec{r}{\Var{S}_\R}$.
\end{remark}

\subsection{Smooth and regular points}
A difficulty in working with varieties---as contrasted to smooth manifolds---is that the local dimension of a complex variety $\Var{V}$ is not upper semicontinuous: there may exist points $p \in \Var{V}$ where $\dim \Tang{p}{\Var{V}} > \dim \Var{V}$. A point $p \in \Var{V}$ where $\dim \Tang{p}{\Var{V}} = \dim \Var{V}$ is called a \emph{smooth point} of $\Var{V}$; otherwise, it is called a \emph{singular point}. The set of singular points is called its \emph{singular locus}, which is a subvariety of dimension strictly less than $n$. Similarly, the set of smooth points of $\Var{V}$ forms an $n$-dimensional submanifold of $\C^\Pi$; it is Zariski-dense in $\Var{V}$. For consistency with the real case, which is treated next, I will call a smooth point of a $\C$-variety a \emph{regular point}.

For semi-algebraic sets $\Var{V} \subset \R^\Pi$, defining the smooth locus is more subtle, because $\Var{V}$ may consist of several semi-algebraically connected components of different dimensions. Attempting to define the set of smooth points geometrically as the set of points where $\dim \Tang{p}{\Var{V}} = \dim \Var{V}$ may in general include points that are singular points in the proper, algebraic definition \cite[Definition 3.3.9]{BCR1998}. I prefer an alternative geometrical definition that suffices for the purpose of this paper. Every semi-algebraic set $\Var{V}$ admits a \emph{Nash stratification}, meaning that it can be decomposed as a finite union of disjoint semi-algebraically connected Nash manifolds \cite[Section 9.1]{BCR1998}: $\Var{V} = \bigcup_{i} \Var{M}_i$, where $\Var{M}_i$ is a smooth $\R$-manifold of dimension $d_i = \dim \Var{M}_i$ and $\Var{M}_i \cap \Var{M}_j = \emptyset$ if $i\ne j$. Let $D = \{ i \;|\; \dim \Var{M}_i = \dim \Var{V} \}$ be the set of all $\Var{M}_i$ of maximal dimension; it is nonempty. I will call $\bigcup_{i\in D} \Var{M}_i$ the set of \emph{regular} points of the semi-algebraic set $\Var{V}$. All regular points are smooth but in general the converse is false. The regular locus is neither Zariski-dense nor Euclidean-dense, in general.\footnote{With some additional effort the smooth locus could be defined accurately so that it would be a Euclidean-dense subset of $\Var{V}$ but for the purpose of this paper the regular locus suffices since the condition number in Theorems \ref{thm_informal_theorem} and \ref{thm_condition} will be $\infty$ for non-regular points.}

\subsection{Terracini's Jacobian} \label{sec_terracini}
The main object of study in \refthm{thm_informal_theorem} is the set of singular values of the Jacobian matrix of partial derivatives of $f$ evaluated at $\vect{p}$. This Jacobian matrix is well-known \cite{AOP2009,VVM2014} and its structure has been analyzed in the context of optimization algorithms for computing tensor rank decompositions \cite{Acar2011,Phan2013a,Sorber2013a,VMV2015}.
Assume that we are given a set of representatives $b_i = (\sten{a}{i}{1}, \ldots, \sten{a}{i}{d})$ of the points $p_i = \operatorname{Seg}(b_i) \in \Var{S}_\F$, and consider the tensor rank decomposition
\[
 \tensor{A} = f(\vect{p}) = \sum_{i=1}^r p_i = \sum_{i=1}^r \sten{a}{i}{1} \otimes \sten{a}{i}{2} \otimes \cdots \otimes \sten{a}{i}{d},
\]
where $\vect{p} = \operatorname{vecr}(b_1,b_2,\ldots,b_r)$. Then, the Jacobian matrix is given explicitly by
\begin{align}\label{eqn_terracini}
T_\vect{p} = \begin{bmatrix} T_{1} & T_{2} & \cdots & T_{r} \end{bmatrix}, \text{ where } 
T_{i} = \begin{bmatrix} I_{n_1} \otimes \sten{a}{i}{2} \otimes \cdots \otimes \sten{a}{i}{d} & \cdots & \sten{a}{i}{1} \otimes \cdots \otimes \sten{a}{i}{d-1} \otimes I_{n_d}
\end{bmatrix}.
\end{align}
I will refer to $T_\vect{p}$ as \emph{Terracini's matrix} associated with the vectorized factor matrices $\vect{p}$ because it describes that part of the tangent space to $\Sec{r}{\Var{S}_\F}$ that is obtained from an application of Terracini's Lemma \cite{Terracini1911,Landsberg2012}. At regular points of $\Sec{r}{\Var{S}_\F}$, the column span of Terracini's matrix coincides with the Zariski tangent space. It then follows from \refconj{conj_aop} that $T_\vect{p}$ is expected to be of maximal rank $N = r(\Sigma+1)$ for subgeneric $r \le \tfrac{\Pi}{\Sigma+1}$, modulo the known exceptions presented in \refconj{conj_aop}. For $\F = \C$ this implies that Terracini's matrix is of rank $N$ on a dense subset of $\F^{r(\Sigma+d)}$, while for $\F = \R$ we can only conclude that its rank is $N$ on an open subset of $\F^{r(\Sigma+d)}$.

As the number of columns of $T_\vect{p}$ is $r(\Sigma + d) > N$, there is a kernel of dimension at least $r(d-1)$. If $\Sec{r}{\Var{S}_\F}$ has the expected dimension, then the kernel can be described explicitly as follows. A basis of the kernel of $T_{p_i}$ is
\begin{subequations}  \label{eqn_terracini_kernel}
\begin{align}
 K_i = \left\{ 
\begin{bmatrix} \sten{a}{i}{1} \\ -\sten{a}{i}{2} \\ 0 \\ 0 \\ \vdots \\ 0 \end{bmatrix},  
\begin{bmatrix} \sten{a}{i}{1} \\ 0 \\ -\sten{a}{i}{3} \\ 0 \\ \vdots \\ 0 \end{bmatrix}, \ldots,
\begin{bmatrix} \sten{a}{i}{1} \\ 0 \\ 0 \\ \vdots \\ 0 \\ -\sten{a}{i}{d} \end{bmatrix}
\right\} = \{ \sten{k}{i}{2}, \sten{k}{i}{3}, \ldots, \sten{k}{i}{d} \},
\end{align}
which contains exactly $d-1$ linearly independent basis vectors. It is straightforward to check that $T_{i} K_i = 0$.
It is an exercise to verify that 
\begin{align}
 K = \{ \sten{k}{i}{2} \otimes \vect{e}_i, \sten{k}{i}{3} \otimes \vect{e}_i, \ldots, \sten{k}{i}{d} \otimes \vect{e}_i \}_{i=1}^r,
\end{align}
\end{subequations}
where $\vect{e}_i$ is the $i$th standard basis vector of $\F^{r}$, forms a basis of the kernel of $T_\vect{p}$. 

\section{The solution fiber and identifiability} \label{sec_fiber}
Since $f$ is never an injective function, the fiber $f^\dagger(\tensor{A})$ of $f$ at a rank-$r$ tensor $\tensor{A}$ contains several elements. Two trivial sources of this multivaluedness are well-known \cite{Kruskal1977}. The \emph{permutation indeterminacy} arises because the order in which the rank-$1$ terms appear in \refeqn{eqn_cpd} is not uniquely determined by $\tensor{A}$; if $\vect{p} = \operatorname{vecr}(b_1,b_2,\ldots,b_r) \in f^\dagger(\tensor{A})$ is a decomposition of $\tensor{A}$, then so is $\vect{p}_\pi = \operatorname{vecr}(b_{\pi_1},b_{\pi_2},\ldots,b_{\pi_r}) \in f^\dagger(\tensor{A})$ for every permutation $\pi$ of $\{1,2,\ldots,r\}$. If a rank-$1$ tensor is represented in the standard way by a $d$-tuple of vectors $(\sten{a}{}{1}, \ldots, \sten{a}{}{d}) \in \F^{n_1} \times \cdots \times \F^{n_d}$, then this representation suffers from a \emph{scaling indeterminacy} in the sense that $(\alpha_1 \sten{a}{ }{1}, \ldots, \alpha_d \sten{a}{ }{d})$ with $\alpha_1 \cdots \alpha_d = 1$ represents the same rank-$1$ tensor. 
Two rank decompositions are considered to be \emph{essentially equal} if the rank-$1$ tensors appearing in their respective decompositions are equal.\footnote{If one prefers to think projectively: two rank decompositions are essentially equal if the points on the Segre variety in projective space are equal. Note that we can ignore multiplicities of the points because a decomposition in which the same rank-$1$ tensor appears multiple times is not of rank $r$, contradicting the assumption on $\tensor{A}$'s rank.} 

Let us be more concrete about the aforementioned relation. Define\footnote{The letters are mnemonics for Diagonal scaling, Block diagonal scaling, the Permutation indeterminacies, and the Trivial indeterminacies.}
\begin{align*}
\Var{D} 
&= \bigl\{ \diag\bigl( (\alpha_2\cdots\alpha_d)^{-1} I_{n_1}, \alpha_2 I_{n_2}, \ldots, \alpha_d I_{n_d} \bigr) \;\big|\; \alpha_2, \ldots, \alpha_d \in \F\setminus\{0\} \bigr\}, \\
\Var{B}
&= \bigl\{ \diag(D_1,D_2,\ldots,D_r) \;|\; D_1, D_2, \ldots, D_r \in \Var{D} \bigr\}, \\
\Var{P} 
&= \bigl\{ P \otimes I_{\Sigma+d} \;|\; P \text{ is an } r \times r \text{ permutation matrix} \bigr\}, \text{ and} \\
\Var{T}
&= \Var{P}\Var{B} = \bigl\{ P B \;|\; P \in \Var{P} \text{ and } B \in \Var{B} \bigr\}.
\end{align*}
It is not difficult to verify that for every $P\otimes I_{\Sigma+d} \in \Var{P}$ and diagonal matrices $D_i \in \F^{\Sigma+d \times \Sigma+d}$ one has
\begin{align} \label{eqn_diagonal_congruence}
(P\otimes I_{\Sigma+d}) \cdot \diag(D_1, D_2, \ldots, D_r) \cdot (P \otimes I_{\Sigma+d})^T = \diag( D_{\pi_1}, D_{\pi_2}, \ldots, D_{\pi_r} ),
\end{align}
where $\pi$ is the permutation represented by $P$. It is straightforward to check that both $\Var{B}$ and $\Var{P}$ are multiplicative groups. From thence one can verify that $\Var{T}$ satisfies the properties of a group, however I present here a more elegant alternative that was suggested to me by P.~Breiding: \refeqn{eqn_diagonal_congruence} defines a map $\psi : \Var{P} \to \mathrm{Aut}(\Var{B})$ that takes $x \in \Var{P}$ to $\psi_x \in \mathrm{Aut}(\Var{B})$ defined by $\psi_x : y \mapsto x y x^{-1}$. Hence, it follows immediately (see, e.g., \cite[Chapter 7]{Rotman}) that $\Var{T}$ is the semidirect product $\Var{T} = \Var{P} \rtimes \Var{B}$, which is a multiplicative group.
Writing $\vect{p} \sim \vect{q}$ if and only if $\vect{p} = T \vect{q}$ for some $T \in \Var{T}$,\footnote{That is, $\vect{q}$ is in the orbit of $\vect{p}$.} it is now easy to verify that $\sim$ is an equivalence relationship and that $f$ is a morphism for $\sim$. If $\vect{p} \sim \vect{q}$ then $\vect{p}$ and $\vect{q}$ represent tensor rank decompositions that are essentially equal, as we defined above. Naturally, if $\vect{p} \sim \vect{q}$ then $f(\vect{p}) = f(\vect{q})$, however \emph{the converse is not a priori true}. If the converse is also true, i.e., $\tensor{A} = f(\vect{p}) = f(\vect{q})$ implies that $\vect{p} \sim \vect{q}$, then we say that $\tensor{A}$ is \emph{$r$-identifiable}. In other words, if
\[
 | f^\dagger(\tensor{A}) / \sim | = 1,
\]
then $\tensor{A}$ is of rank $r$ and has one essentially unique tensor rank decomposition.

In the case of the tensor rank decomposition, there is very extensive evidence that most tensors of small rank are identifiable, substantially simplifying the interpretation of the condition number that will be explored in this paper. As is well-known, Kruskal \cite{Kruskal1977} proposed an explicit criterion for testing whether $\tensor{A} = f(\vect{x}) \in \SecZ{r}{\Var{S}_\F}$ is $r$-identifiable, assuming that $\vect{x}$ is provided as input to the test. Several other such criteria have been proposed since---e.g., \cite{SB2000,JS2004,Domanov2013b,COV2013}. More powerful theorems and conjectures are known in the setting of \emph{generic $r$-identifiability}. A Segre variety $\Var{S}_\F$ is called generically $r$-identifiable if there exists a Zariski-closed set $Z \supset \Sec{r}{\Var{S}_\F}\setminus\SecZ{r}{\Var{S}_\F}$ such that every $\tensor{A} \in \Sec{r}{\Var{S}_\F} \setminus Z$ is $r$-identifiable. 
Several papers have proved generic $r$-identifiability of the secant varieties $\Sec{r}{\Var{S}_\F}$, such as \cite{S1983,Chiantini2012,BCO2013,DdL2015,HOOS2015}. 
In the case of the algebraically closed field $\F = \C$, the results are much more developed than in the real case;
the following conjecture paints a nearly complete picture:

\begin{conjecture}[Chiantini, Ottaviani, and Vannieuwenhoven \cite{COV2013}]\label{conj_cov}
 Let $\Var{S}_\C = \Seg{\C^{n_1}\times\C^{n_2}\times\cdots\times\C^{n_d}} \subset \C^{\Pi}$ be a Segre variety with $d \ge 3$ and $n_1 \ge n_2 \ge \cdots \ge n_d \ge 2$. Then, $\Var{S}_\C$ is generically $r$-identifiable if $r < \tfrac{\Pi}{\Sigma+1}$, unless $(n_1, n_2, \ldots, n_d)$ is
 \begin{enumerate}
  \item $(4,4,3)$, 
  \item $(4,4,4)$, 
  \item $(6,6,3)$, 
  \item $(n,n,2,2)$ with $n \in \mathbb{N}$, 
  \item $(2,2,2,2,2)$, or 
  \item $n_1 > \prod_{k=2}^d n_k - \sum_{k=2}^d (n_k - 1)$.
 \end{enumerate}
\end{conjecture}

A computer-assisted proof \cite{COV2013} shows that the conjecture is true for all $n_1 \cdots n_d \le 15000$. The conjecture is nearly optimal in the sense that a generic $\vect{p} \in \C^{r(\Sigma+d)}$ with $r > \tfrac{\Pi}{\Sigma+1}$ has a solution fiber at $\tensor{A} = f(\vect{p})$ of positive dimension, i.e., $\dim f^\dagger(\tensor{A}) \ge r(\Sigma+1) - \Pi > 0$, so that it is not identifiable. For ``perfect'' spaces where $r = \tfrac{\Pi}{\Sigma+1}$ and for which $\Sec{r}{\Var{S}_\C}$ is nondefective, a generic $\tensor{A} \in \Sec{r}{\Var{S}_\C} = \C^\Pi$ has $|f^\dagger(\tensor{A})/\sim| = k$, where $k$ is expected to be strictly larger than $1$ according to \cite{HOOS2015}.

Closely related with generic $r$-identifiability is the {a priori} weaker notion of \emph{robust $r$-identifiability}, which was recently explored in \cite{Bhaskara2014}. 

\begin{definition}[Robust $r$-identifiability]
A rank-$r$ tensor $\tensor{A} \in \SecZ{r}{\Var{S}_\F}$ is said to be robustly $r$-identifiable if it is a regular point of $\SecZ{r}{\Var{S}_\F}$ and there exists an open neighborhood $\Var{N}$ of $\tensor{A}$ such that 
\[
 \forall \tensor{B} \in \Var{N} \subset \SecZ{r}{\Var{S}_\F} :\: |f^\dagger(\tensor{B})/\sim| = 1.
\]
\end{definition}
Notice that the assumption on the regularity of $\tensor{A}$ ensures that Terracini's matrix coincides with the Zariski tangent space and that (an open semi-algebraic subset of) the open neighborhood $\Var{N}$ is Nash diffeomorphic with a semi-algebraic subset of $\F^{r(\Sigma+1)}$ \cite[Section 2.9]{BCR1998}.
\emph{Robust $r$-identifiability is a most convenient technical assumption for defining a condition number of the tensor rank decomposition problem. In fact, if a tensor is not robustly $r$-identifiable, then I will define the condition number to be $\infty$.} 
Fortunately, the technical assumption on robust $r$-identifiability is very mild because of the following result.

\begin{lemma}\label{lem_robust_idf}
Let $\Var{S}_\F$ be a generically $r$-identifiable Segre variety, and let $Z$ be the Zariski-closed locus where $r$-identifiability fails. Then, every $\tensor{A} \in \SecZ{r}{\Var{S}_\F}\setminus Z$ is robustly $r$-identifiable. 
\end{lemma}
\begin{proof}
Since $\Sec{r}{\Var{S}_\F}$ lives in a Euclidean space there exists a minimum distance $\delta > 0$ from a fixed point $\tensor{A} \in \SecZ{r}{\Var{S}_\F}\setminus Z$ to the Euclidean-closed set $Z$. Taking the Euclidean-open ball $\Var{B}_\delta$ of radius $\delta$ around $\tensor{A}$ by definition does not intersect $Z$. Hence, $\Var{B}_\delta \cap \SecZ{r}{\Var{S}_\F} \subset \SecZ{r}{\Var{S}_\F}\setminus Z$ contains only $r$-identifiable points and the radius $\delta > 0$ can be chosen so that it contains only one connected component, concluding the proof.
\end{proof}

One may wonder why I choose to introduce this technical restriction for defining a condition number. After all, condition numbers could also be defined for set-valued maps $g$ by defining a reasonable premetric for measuring distances between sets. 
I believe there are some reasons that justify the approach developed in this paper.
The first reason is very pragmatic and common in (semi)algebraic geometry: the tensors that are not robustly $r$-identifiable are expected to be contained in a semi-algebraic set of codimension at least $1$ in $\Sec{r}{\Var{S}_\F}$ by \refconj{conj_cov} and \reflem{lem_robust_idf}, so that the definition is valid generically. In other words, if $\Var{S}_\F$ is a generically $r$-identifiable Segre variety and one samples a random rank-$r$ decomposition $\tensor{A} \in \SecZ{r}{\Var{S}_\F}$ from any probability distribution whose support is not contained in the Zariski-closed locus $Z$ where $r$-identifiability fails, then $\tensor{A}$ is both robustly $r$-identifiable and its condition number is well-defined with probability $1$. 
Another reason is that the condition number admits a completely straightforward interpretation for robustly $r$-identifiable tensors: it asymptotically measures the ratio of the forward error and the backward error as in \refeqn{eqn_forward_error_estimate} with respect to some premetric that will be defined in the next section. It admits this straightforward interpretation because $f^\dagger(\tensor{A})/\sim$ locally behaves like an invertible function when $\tensor{A}$ is robustly $r$-identifiable. A third motivation concerns the technique that I employed for defining the condition number in \refthm{thm_condition}: it is inherently local. Hence, if $\tensor{A}$ is robustly $r$-identifiable, a local analysis of one of the solutions $\vect{x} \in f^\dagger(\tensor{A})$ suffices for understanding the global behavior of all solutions $f^\dagger(\tensor{A})$ in a neighborhood of $\tensor{A}$. Finally, since many $\Sec{r}{\Var{S}_\F}$ are generically $r$-identifiable, it is sensible to redefine the tensor rank decomposition problem as the problem of finding the \emph{unique} element in $f^\dagger(\tensor{A})/\sim$ for a given $\tensor{A} \in \SecZ{r}{\Var{S}_\F}$. Then, the set of tensors that admit multiple tensor rank decompositions are defined to be the \emph{ill-posed} inputs. I should mention that this leads to a very natural alternative definition of a condition number, namely the minimum distance to an ill-posed input in the spirit of, e.g., \cite{Demmel1987,BC2013}. However, I will not pursue this interpretation of the proposed condition number in the present paper.

\section{A suitable distance measure}\label{sec_distance}
In interpreting the $N$th largest singular value of Terracini's matrix as a condition number, as is stated in \refthm{thm_informal_theorem}, understanding the influence of the part of the kernel given by \refeqn{eqn_terracini_kernel} is crucial. Consider the following concrete situation. Let $b = (\vect{a}, \vect{b}, \vect{c}) \in \F^{n_1} \times \F^{n_2} \times \F^{n_3}$ be a representative of the rank-$1$ tensor $f(\vect{p}) = \vect{a}\otimes \vect{b}\otimes \vect{c}$, where $\vect{p} = \vecc{b}$. Suppose that the vectorized factor matrices $\vect{p}$ are perturbed by the vector $\Delta_\epsilon = \begin{bmatrix} \epsilon \vect{a}^T & -\epsilon \vect{b}^T & 0\end{bmatrix}^T$, which is contained in the span of $K$, where $K$ is the specific basis in \refeqn{eqn_terracini_kernel}. Then, 
\[
 \vect{p} = 
\begin{bmatrix}
\vect{a} \\ 
\vect{b} \\ 
\vect{c}             
\end{bmatrix}
\quad\text{and}\quad
\vect{p}_\epsilon = \vect{p} + \Delta_\epsilon =
\begin{bmatrix}
(1+\epsilon) \vect{a} \\
(1-\epsilon) \vect{b} \\
\vect{c}
\end{bmatrix}.
\]
For simplicity, let us assume that $\|\vect{a}\|=\|\vect{b}\|=\|\vect{c}\|=1$. Then,
\[
 \|\vect{p} - \vect{p}_\epsilon\|=\sqrt{2} \epsilon 
\quad\text{while}\quad
\| f(\vect{p}) - f(\vect{p}_\epsilon)\| = \|\vect{a}\otimes\vect{b}\otimes\vect{c} - (1+\epsilon)(1-\epsilon)\vect{a}\otimes\vect{b}\otimes\vect{c}\| = \epsilon^2.
\]
This implies that if we would consider the condition number \refeqn{eqn_general_condition_number} with respect to the usual Euclidean norm, then we find that this particular perturbation bounds the relative condition number of the rank decomposition problem at $\vect{p}$ from below by
\[
 \kappa_{2,2} \ge \lim_{\epsilon\to0} \frac{ \| \vect{p} - \vect{p}_\epsilon  \| \cdot \| f(\vect{p}) \| }{ \|f(\vect{p}) - f(\vect{p}_\epsilon) \| \cdot \|\vect{p}\| } = \lim_{\epsilon\to0} \frac{\sqrt{2}}{\sqrt{3}} \epsilon^{-1} \to \infty.
\]
That is, a perturbation in the direction of the kernel results in an unbounded condition number, as was to be expected because the kernel originates from the group action of $\Var{T}$. 
Perturbations $\Delta_\epsilon \in \operatorname{range}(K)$ of sufficiently small norm thus behave like a rescaling. Since rank-1 tensors are determined uniquely up to the action of $\Var{T}$, it seems a natural desire to ignore this action. Acting on this idea, suppose that we choose a different representative of $f(\vect{p}_\epsilon)$ by exploiting the additional freedom offered by $\Var{T}$. Consider the following representative
\[
 \vect{p}_\epsilon' = 
\diag(1-\epsilon, (1-\epsilon)^{-1}, 1) \vect{p}_\epsilon = 
\begin{bmatrix}
(1 - \epsilon^2) \vect{a} \\
\vect{b} \\
\vect{c}
\end{bmatrix} \in \Var{T}\vect{p}_\epsilon.
\]
Now $\|f(\vect{p}) - f(\vect{p}_\epsilon')\| = \epsilon^2$, while $\|\vect{p} - \vect{p}_\epsilon'\| = \epsilon^2$, implying that the condition number should be
\[
 \kappa \ge \lim_{\epsilon\to0} \frac{\epsilon^2}{\epsilon^2 \sqrt{3}} = \frac{1}{\sqrt{3}}.
\]
It will be shown in \refprop{prop_rank1} that the relative condition number with respect to the distance measure that will be introduced shortly is, in fact, equal to $\frac{1}{\sqrt{3}}$.

The foregoing example essentially illustrates that we cannot meaningfully define the condition number of the tensor rank decomposition \emph{with respect to the standard Euclidean distance}. With this metric, the condition number is always unbounded. However, the example also suggests that exploiting the action of $\Var{T}$ is crucial for obtaining a bounded condition number. I propose employing the following measure of distance for defining a condition number of the tensor rank decomposition.
\begin{definition}\label{def_distance_measure}
 Let $b_i = (\sten{a}{i}{1}, \sten{a}{i}{2}, \ldots, \sten{a}{i}{d})$ and $b_i' = (\sten{c}{i}{1}, \sten{c}{i}{2}, \ldots, \sten{c}{i}{d})$ be two set of representatives of the rank-$1$ tensors $p_i = \operatorname{Seg}(b_i)$ and $q_i = \operatorname{Seg}(b_i')$ respectively for $i=1,\ldots,r$.
Let $\vect{p} = \operatorname{vecr}(b_1,b_2,\ldots,b_r)$ and $\vect{q} = \operatorname{vecr}(b_1',b_2',\ldots,b_r')$.
The distance $d(\vect{p}, \vect{q})$ between the vectorized factor matrices $\vect{p}$ and $\vect{q}$ is defined as the least Euclidean distance between $\vect{p}$ and the $\Var{T}$-orbit of $\vect{q}$:
\begin{align*}
d(\vect{p}, \vect{q}) = \inf_{T \in \Var{T}} \| \vect{p} - T \vect{q} \|.
\end{align*}
\end{definition}

\begin{proposition}\label{prop_distance_measure}
Let $b_i$, $b_i'$, $\vect{p}$ and $\vect{q}$ be as in \refdef{def_distance_measure}. Then, the infimum is always attained:
\begin{align*}
d(\vect{p}, \vect{q}) = \min_{T \in \Var{T}} \| \vect{p} - T \vect{q} \| \ge 0. 
\end{align*}
Moreover, the distance $d(\vect{p}, \vect{q}) = d(\vect{q}, \vect{p}) = 0$ if and only if $\vect{p} \sim \vect{q}$. 
\end{proposition}
\begin{proof}
First note that $g(T) = \| \vect{p} - T \vect{q} \|$ is a coercive function in the parameters defining $T \in\Var{T}$, so that for a fixed input the minimizer of the optimization problem is always attained. Indeed, since $I \in \Var{T}$, it follows that the optimal $(P \otimes I_{\Sigma+d}) B \in \Var{T}$ satisfies
\begin{align*}
 \|\vect{p} - (P \otimes I_{\Sigma+d}) B \vect{q}\|^2 
 &= \sum_{i=1}^r \| \vecc{b_i} - D_{\pi_i} \vecc{b_{\pi_i}'} \|^2 \\
 &= \sum_{i=1}^r \sum_{k=2}^{d} \| \sten{a}{i}{k} - \theta_{k,\pi_i} \sten{c}{\pi_i}{k} \|^2 + \sum_{i=1}^r \| \sten{a}{i}{1} - (\theta_{2,\pi_i}\cdots\theta_{d,\pi_i})^{-1} \sten{c}{\pi_i}{1} \|^2 \\
 &\le \| \vect{p} - \vect{q}\|^2 < \infty
\end{align*}
where $B = \diag(D_1,\ldots,D_r)$, $D_i = \diag(\theta_{2,i}^{-1}\cdots\theta_{d,i}^{-1} I_{n_1}, \theta_{2,i} I_{n_2},\ldots,\theta_{d,i} I_{n_{d}})$ and $\pi$ is the permutation represented by the permutation matrix $P \in \F^{r \times r}$. Since we have a sum of positive reals, it follows that all $|\theta_{k,i}|$ are uniformly bounded from above by some positive constant $C$, and, hence there is also a uniform lower bound $0 < c$. So we may optimize $\theta_{k,i}$ over the bounded intervals \( [-C, -c] \cup [c, C] \) leading to the conclusion that the infimum is attained.

The claim about $d(\vect{p},\vect{q}) = 0$ is trivial.
\end{proof}

Note in the foregoing proposition that if $\vect{q}' \in \Var{T} \vect{q}$ then $d(\vect{p}, \vect{q}') = d(\vect{p}, \vect{q})$. In particular, if $\tensor{A}$ is an $r$-identifiable tensor, then the slightly abused notation $d(\vect{p}, f^\dagger(\tensor{A}))$ is well-defined.

The distance measure in \refdef{def_distance_measure} may seem artificial, however it naturally measures the error between a given set of vectorized factor matrices $\vect{p} \in \F^{r(\Sigma+d)}$ and a solution of the rank decomposition problem---whose solution \emph{is} only determined up to the action of $\Var{T}$. 
The \texttt{cpderr} function in the popular Tensorlab \cite{Tensorlab} software package for Matlab, for instance, computes the error in a similar way. Let $p_i = (\sten{a}{i}{1}, \sten{a}{i}{2}, \ldots, \sten{a}{i}{d})$ and $q_i = (\sten{b}{i}{1}, \sten{b}{i}{2}, \ldots, \sten{b}{i}{d})$, $i=1,\ldots,r$, be two set of representatives, and then let $A^k = [\sten{a}{j}{k}]_{j=1}^r$ and $B^k = [\sten{b}{j}{k}]_{j=1}^r$ be the corresponding set of factor matrices, where $k=1,\ldots,d$. Then, Tensorlab measures the error between two sets of factor matrices $(A^1,\ldots,A^d)$ and $(B^1,\ldots,B^d)$ as follows:
it returns a $d$-tuple $(\xi_1, \xi_2, \ldots, \xi_d)$ in which $\xi_i = \|A^i - B^i D_i P\|$, where $D_i$ is a diagonal matrix and $P$ is a permutation matrix. Tensorlab does not enforce $D_1 D_2 \dots D_d = I_r$, however if this additional constraint is imposed then a natural measures of the overall error, treating each of the factors equally, is $\sqrt{ \xi_1^2 + \xi_2^2 + \cdots + \xi_d^2 } = d(\vect{a},\vect{b})$, where $\vect{a}=\operatorname{vecr}(p_1,p_2,\ldots,p_r)$ and $\vect{b}=\operatorname{vecr}(q_1,q_2,\ldots,q_r)$.

Before concluding this section, it is natural to wonder whether we could eliminate the scaling indeterminacies on \emph{both} representatives in \refdef{def_distance_measure} by considering
\[
 \widehat{d}(\vect{p}, \vect{q}) = \inf_{T_1, T_2 \in \Var{T}} \| T_1 \vect{p} - T_2 \vect{q} \|,
\]
which is essentially a premetric on the quotient space $\F^{r(\Sigma+1)} / \Var{T}$ and has some additional appeal because at least it is symmetric, i.e., $\widehat{d}(\vect{p}, \vect{q}) = \widehat{d}(\vect{q}, \vect{p})$. However, I believe that this particular choice will not lead to a sensible definition of a condition number because the Hausdorff separation property is lost, i.e., $\widehat{d}(\vect{p},\vect{q}) = 0$ even when $\vect{p} \not\sim \vect{q}$ and $f(\vect{p}) \ne f(\vect{q})$. The reason is that the set over which we optimize is not closed, and this time the two matrices can counteract each other. The simplest example of this phenomenon are the following representatives of two distinct rank-$1$ tensors 
\[
 \vect{p} = \begin{bmatrix} \vect{a}_1 \\ \vect{a}_2 \\ \vect{c} \end{bmatrix}
 \quad\text{and}\quad
 \vect{q} = \begin{bmatrix} \vect{b}_1 \\ \vect{b}_2 \\ \vect{c} \end{bmatrix}.
\]
Choosing $T_1 = T_2 = \diag(\theta,\theta,\theta^{-2})$, we find
\[
 \lim_{\theta \to 0} \left\|\begin{bmatrix} \theta \vect{a}_1 - \theta \vect{b}_1 \\ \theta \vect{a}_2 - \theta \vect{b}_2 \\ \theta^{-2} \vect{c} - \theta^{-2} \vect{c} \end{bmatrix}\right\| = \lim_{\theta\to0} \left\|\begin{bmatrix} \theta \vect{a}_1 - \theta \vect{b}_1 \\ \theta \vect{a}_2 - \theta \vect{b}_2 \\ 0 \end{bmatrix}\right\| \to 0,
\]
hence $\widehat{d}(\vect{p}, \vect{q}) = 0$ while $f(\vect{p}) = \vect{a}_1 \otimes \vect{a}_2 \otimes \vect{c} \ne \vect{b}_1 \otimes \vect{b}_2 \otimes \vect{c} = f(\vect{q})$. That is, the orbits $\Var{T} \vect{p}$ and $\Var{T} \vect{q}$ meet at infinity.

\section{The Iterated Scaling Lemma}\label{sec_isl}
We are now in a position to state the main technical result in connection to Terracini's matrix that we will need for deriving the condition number with respect to the distance measure in \refdef{def_distance_measure}. It states that a perturbation of a representative $p_i$ of a rank-$1$ tensor in the direction of the kernel of Terracini's matrix may always be interpreted as choosing a new representative of that rank-$1$ tensor $\operatorname{Seg}(p_i) \in \Var{S}_\F$ plus a small perturbation $\Delta$ that is contained in the column span of Terracini's matrix. In the formulation and proof of the following lemma, I will slightly abuse notation for brevity, writing $q \in \Var{D} p$ when I mean $\vecc{q} \in \Var{D} \vecc{p}$ with $q, p \in \F^{n_1} \times \cdots \times \F^{n_d}$.

%
%
%
%
\begin{lemma}[Iterated Scaling] \label{lem_isl}
Let $\Var{S}_\F$ be a Segre variety.
For $i=1,2,\ldots, r$, let
\[
p_i = (\sten{a}{i}{1}, \ldots, \sten{a}{i}{d}), \quad
\nabla_i = (\sten{n}{i}{1}, \ldots, \sten{n}{i}{d}),
\quad\text{and}\quad
q_i = p_i + \nabla_i = (\sten{a}{i}{1} + \sten{n}{i}{1}, \ldots, \sten{a}{i}{d} + \sten{n}{i}{d}),
\]
where $\operatorname{Seg}(p_i), \operatorname{Seg}(q_i) \in \Var{S}_\F$. 
Assume that the perturbation $\forcebold{\nabla} = \operatorname{vecr}(\nabla_1, \nabla_2, \ldots, \nabla_r)$ is of sufficiently small norm:
\begin{align} \label{eqn_isl_perturbation_bound}
 \|\forcebold{\nabla}\| \le \frac{1}{2} \lambda^{-1} = \frac{ 1 }{ 2^{d+4} (d-1)^{3/2} } \cdot \min_{1\le i \le r} \min_{2 \le k \le d} \bigl( \|\sten{a}{i}{1}\|^{-1} \|\sten{a}{i}{k}\|^2 \bigr).
\end{align}
If Terracini's matrix $T_\vect{p}$ associated with the vectorized factor matrices $\vect{p} = \operatorname{vecr}(p_1,\ldots,p_r)$ has rank $r(\Sigma+1)$ and if $\forcebold{\nabla}$ is contained in its kernel, i.e.,
\[
T_\vect{p} \forcebold{\nabla} = 0,\] then there is a representative $\dot{p}_i \in \Var{D} p_i$ such that
\(
p_i + \nabla_i = q_i = \dot{p}_i + \Delta_i 
\)
where 
\(\vect{\Delta}_i = \vecc{\Delta_i} \in \cspan(K_i)^\perp,\) and
\(\| \vect{\Delta}_i \| \le 2 \lambda \|\forcebold{\nabla}_i\|^2 \le \|\forcebold{\nabla}_i\|\)
with $\forcebold{\nabla}_i = \vecc{\nabla_i}$ and $K_i$ as in \refeqn{eqn_terracini_kernel}. 
In other words, there exists a factorization
\[
 \vect{p} + \forcebold{\nabla} = \dot{\vect{p}} + \vect{\Delta},
\]
for which
\[
\dot{\vect{p}} \sim \vect{p},\quad  \vect{\Delta} \in \cspan(K)^\perp, \text{ and}\quad \|\vect{\Delta}\| \le 2\lambda \|\forcebold{\nabla}\|^{2} \le \|\forcebold{\nabla}\|,
\]
where $K$ is as in \refeqn{eqn_terracini_kernel}. 
\end{lemma}
\begin{proof}
We prove the assertion by proving the existence of a linearly convergent sequence of representatives $p_i^{(k)} \in \Var{D} p_i$, $k=1,2,\ldots$, with the following properties
\begin{align} \tag{H0a} \label{eqn_isl_proof_h0}
 q_i - p_i^{(k)} = \Delta_i^{(k)} + \nabla_i^{(k)}, 
\end{align}
where
\[
\vect{\Delta}_i^{(k)} = \vecc{\Delta_i^{(k)}} \in \cspan(K_i)^\perp \quad\text{ and }\quad \forcebold{\nabla}_{i}^{(k)} = \vecc{\nabla_i^{(k)}} \in \cspan(K_i).
\]
The sequence will be constructed in such a way that
\begin{align} \tag{H0b} \label{eqn_isl_proof_h0prime}
 \lim_{k \to \infty} \| \forcebold{\nabla}_i^{(k)} \| \to 0.
\end{align}
As $\forcebold{\nabla}_i^{(k)}$ is contained in a finite-dimensional subspace of $\F^{\Sigma+d}$, it follows immediately that 
\(
 \forcebold{\nabla}_i^{(k)} \to 0
\)
as $k\to\infty$.
Under the assumptions of the lemma both $\vect{p}_i^{(k)} = \vecc{p_i^{(k)}}$ and $\vect{\Delta}_i^{(k)}$ will be shown to be of uniformly bounded norm, so that at least a convergent subsequence exists for which both 
\[
\lim_{k\to\infty} \sten{\Delta}{i}{(k)} \to \vect{\Delta}_i = \vecc{\Delta_i} \quad\text{ and }\quad \lim_{k\to\infty} \sten{p}{i}{(k)} \to \dot{\vect{p}}_i = \vecc{\dot{p}_i}
\]
are well defined. So it remains to prove that a sequence satisfying hypotheses \refeqn{eqn_isl_proof_h0} and \refeqn{eqn_isl_proof_h0prime} exists. 

\paragraph{\bf Part I: A recurrence relation}
Let $p_i^{(1)} = p_i$ and $\Delta_i^{(1)} = 0$, and define $\nabla^{(1)}_i = \nabla_i$. 
Since $\forcebold{\nabla}^{(1)} = \operatorname{vecr}( \nabla^{(1)}_1, \ldots, \nabla^{(1)}_r ) = \forcebold{\nabla}$ is contained in the kernel, we have
\[
 T_\vect{p} \forcebold{\nabla}^{(1)} = 
\sum_{i=1}^r T_{i} \forcebold{\nabla}_i^{(1)} = 0,
\]
where $T_\vect{p}$ and $T_i$ are as in \refeqn{eqn_terracini}.
Suppose that $T_{i}\forcebold{\nabla}_i^{(1)} \ne 0$, so that $\forcebold{\nabla}^{(1)}_i$ would not be contained in the span of $K_i$. Then $T_\vect{p}$'s rank would be strictly less than the expected value $r(\Sigma+1)$, which is a contradiction.
Consequently, the base case $k=1$ of \refeqn{eqn_isl_proof_h0} is true.
Assume now that the statement is true for all $l = 1, 2, \ldots, k$, and then we show that it holds for $k+1$ as well. 
Since $K_i$ is a basis of the kernel of $T_{i}$, we can express 
\begin{align}\label{eqn_isl_proof_expression_v}
\sten{\forcebold{\nabla}}{i}{(k)} = \begin{bmatrix} \sten{k}{i}{2} & \cdots & \sten{k}{i}{d} \end{bmatrix} \sten{v}{i}{(k)} = K_i \sten{v}{i}{(k)}
\end{align}
for some $\sten{v}{i}{(k)} \in \F^{d-1}$. By hypothesis \refeqn{eqn_isl_proof_h0}, $p_i^{(k)} \in \Var{D} p_i$, so that we can write it explicitly as
\begin{align} \label{eqn_isl_proof_pdef}
 p_i^{(k)} = ( \gamma_{1,i}^{(k)} \sten{a}{i}{1}, \gamma_{2,i}^{(k)} \sten{a}{i}{2}, \ldots, \gamma_{d,i}^{(k)} \sten{a}{i}{d} ); 
\end{align}
clearly this is true with $\gamma_{j,i}^{(1)} = 1$ for the base case, and will follow shortly for the induction step as well. Consider the perturbed representative $p_i^{(k)} + \nabla_i^{(k)}$.
By writing this perturbation $\nabla_i^{(k)}$ with respect to the particular basis $K_i$ of the kernel of $T_i$ as in \refeqn{eqn_isl_proof_expression_v}, we find that 
\begin{align} \label{eqn_isl_proof_zdef}
z_i^{(k)}
= p_i^{(k)} + \nabla_i^{(k)} 
= \biggl( (\gamma_{1,i}^{(k)} + v_{2,i}^{(k)} + \cdots + v_{d,i}^{(k)} ) \sten{a}{i}{1}, ( \gamma_{2,i}^{(k)} - v_{2,i}^{(k)} ) \sten{a}{i}{2}, \ldots, ( \gamma_{d,i}^{(k)} - v_{d,i}^{(k)}) \sten{a}{i}{d} \biggr),
\end{align}
where $v^{(k)}_{j,i}$ denotes the $(j-1)$th element of $\vect{v}^{(k)}_i$.
Consequently, $\operatorname{Seg}( z_i^{(k)} )$ and $\operatorname{Seg}( p_i^{(k)} )$ are linearly dependent, i.e., multiples of each other. I claim that the following representative
\begin{align} \label{eqn_isl_proof_defpknext}
 p_i^{(k+1)} = \Bigl( \prod_{j=2}^d ( \gamma_{j,i}^{(k)} - v_{j,i}^{(k)} )^{-1} \sten{a}{i}{1} , ( \gamma_{2,i}^{(k)} - v_{2,i}^{(k)} ) \sten{a}{i}{2}, \ldots, ( \gamma_{d,i}^{(k)} - v_{d,i}^{(k)}) \sten{a}{i}{d} \Bigr) \in \Var{D} p_i
\end{align}
will induce the required sequence. Define
\begin{align} \label{eqn_isl_proof_nextnabla}
 z_i^{(k)} - p_i^{(k+1)} = \biggl( \Bigl( \gamma_{1,i}^{(k)} + \sum_{j=2}^d v_{j,i}^{(k)} - \prod_{j=2}^d (\gamma_{j,i}^{(k)} - v_{j,i}^{(k)})^{-1} \Bigr) \sten{a}{i}{1}, 0, \ldots, 0 \biggr) = \nabla_{i}^{(k+1)} + \widehat{\Delta}_i^{(k+1)},
\end{align}
where the factorization is such that $\forcebold{\nabla}_i^{(k+1)} \in \cspan(K_i)$ and $\vecc{\widehat{\Delta}_i^{(k+1)}} \in \cspan(K_i)^\perp$. Then, 
\[
z_i^{(k)} = p_i^{(k)} + \nabla_i^{(k)} = p_i^{(k+1)} + \nabla_{i}^{(k+1)} + \widehat{\Delta}_i^{(k+1)}.
\]
Since this must be inductively true for $l = 1, 2, \ldots, k$, we find that 
\begin{align} \label{eqn_isl_proof_qi}
 q_i = z_i^{(1)} = p_i^{(k+1)} + \nabla_i^{(k+1)} + \sum_{j=2}^{k+1} \widehat{\Delta}_i^{(j)} = p_i^{(k+1)} + \nabla_i^{(k+1)} + \Delta_i^{(k+1)}.
\end{align}
This proves \refeqn{eqn_isl_proof_h0}. 

\paragraph{\bf Part II: Convergence}
It only remains to show \refeqn{eqn_isl_proof_h0prime}, i.e., the sequence of $\forcebold{\nabla}_i^{(k)}$'s converges to the zero vector. 
We demonstrate that 
\begin{align} \tag{H1} \label{eqn_isl_proof_h1} 
\|\forcebold{\nabla}_i^{(k)}\| \le \lambda^{k-1} \|\forcebold{\nabla}_i\|^k \quad\text{for some constant } \lambda \in \R,\,  0 < \lambda;
\end{align}
for the base case $k=1$ this is readily true as $\forcebold{\nabla}_i^{(1)} = \forcebold{\nabla}_i$ by definition. We may assume the truth of the above statement for $l = 1, 2, \ldots, k$. From \refeqn{eqn_isl_proof_nextnabla} it immediately follows that 
\[
 \| \forcebold{\nabla}_i^{(k+1)} \| \le \| \vect{z}_i^{(k)} - \vect{p}_i^{(k+1)} \| = \|\sten{a}{i}{1}\| \cdot \biggl| \gamma_{1,i}^{(k)} + \sum_{j=2}^{d} v_{j,i}^{(k)} - \prod_{j=2}^d ( \gamma_{j,i}^{(k)}  - v_{j,i}^{(k)} )^{-1} \biggr|,
\]
where $\sten{z}{i}{(k)} = \vecc{ z_i^{(k)} }$.
Our efforts will be focused on showing that the right hand side is bounded by $\lambda^k \|\forcebold{\nabla}_i \|^{k+1}$.
As an intermezzo, note that the boundedness of $\| \vect{\Delta}_i^{(k+1)} \|$ would then follow from the observation that 
\[
 \| \vect{\Delta}_i^{(k+1)} \| \le \sum_{j=2}^{k+1} \| \widehat{\vect{\Delta}}_i^{(j)} \| \le \sum_{j=1}^{k} \| \vect{z}_i^{(j)} - \vect{p}_i^{(j+1)} \| \le \sum_{j=1}^{k} \lambda^{j} \|\forcebold{\nabla}_i\|^{j+1} \le \lambda^{-1} \sum_{j=1}^{\infty} (\lambda \|\forcebold{\nabla}_i\|)^{j+1} = \lambda \frac{\|\forcebold{\nabla}_i\|^2}{1 - \lambda \|\forcebold{\nabla}_i\|},
\]
provided that $\lambda \|\forcebold{\nabla}_i\| < 1$ but this will be satisfied by our assumptions. The boundedness of $\vect{p}_i^{(k)}$ is then an immediate consequence of \refeqn{eqn_isl_proof_qi} and the uniform boundedness of $\vect{q}_i = \vecc{q_i}$, $\forcebold{\nabla}_i^{(k+1)}$ and $\vect{\Delta}_i^{(k+1)}$. In fact, since we will assume \refeqn{eqn_isl_proof_b3}, i.e., $\lambda \|\forcebold{\nabla}_i\| \le \tfrac{1}{2}$, it follows that $\| \vect{\Delta}_i^{(k+1)} \| \le 2 \cdot \tfrac{1}{2} \|\forcebold{\nabla}_i\|$, which already proves the bound on $\|\vect{\Delta}_i\|$.

For bounding $\|\sten{z}{i}{(k)} - \sten{p}{i}{(k+1)}\|$ from above, we proceed as follows. During our derivations, we will assume some additional convenient constraints on certain quantities; they will be considered more carefully in the next part. Consider the coefficient of $\sten{a}{i}{1}$ in \refeqn{eqn_isl_proof_nextnabla}.
First, a bound on the $v_{j,i}^{(k)}$'s is obtained as follows. 
Since $K_i \in \F^{(\Sigma+d) \times (d-1)}$ is a basis for the kernel of $T_i$ it has linearly independent columns. Therefore, the Moore--Penrose pseudoinverse is $K_i^\dagger = (K_i^H K_i)^{-1} K_i^H$, and as $\forcebold{\nabla}_i^{(k)} \in \cspan(K_i)$ one has $\forcebold{\nabla}_i^{(k)} = K_i K_i^\dagger \forcebold{\nabla}_i^{(k)}$, because $K_i K_i^\dagger$ is a projector onto the column span of $K_i$; so $\vect{v}_i^{(k)} = K_i^\dagger \forcebold{\nabla}_i^{(k)}$. If $\varsigma_i(A)$ denotes the $i$th largest singular value of $A \in \F^{m \times n}$ and $\lambda_i(A)$ denotes the $i$th largest eigenvalue of a Hermitian matrix $A \in \F^{m \times m}$, then it is well-known that
\begin{align*}
 \|K_i^\dagger\|_2 = \varsigma_1(K_i^\dagger) = \bigl( \varsigma_{d-1}(K_i) \bigr)^{-1} = \bigl( \lambda_{d-1}( K_i^H K_i ) \bigr)^{-\frac{1}{2}}.
\end{align*}
One may verify by direct computation that 
\[
K_i^H K_i = \diag(\|\sten{a}{i}{2}\|^2, \|\sten{a}{i}{3}\|^2, \ldots, \|\sten{a}{i}{d}\|^2) + \|\sten{a}{i}{1}\|^2 \vect{1}\vect{1}^T,
\] 
where $\vect{1}$ is a vector of length $d-1$ containing only ones. From \cite[Section 2.41]{Wilkinson1965} it follows that the eigenvalues of $K_i^H K_i$ satisfy 
\[
 \lambda_j( K_i^H K_i ) \ge \lambda_j \bigl( \diag(\|\sten{a}{i}{2}\|^2, \|\sten{a}{i}{3}\|^2, \ldots, \|\sten{a}{i}{d}\|^2)  \bigr), \quad j = 1, 2, \ldots, d-1,
\]
so in particular we can conclude that 
\begin{align}\label{eqn_isl_proof_chi}
\|K_i^\dagger\|_2 \le \Bigl( \min_{2\le k\le d} \|\sten{a}{i}{k}\| \Bigr)^{-1} = \chi_i.
\end{align}
We can now bound the $(j-1)$th element $v_{j,i}^{(k)}$, $j=2,3,\ldots,d$, of $\sten{v}{i}{(k)}$ as follows:
\begin{align} 
|v_{j,i}^{(k)}| 
\le \|\sten{v}{i}{(k)}\| &\le \| K_i^\dagger \|_2 \|\forcebold{\nabla}_i^{(k)}\| 
\le \chi_i \| \forcebold{\nabla}_i^{(k)} \| 
\le \chi_i \lambda^{k-1} \|\forcebold{\nabla}_i\|^{k}, \label{eqn_isl_proof_vbound}
\end{align}
where the last step is due to the induction hypothesis \refeqn{eqn_isl_proof_h1}.

For convenience, we define
\[
 \epsilon_{j,i}^{(k)} = \sum_{\ell=1}^{k} v_{j,i}^{(\ell)} \quad\text{ and }\quad \epsilon_{j,i}^{(0)} = 0, 
\]
so that 
\begin{align} \label{eqn_isl_proof_epsbound}
|\epsilon_{j,i}^{(k)}|
\le \sum_{\ell=1}^{k} |v_{j,i}^{(\ell)}| 
\le \chi_i \lambda^{-1} \sum_{\ell=1}^{k} (\lambda \|\forcebold{\nabla}_i\|)^{\ell} 
\le \chi_i \lambda^{-1} \sum_{\ell=1}^{\infty} (\lambda \|\forcebold{\nabla}_i\|)^{\ell} 
= \frac{\chi_i \|\forcebold{\nabla}_i\|}{1 - \lambda\|\forcebold{\nabla}_i\|} = C_i';
\end{align}
the second inequality is by the induction hypothesis, and the penultimate equality holds for $\lambda \|\forcebold{\nabla}_i\| < 1$, which will be satisfied by our assumptions. We additionally assume the following bound
\begin{align} \tag{B0} \label{eqn_isl_proof_b0}
 C_i' \le \frac{1}{2(d-1)},
\end{align}
which will be investigated in more detail later.
Then, for $j=2, 3, \ldots, d$ we find from \refeqn{eqn_isl_proof_pdef} and \refeqn{eqn_isl_proof_defpknext} that
\begin{subequations}\label{eqn_isl_proof_gamma}
\begin{align}
\gamma_{j,i}^{(k+1)} = \gamma_{j,i}^{(k)} - v_{j,i}^{(k)} = \cdots = \gamma_{j,i}^{(1)} - \sum_{\ell=1}^{k} v_{j,i}^{(\ell)} = 1 - \epsilon_{j,i}^{(k)}, \quad j = 2, 3, \ldots, d.
\end{align}
By induction, the above formula also applies for $\gamma_{j,i}^{(\ell)}$ with $\ell = 1, 2, \ldots, k$. 
For $j=1$ we have 
\begin{align}
\gamma_{1,i}^{(k+1)} = \prod_{j=2}^d \bigl( \gamma_{j,i}^{(k)} - v_{j,i}^{(k)} \bigr)^{-1} 
= \prod_{j=2}^d \bigl(1 - \epsilon_{j,i}^{(k-1)} - v_{j,i}^{(k)} \bigr)^{-1}
= \prod_{j=2}^d \bigl(1 - \epsilon_{j,i}^{(k)} \bigr)^{-1},
\end{align}
\end{subequations}
which by induction is also true for $\gamma_{1,i}^{(\ell)}$ with $\ell =1,2,\ldots, k$.
With the above observations, the coefficient of $\sten{a}{i}{1}$ in \refeqn{eqn_isl_proof_nextnabla} may be written as
\begin{align} \label{eqn_isl_proof_acoeff}
\gamma_{1,i}^{(k)} + \sum_{j=2}^d v_{j,i}^{(k)} - \gamma_{1,i}^{(k+1)} 
= \prod_{j=2}^d \bigl( 1 - \epsilon_{j,i}^{(k-1)} \bigr)^{-1} + \sum_{j=2}^d v_{j,i}^{(k)} - \prod_{j=2}^d \bigl( 1 - \epsilon_{j,i}^{(k-1)} - v_{j,i}^{(k)} \bigr)^{-1}.
\end{align}
For proceeding, we will need some fairly convoluted series expansions; however, the key idea revolves around expanding $(1 - x)^{-1}$ with $x \approx 0$ by a Maclaurin series. Since we will be confronted with products of such expansions, it is worthwhile to recall the following general formula:
\begin{align} \label{eqn_isl_proof_maclaurin}
 \prod_{j=1}^n (1 - x_j)^{-1} = \prod_{j=1}^n \sum_{\kappa=0}^\infty x_j^\kappa = \sum_{\kappa=0}^\infty \; \sum_{\|\forcebold{\ell}\|_1 = \kappa} \; \prod_{j=1}^n x_{j}^{\ell_j}, 
\end{align}
where $\forcebold{\ell} = \left[\begin{smallmatrix} \ell_1 & \ell_2 & \cdots & \ell_n\end{smallmatrix}\right] \in \mathbb{N}^{n}$ and $\|\forcebold{\ell}\|_1 = \sum_{j=1}^n |\ell_j|$ is the $1$-norm of $\forcebold{\ell}$;
in the expression it is assumed that $|x_j| < 1$ so that all expansions, including their products, are absolutely convergent.
The last term in \refeqn{eqn_isl_proof_acoeff} can be rewritten as follows:
\begin{align*}
 Z 
= \prod_{j=2}^d \bigl( 1 - \epsilon_{j,i}^{(k-1)} - v_{j,i}^{(k)} \bigr)^{-1} 
&= \prod_{\ell=2}^d \bigl(1 - \epsilon_{\ell,i}^{(k-1)} \bigr)^{-1} \prod_{j=2}^d \frac{1-\epsilon_{j,i}^{(k-1)}}{1 - \epsilon_{j,i}^{(k-1)} - v_{j,i}^{(k)} } \\
&= \prod_{\ell=2}^d \bigl(1 - \epsilon_{\ell,i}^{(k-1)} \bigr)^{-1} \prod_{j=2}^d \Bigl( 1 - \frac{v_{j,i}^{(k)}}{ 1 - \epsilon_{j,i}^{(k-1)} } \Bigr)^{-1} \\
&= \prod_{\ell=2}^d \bigl(1 - \epsilon_{\ell,i}^{(k-1)} \bigr)^{-1} \sum_{\kappa=0}^\infty \sum_{\|\forcebold{\ell}\|_1 = \kappa} \prod_{j=2}^d \Bigl( \frac{ v_{j,i}^{(k)} }{ 1 - \epsilon_{j,i}^{(k-1)} } \Bigr)^{\ell_j},
\end{align*}
where the last step was by \refeqn{eqn_isl_proof_maclaurin}, which requires that 
\begin{align} \tag{B1} \label{eqn_isl_proof_b1}
|v_{j,i}^{(k)}| < |1 - \epsilon_{j,i}^{(k-1)}|;
\end{align}
this hypothesis will be investigated later.
Let us write $Z = S_0 + S_1 + S_\infty$, where 
\begin{align*}
S_0 &= \prod_{j=2}^d \bigl(1 - \epsilon_{j,i}^{(k-1)} \bigr)^{-1} = \gamma_{1,i}^{(k)}, \\ 
S_1 &= \prod_{\ell=2}^d \bigl(1 - \epsilon_{\ell,i}^{(k-1)} \bigr)^{-1} \cdot \sum_{j=2}^d \frac{v_{j,i}^{(k)}}{1 - \epsilon_{j,i}^{(k-1)}}, \text{ and} \\
S_\infty &= \prod_{\ell=2}^d \bigl(1 - \epsilon_{\ell,i}^{(k-1)} \bigr)^{-1} \cdot \sum_{\kappa=2}^\infty \; \sum_{\|\forcebold{\ell}\|_1 = \kappa} \prod_{j=2}^d \Bigl( \frac{ v_{j,i}^{(k)} }{ 1 - \epsilon_{j,i}^{(k-1)} } \Bigr)^{\ell_j}.
\end{align*}
We will need to apply \refeqn{eqn_isl_proof_maclaurin} again to $S_1$ for obtaining the desired result. 
\begin{align*}
 S_1 
&= \sum_{j=2}^d v_{j,i}^{(k)} \Bigl( 1 + \frac{\epsilon_{j,i}^{(k-1)} }{1 - \epsilon_{j,i}^{(k-1)}} \Bigr) \cdot \prod_{\ell=2}^d \bigl(1 - \epsilon_{\ell,i}^{(k-1)} \bigr)^{-1} \\
&= \sum_{j=2}^d v_{j,i}^{(k)} \Bigl( 1 + \frac{\epsilon_{j,i}^{(k-1)} }{1 - \epsilon_{j,i}^{(k-1)}} \Bigr) \cdot \Bigl(1 + \sum_{\kappa=1}^\infty \sum_{\|\forcebold{\ell}\|_1 = \kappa} \prod_{l=2}^d \bigl( \epsilon_{l,i}^{(k-1)} \bigr)^{\ell_l} \Bigr) \\
&= 
\underset{S_1'}{\underbrace{\sum_{j=2}^d v_{j,i}^{(k)}}}
+ \underset{S_1''}{\underbrace{\sum_{j=2}^d v_{j,i}^{(k)} \Bigl(\sum_{\kappa=1}^\infty \sum_{\|\forcebold{\ell}\|_1 = \kappa} \prod_{l=2}^d \bigl( \epsilon_{l,i}^{(k-1)} \bigr)^{\ell_l} \Bigr) }}
+ \underset{S_1'''}{\underbrace{\sum_{j=2}^d v_{j,i}^{(k)} \frac{\epsilon_{j,i}^{(k-1)}}{1 - \epsilon_{j,i}^{(k-1)}} \cdot \prod_{\ell=2}^d \bigl(1 - \epsilon_{\ell,i}^{(k-1)} \bigr)^{-1}}}.
\end{align*}
It is now clear that 
\begin{align*}
 \refeqn{eqn_isl_proof_acoeff} = S_0 + S_1' - Z = S_0 + S_1' - (S_0 + S_1' + S_1'' + S_1''' + S_\infty) = -(S_1'' + S_1''' + S_\infty).
\end{align*}
If we could show that the last inequality in
\[
|\refeqn{eqn_isl_proof_acoeff}| = |S_1'' + S_1''' + S_\infty| \le |S_1''| + |S_1'''| + |S_\infty| \le \|\sten{a}{i}{1}\|^{-1} \lambda^{k} \|\forcebold{\nabla}_i\|^{k+1}
\]
holds, then the proof of \refeqn{eqn_isl_proof_h1} would be concluded. We proceed along this line.
From \refeqn{eqn_isl_proof_epsbound}, 
\[
 1 - \epsilon_{j,i}^{(k-1)} \ge 1 - |\epsilon_{j,i}^{(k-1)}| \ge 1 - C_i',
\]
and as $|\epsilon_{j,i}^{(k-1)}| \le \tfrac{1}{2}$ by \refeqn{eqn_isl_proof_b0}, it follows that 
\begin{align*}
\bigl| 1 - \epsilon_{j,i}^{(k-1)} \bigr|^{-1} = \bigl(1 - \epsilon_{j,i}^{(k-1)} \bigr)^{-1} \le ( 1 - C_i' )^{-1} = C_i;
\end{align*}
this bound is uniform in the sense that it depends neither on $j$ nor $k$.
Consequently,
\begin{align}
 \Biggl| \sum_{\|\forcebold{\ell}\|_1 = \kappa} \prod_{j=2}^d \Bigl( \frac{ v_{j,i}^{(k)} }{ 1 - \epsilon_{j,i}^{(k-1)} } \Bigr)^{\ell_j} \Biggr| 
\le C_i^\kappa \sum_{\|\forcebold{\ell}\|_1 = \kappa} \prod_{j=2}^d \bigl| v_{j,i}^{(k)} \bigr|^{\ell_j}
&\le C_i^{\kappa} \| \sten{v}{i}{(k)} \otimes \cdots \otimes \sten{v}{i}{(k)} \|_1 \nonumber \\
&\le C_i^{\kappa} (d-1)^{\kappa/2} \|\sten{v}{i}{(k)}\|^\kappa, \label{eqn_isl_proof_sigpi1}
\end{align}
where the second step is because every $\forcebold{\ell}$ with $\|\forcebold{\ell}\|_1=\kappa$ can be identified with the multi-index
\[
 ( \underset{\ell_1}{\underbrace{1, \ldots, 1}}, \underset{\ell_2}{\underbrace{2,\ldots,2}}, \ldots, \underset{\ell_{d-1}}{\underbrace{d-1, \ldots, d-1}} )
\]
whose length is $\kappa$. From this identification, it follows that the symmetric part of $\sten{v}{i}{(k)} \otimes \cdots \otimes \sten{v}{i}{(k)}$ contains precisely all summands of $\sum_{\|\forcebold{\ell}\|_1 = \kappa} \prod_{j=2}^d \bigl| v_{j,i}^{(k)} \bigr|^{\ell_j}$. As a result, $S_\infty$ can be bounded as follows
\begin{align*}
|S_\infty| 
\le C_i^{d-1} \sum_{\kappa=2}^\infty \bigl( C_i \sqrt{d-1} \|\vect{v}_i^{(k)}\| \bigr)^\kappa
\le \frac{ C_i^{d+1} (d-1) \chi_i^{2} \lambda^{2k-2} \|\forcebold{\nabla}_i\|^{2k}}{1 - C_i \sqrt{d-1} \chi_i \lambda^{k-1}\|\forcebold{\nabla}_i\|^k},
\end{align*}
where \refeqn{eqn_isl_proof_vbound} was used to bound $\|\sten{v}{i}{(k)}\|$ in the second inequality; herein, we assumed that
\begin{align} \tag{B2} \label{eqn_isl_proof_b2}
 C_i \sqrt{d-1} \chi_i \lambda^{k-1} \| \forcebold{\nabla}_i \|^k \le \frac{1}{2}
\end{align}
so that the series is convergent.
In an analogous fashion as in the derivation of \refeqn{eqn_isl_proof_sigpi1}, one finds
\begin{align}\label{eqn_isl_proof_epssumbound}
 \Bigl| \sum_{\|\forcebold{\ell}\|_1=\kappa} \prod_{j=2}^d \bigl( \epsilon_{j,i}^{(k-1)} \bigr)^{\ell_j} \Bigr| \le \|\forcebold{\epsilon}_i^{(k-1)} \otimes \cdots \otimes \forcebold{\epsilon}_i^{(k-1)}\|_1 \le (d-1)^{\kappa/2} \| \forcebold{\epsilon}_i^{(k-1)} \|^\kappa \le \bigl( (d-1) C_i' \bigr)^\kappa,
\end{align}
where $\forcebold{\epsilon}_{i}^{(k)} = \bigl[\begin{smallmatrix} \epsilon_{2,i}^{(k-1)} & \cdots & \epsilon_{d,i}^{(k-1)} \end{smallmatrix}\bigr]$ and the last step is due to \refeqn{eqn_isl_proof_epsbound}. Then, the bound for $S_1''$ is 
\begin{align*}
|S_1''| 
&\le \sum_{j=2}^d \bigl| v_{j,i}^{(k)} \bigr| \cdot \Bigl( \sum_{\kappa=1}^\infty \bigl( (d-1) C_i' \bigr)^\kappa  \Bigr) 
= \frac{(d-1)C_i'}{1-(d-1)C_i'} \sum_{j=2}^d \bigl| v_{j,i}^{(k)} \bigr| = \frac{(d-1)C_i'}{1-(d-1)C_i'} \| \vect{v}_i^{(k)} \|_1;
\end{align*}
the Maclaurin series is convergent because of \refeqn{eqn_isl_proof_b0}.
We continue with a bound for $S_1'''$: 
\begin{align*}
 |S_1'''| 
\le \sum_{j=2}^d |v_{j,i}^{(k)} | C_i' C_i \prod_{\ell=2}^d |1 - \epsilon_{\ell,i,k-1}|^{-1} 
\le C_i' C_i^{d} \sum_{j=2}^d |v_{j,i}^{(k)} | = C_i' C_i^d \| \vect{v}_i^{(k)} \|_1.
\end{align*}
It follows from the foregoing two bounds that 
\begin{align*}
|S_1''| + |S_1'''| 
\le \Bigl( \frac{(d-1)C_i'}{1-(d-1)C_i'} + C_i' C_i^d \Bigr) \| \vect{v}_i^{(k)} \|_1 
&\le C_i' \sqrt{d-1} \Bigl( \frac{d-1}{1-(d-1)C_i'} + C_i^d \Bigr) \chi_i \lambda^{k-1} \|\forcebold{\nabla}_i\|^k \\
&\le \Bigl( \frac{d-1}{1-(d-1)C_i'} + C_i^d \Bigr) \frac{\chi_i^{2} \sqrt{d-1}}{1 - \lambda\|\forcebold{\nabla}_i\|} \lambda^{k-1} \|\forcebold{\nabla}_i\|^{k+1},
\end{align*}
where we used $\|\vect{v}_i^{(k)}\|_1 \le \sqrt{d-1} \|\vect{v}_i^{(k)}\|$ and \refeqn{eqn_isl_proof_vbound}. Let $\zeta = \frac{d-1}{1-(d-1)C_i'} + C_i^d$. Combining the above with the bound for $S_\infty$, we get  
\[
 |\refeqn{eqn_isl_proof_acoeff}| 
\le \chi_i^{2} \sqrt{d-1} \Bigl(  \frac{\zeta}{1-\lambda\|\forcebold{\nabla}_i\|} + \frac{ C_i^{d+1} \sqrt{d-1} \lambda^{k-1} \|\forcebold{\nabla}_i\|^{k-1} }{1 - C_i \sqrt{d-1} \chi_i \lambda^{k-1}\|\forcebold{\nabla}_i\|^k} \Bigr) \lambda^{k-1} \|\forcebold{\nabla}_i\|^{k+1}.
\]
Therefore, it suffices to demonstrate that the right hand side in the above expression is less than $\|\sten{a}{i}{1}\|^{-1} \lambda^k \|\forcebold{\nabla}_i\|^{k+1}$. So, it suffices showing that
\begin{align} \label{eqn_isl_proof_abound}
\|\sten{a}{i}{1}\| \chi_i^2  \sqrt{d-1} \Bigl(  \frac{\zeta}{1-\lambda\|\forcebold{\nabla}_i\|} + \frac{ C_i^{d+1} \sqrt{d-1} \lambda^{k-1} \|\forcebold{\nabla}_i\|^{k-1} }{1 - C_i \sqrt{d-1} \chi_i \lambda^{k-1}\|\forcebold{\nabla}_i\|^k} \Bigr) \le \lambda.
\end{align}
Let us assume additionally that 
\begin{align} \tag{B3} \label{eqn_isl_proof_b3}
 \lambda \|\forcebold{\nabla}_i\| \le \frac{1}{2}.
\end{align}
Then, by exploiting \refeqn{eqn_isl_proof_b0}, \refeqn{eqn_isl_proof_b2}, and \refeqn{eqn_isl_proof_b3}, it follows that \refeqn{eqn_isl_proof_abound} is implied by the following inequality
\begin{align}\label{eqn_isl_proof_b4}\tag{B4}
\|\sten{a}{i}{1}\| \chi_i^2  \sqrt{d-1} \bigl( 4(d-1) + 2^{d+1} + 2^{d-k+3} \sqrt{d-1} \bigr)
\le \lambda. 
\end{align}
Note that $\lambda > 0$ is a free parameter, i.e., it is not specified as input to the lemma, so it can simply be chosen so as to satisfy the above inequality. For instance, let us choose 
\[
\lambda = 2^{d+3} (d-1)^{3/2} \cdot \max_{1 \le i \le r} \bigl( \chi_i^2 \|\sten{a}{i}{1}\| \bigr).
\]
Provided that \refeqn{eqn_isl_proof_b0}, \refeqn{eqn_isl_proof_b1}, \refeqn{eqn_isl_proof_b2}, \refeqn{eqn_isl_proof_b3} and \refeqn{eqn_isl_proof_b4} are true, this proves \refeqn{eqn_isl_proof_h1}, which implies \refeqn{eqn_isl_proof_h0prime}.

\paragraph{\bf Part III: Eliminating assumptions}
For concluding the proof, we show that \refeqn{eqn_isl_perturbation_bound} implies \refeqn{eqn_isl_proof_b0}, \refeqn{eqn_isl_proof_b1}, \refeqn{eqn_isl_proof_b2} and \refeqn{eqn_isl_proof_b3}. 
The assumption \refeqn{eqn_isl_proof_b0} can be eliminated as follows:
\begin{align*}
\refeqn{eqn_isl_proof_b0} 
&\quad\Leftrightarrow\quad \frac{\|\forcebold{\nabla}_i\|}{1 - \lambda\|\forcebold{\nabla}_i\|} \le \frac{1}{2(d-1)} \chi_i^{-1}
\quad\Leftarrow\quad \Bigl( \|\forcebold{\nabla}_i\| \le \frac{1}{4(d-1)} \chi_i^{-1} \Bigr) \wedge \refeqn{eqn_isl_proof_b3},
\end{align*}
where $\wedge$ denotes the logical conjunction. The last statement also implies \refeqn{eqn_isl_proof_b1}, because we have
\begin{align*}
\refeqn{eqn_isl_proof_b1} 
\quad\Leftarrow\quad
\chi_i \lambda^{k-1} \|\forcebold{\nabla}_i\|^k < 1 - \frac{\chi_i \|\forcebold{\nabla}_i\|}{1 - \lambda\|\forcebold{\nabla}_i\|}
&\quad\Leftarrow\quad \Bigl( \chi_i \|\forcebold{\nabla}_i\| < 1 - \frac{\chi_i \|\forcebold{\nabla}_i\|}{1 - \lambda\|\forcebold{\nabla}_i\|} \Bigr) \wedge \refeqn{eqn_isl_proof_b3} \\
&\quad\Leftarrow\quad \biggl( \chi_i \|\forcebold{\nabla}_i\| \Bigl(1 + \frac{1}{1 - \lambda\|\forcebold{\nabla}_i\|} \Bigr)  < 1 \biggr) \wedge \refeqn{eqn_isl_proof_b3} \\
&\quad\Leftarrow\quad \Bigl( \|\forcebold{\nabla}_i\| \le \frac{1}{3} \chi_i^{-1} \Bigr) \wedge \refeqn{eqn_isl_proof_b3}.
\end{align*}
The elimination of \refeqn{eqn_isl_proof_b2} proceeds as follows:
\begin{align*}
\refeqn{eqn_isl_proof_b2}
\quad\Leftrightarrow\quad
\lambda^{k-1}\|\forcebold{\nabla}_i\|^k \le \frac{1}{2 \sqrt{d-1} C_i } \chi_i^{-1}
&\quad\Leftarrow\quad \Bigl( \|\forcebold{\nabla}_i\| \le \frac{1}{4\sqrt{d-1}} \chi_i^{-1} \Bigr) \wedge \refeqn{eqn_isl_proof_b0} \wedge \refeqn{eqn_isl_proof_b3} \\
&\quad\Leftarrow\quad \Bigl( \|\forcebold{\nabla}_i\| \le \frac{1}{4(d-1)} \chi_i^{-1} \Bigr) \wedge \refeqn{eqn_isl_proof_b3}.
\end{align*}
It follows that \refeqn{eqn_isl_proof_b3}, or, equivalently, $\|\forcebold{\nabla}_i\| \le \tfrac{1}{2} \lambda^{-1}$ is the strongest bound, because we have
\begin{align*}
\|\forcebold{\nabla}_i\| 
\le \frac{1}{2} \lambda^{-1} 
= \frac{1}{2^{d+2}\sqrt{d-1}} \cdot \frac{1}{4(d-1)} \min_{1 \le i \le r} \bigl( \|\sten{a}{i}{1}\|^{-1} \chi_i^{-2} \bigr)
< \frac{1}{4(d-1)} \chi_i^{-1}
< \frac{1}{3} \chi_i^{-1},
\end{align*}
where we exploited $\|\sten{a}{i}{1}\|^{-1} \chi_i^{-1} = \|\sten{a}{i}{1}\|^{-1} \min_{2 \le k \le d} \|\sten{a}{i}{k}\| \le \|\sten{a}{i}{1}\| \|\sten{a}{i}{1}\|^{-1} = 1$.
Thus, \refeqn{eqn_isl_proof_b3} implies the assumptions \refeqn{eqn_isl_proof_b0}, \refeqn{eqn_isl_proof_b1}, \refeqn{eqn_isl_proof_b2}, \refeqn{eqn_isl_proof_b3}, \refeqn{eqn_isl_proof_b4} as well as the hypotheses \refeqn{eqn_isl_proof_h0}, \refeqn{eqn_isl_proof_h0prime} and \refeqn{eqn_isl_proof_h1} for all $i=1,2,\ldots,r$ simultaneously, hereby concluding the main proof.

\paragraph{\bf Part IV: Alternative formulation}
It is clear that taking 
\[
 \vect{\Delta}^T = \begin{bmatrix} \vect{\Delta}_1^T & \cdots & \vect{\Delta}_r^T \end{bmatrix}
 \quad\text{and}\quad
 \dot{\vect{p}}^T = \begin{bmatrix} \vecc{\dot{p}_1}^T & \cdots & \vecc{\dot{p}_r}^T \end{bmatrix}
\]
satisfies $\dot{\vect{p}} \sim \vect{p}$ and $\|\vect{\Delta}\| \le \|\forcebold{\nabla}\|$. 
By assumption, the columns of $K$ form a basis of the kernel of $T_\vect{p}$. Hence, the Moore--Penrose pseudoinverse of $K$ is $K^\dagger = (K^H K)^{-1} K^H$. As $K K^\dagger$ is a projector onto the column span of $K$, the claim $\vect{\Delta} \in \cspan(K)^\perp$ is equivalent with 
\(
K K^\dagger \vect{\Delta} = 0. 
\)
Since 
\[
 K^H \vect{\Delta} = 
 \sum_{i=1}^r K_i^H \vect{\Delta}_i = 0,
\]
the proof is concluded.
\end{proof}
%
%
%
%
%

As can be understood from Part I in the proof of the Iterated Scaling Lemma, there even exists an iterative algorithm for obtaining the desired factorization. The lemma essentially proves that it always converges for sufficiently small input perturbations $\forcebold{\nabla}$. An implementation of this algorithm is included in the ancillary files accompanying this paper; some numerical experiments illustrating the Iterated Scaling Lemma will be presented in \refsec{sec_illustration_main_thm}.

A consequence of the proof of \reflem{lem_isl} is that Terracini's matrix in the new representatives $\dot{p}_i$ is not very different from Terracini's matrix in the original unperturbed representatives $p_i$.

%
%
%
%
%
%
\begin{corollary}\label{cor_isl_terracini}
Let all the assumptions of \reflem{lem_isl} hold, and let $\forcebold{\nabla}$, $p_i$ and $\dot{p}_i$ be as in \reflem{lem_isl}. Let $T_\vect{p}$ be Terracini's matrix \refeqn{eqn_terracini} in $\vect{p}=\operatorname{vecr}(p_1,\ldots,p_r)$, and let $T_{\dot{\vect{p}}}$ be Terracini's matrix in $\dot{\vect{p}} = \operatorname{vecr}(\dot{p}_1,\ldots,\dot{p}_r)$.
Then, 
\[
 T_{\dot{\vect{p}}} = T_\vect{p} \dot{D} = T_\vect{p} (I + \dot{E}),
\]
where the diagonal matrix $\dot{D}$ tends to the identity as $\|\forcebold{\nabla}\|$ tends to zero; specifically, there is a constant $C>0$ such that 
\(
\|\dot{E}\|_2 = \|\dot{D} - I\|_2 \le C \|\forcebold{\nabla}\|.
\)
\end{corollary}
\begin{proof}
Consider the definition of the representative $p_i^{(k)}$ in \refeqn{eqn_isl_proof_pdef}. The limit for $k\to\infty$ can be written as
\begin{align}
 \dot{p}_i 
= \lim_{k\to\infty} p_i^{(k)} 
&= \Bigl( \sten{a}{i}{1} \cdot \lim_{k\to\infty} \prod_{j=2}^{d} \bigl( 1 - \epsilon_{j,i}^{(k)} \bigr)^{-1}, \sten{a}{i}{2} \cdot \lim_{k\to\infty} \bigl(1 - \epsilon_{2,i}^{(k)}\bigr), \ldots, \sten{a}{i}{d} \cdot \lim_{k\to\infty} \bigl(1 - \epsilon_{d,i}^{(k)}\bigr) \Bigr) \nonumber \\
&= ( \gamma_{1,i} \sten{a}{i}{1}, \gamma_{2,i} \sten{a}{i}{2}, \ldots, \gamma_{d,i} \sten{a}{i}{d} )
\label{eqn_isl_cor_diagonal}
\end{align}
because of \refeqn{eqn_isl_proof_gamma}.
Note that the series
\begin{align}\label{eqn_isl_cor_eps}
 \gamma_{j,i} = 1 - \epsilon_{j,i}^{(\infty)} = 1 - \sum_{k=1}^\infty v_{j,i}^{(k)} > 0, \quad j=2,3,\ldots,d,
\end{align}
are absolutely convergent because the terms in the infinite sequences may be bounded as in \refeqn{eqn_isl_proof_vbound}, and the sum $\chi_i \lambda^{-1} \sum_{k=1}^\infty (\lambda\|\forcebold{\nabla}\|)^k$ is absolutely convergent because of \refeqn{eqn_isl_proof_b3}. Since 
\begin{align}\label{eqn_isl_cor_gam0}
 \gamma_{1,i} = (\gamma_{2,i} \gamma_{3,i} \cdots \gamma_{d,i})^{-1},
\end{align}
the above $\gamma_{j,i}$ yield the explicit expressions for the coefficients of $\dot{p}_i$. Notice that the above furnishes an alternative proof that the sequence of $p_i^{(k)}$'s converges. It follows immediately from the definition of Terracini's matrix in \refeqn{eqn_terracini} that we can write
\begin{align*}
\dot{T}_{i} 
= T_{i} \cdot \diag\Biggl( I_{n_1} \cdot \prod_{\substack{j=1\\ j \ne 1}}^d \gamma_{j,i},\, \ldots,\, I_{n_d} \cdot \prod_{\substack{j=1\\ j \ne d}}^d \gamma_{j,i} \Biggr) 
= T_{i} \cdot \diag\bigl( \gamma_{1,i}^{-1} I_{n_1}, \gamma_{2,i}^{-1} I_{n_2}, \ldots, \gamma_{d,i}^{-1} I_{n_d} \bigr)
= T_{i} \dot{D}_i,
\end{align*}
so that Terracini's matrix $T_{\dot{\vect{p}}}$ in the points $\dot{p}_i$ is given by
\[
 T_{\dot{\vect{p}}} = \begin{bmatrix}\dot{T}_1 & \cdots & \dot{T}_r \end{bmatrix} = T_\vect{p} \cdot \diag(\dot{D}_1, \dot{D}_2, \ldots, \dot{D}_r) = T_\vect{p} \dot{D}.
\]
From the definition of $\dot{D}$ it suffices demonstrating that $|\gamma_{j,i}^{-1} - 1| \le C \|\forcebold{\nabla}\|$. For $j=2,3,\ldots,d$, we have
\begin{align}\label{eqn_isl_cor_gam1}
|\gamma_{j,i}^{-1} - 1| = | (1 - \epsilon_{j,i}^{(\infty)})^{-1} - 1| \le \sum_{\kappa=1}^\infty |\epsilon_{j,i}^{(\infty)}|^\kappa \le \sum_{\kappa=1}^\infty (2 \chi_i \|\forcebold{\nabla}_i\| )^\kappa \le  4 \|\forcebold{\nabla}\| \max_{1 \le i \le r} \chi_i,
\end{align}
where in the second equality we used \refeqn{eqn_isl_proof_epsbound} and \refeqn{eqn_isl_proof_b3}, and where in the last step we used $2\chi_i \|\forcebold{\nabla}_i\| \le \tfrac{1}{2}$ which is true because of \refeqn{eqn_isl_proof_b3}. The case of $j=1$ is due to
\begin{align}
 |\gamma_{1,i}^{-1}-1| 
&= \Bigl| - 1 + \prod_{j=2}^d \bigl(1 - \epsilon_{j,i}^{(\infty)} \bigr) \Bigr|
= \Bigl| - 1 + 1 + \sum_{\kappa=1}^d \sum_{\substack{\|\forcebold{\ell}\|_1 = \kappa,\\ \forcebold{\ell} \le 1}} (-1)^\kappa \prod_{j=2}^d \bigl( \epsilon_{j,i}^{(\infty)} \bigr)^{\ell_j} \Bigr| \nonumber \\
&\le \sum_{\kappa=1}^\infty \sum_{\|\forcebold{\ell}\|_1 = \kappa} \prod_{j=2}^d |\epsilon_{j,i}^{(\infty)}|^{\ell_j} 
\le \sum_{\kappa=1}^\infty \bigl( (d-1) C_i' \bigr)^\kappa 
= \frac{(d-1) C_i'}{1 - (d-1)C_i'} 
\le  4 (d-1) \|\forcebold{\nabla}\| \max_{1\le i\le r} \chi_i;
\label{eqn_isl_cor_gam2}
\end{align}
herein, the inequality $\forcebold{\ell}\le1$ is meant componentwise, the second inequality is by \refeqn{eqn_isl_proof_epssumbound}, and the last step used \refeqn{eqn_isl_proof_epsbound}, \refeqn{eqn_isl_proof_b0} and \refeqn{eqn_isl_proof_b3}. Letting $C = 4(d-1) \max_{1\le i\le r} \chi_i$ then concludes the proof. 
\end{proof}

%
%
%
%
%
%
%
%
%

%
%
%
%
%
%
%
\section{A condition number} \label{sec_condition}
The main result of this paper states that the least nonzero singular value of Terracini's matrix is an absolute condition number with respect to the distance measure in \refdef{def_distance_measure}.

\begin{theorem}[Absolute condition number] \label{thm_condition}
Let $N = r(\Sigma+1)$. Let 
\(
b_i = (\sten{a}{i}{1}, \sten{a}{i}{2}, \ldots, \sten{a}{i}{d})
\)
be representatives of $p_i = \operatorname{Seg}(b_i) \in \Var{S}_\F$, $i=1,\ldots,r$. Let $\vect{p} = \operatorname{vecr}(b_1,b_2,\ldots,b_r)$ be given.
Assume that 
\[
\tensor{A} = f(\vect{p}) = \sum_{i=1}^r \sten{a}{i}{1} \otimes \cdots \otimes \sten{a}{i}{d}  \quad\in \SecZ{r}{\Var{S}_\F} \subset \F^{\Pi} 
\]
is robustly $r$-identifiable, and let $\varsigma_N$ denote the $N$th largest singular value of Terracini's matrix $T_\vect{p}$. Then, the absolute condition number of the rank decomposition problem at $\vect{p}$ is 
\begin{align*}
\kappa_A = \lim_{\epsilon\to0} \max_{\substack{\|\Delta\tensor{A}\|\le\epsilon,\\\tensor{A}+\Delta\tensor{A}\in\SecZ{r}{\Var{S}_\F}}} \frac{d( \vect{p}, f^{\dagger}(\tensor{A}+\Delta\tensor{A}))}{\|\Delta\tensor{A}\|} = \varsigma_N^{-1}.
\end{align*}
If $\varsigma_N = 0$ or $\tensor{A}$ is not robustly $r$-identifiable, then the condition number is defined to be $\infty$. 
\end{theorem}
\begin{proof}
Observe that $f$ is an analytic multivariate polynomial that is homogeneous of degree $d$ in the $M = r(\Sigma+d)$ variables $x_i$, so that it has a finite, convergent Taylor series expansion at every point. Let $\Delta\vect{p} \in \F^{M}$ be arbitrary. For $f_i(\vect{p}+\Delta\vect{p})$, where $f_i$ is the $i$th component of the vector function $f$, this expansion around $\vect{p}$ is given by
\begin{align*}
 f_i(\vect{p} + \Delta\vect{p}) 
&= f_i(\vect{p}) + \sum_{1 \le k_1 \le M} \Delta p_{k_1} \Bigl(\frac{\partial \, f_i}{\partial x_{k_1}}\Bigr)(\vect{p}) + \sum_{1 \le k_1 < k_2 \le M} \Delta p_{k_1} \Delta p_{k_2} \Bigl(\frac{\partial^2 \, f_i}{\partial x_{k_1} \partial x_{k_2}}\Bigr) (\vect{p}) + \cdots,
\end{align*}
where we exploited the observation that the terms corresponding to the higher-order partial derivatives with respect to the same variable are zero, so that we can get rid of the multinomial coefficients. 
Then,
\[
 f(\vect{p}+\Delta\vect{p}) = f(\vect{p}) + T_\vect{p} \Delta\vect{p} + \mathcal{O}(\|\Delta\vect{p}\|^2),
\]
where $T_\vect{p}$ is Terracini's matrix as in \refeqn{eqn_terracini}.

Let $\Var{N}$ be the neighborhood of $\tensor{A}$ where all tensors are robustly $r$-identifiable and hence of rank $r$. 
Let $\tensor{A}' \in \Var{N}$ be arbitrary. Then, there exist several $\Delta\vect{p}'$ such that $\tensor{A}' = \tensor{A} + \Delta\tensor{A} = f(\vect{p}+\Delta\vect{p}')$. The remainder of the proof will be slightly easier if we enforce one unique particular choice. Note that by $r$-identifiability, $f^\dagger(\tensor{A}') = \Var{T}(\vect{p} +\Delta\vect{p}')$ so that the set of rank-$1$ tensors appearing in the decomposition of $\tensor{A}'$ is unique. The $r$ elements of this set can be ordered uniquely with respect to the lexicographic total order $\le$ on $\F^\Pi$.\footnote{For $\F=\C\simeq\R^2$, one should consider the lexicographic order after the identification $\C^\Pi = \R^{2\Pi}$.} This removes the permutation ambiguity, leaving only the scaling indeterminacies. These can be removed as follows. Let $b_i' = (\sten{x}{i}{1}, \sten{x}{i}{2}, \ldots, \sten{x}{i}{d})$ be a set of representatives such that the corresponding rank-$1$ tensors $\Seg{b_i'}$ are sorted. Then, the following set is uniquely defined regardless of the particular choice of representatives for the rank-$1$ tensors:
\begin{align}\label{eqn_main_proof_def_C}
 \Var{C} = \biggl\{ j_{k,i} = \min \Bigl\{ \underset{1\le\ell\le n_k}{\arg\max} |x_{\ell,i}^{(k)}| \Bigr\} \;\big|\; k=2,3,\ldots,d, \;i=1,2,\ldots,r \biggr\},
\end{align}
where $x_{\ell,i}^{(k)}$ is the $\ell$th element of $\sten{x}{i}{k}\in\F^{n_k}$. Note that $|x_{j_{k,i},i}^{(k)}| = \|\sten{x}{i}{k}\|_\infty > 0$ by $r$-identifiability, and that $|\Var{C}|=r(d-1)$. It follows from $r$-identifiability that there is just one choice of vectorized factor matrices $\vect{p}+\Delta\vect{p} = \operatorname{vecr}(b_1'', b_2'', \ldots, b_r'')$ with 
\begin{align}\label{eqn_main_proof_normalization}
 b_i'' = (\sten{b}{i}{1}, \sten{b}{i}{2}, \ldots, \sten{b}{i}{d}), \quad
 \Seg{b_1''} < \Seg{b_2''} < \cdots < \Seg{b_r''}, 
 \quad\text{and}\quad 
 \forall j_{k,i} \in \Var{C}: b_{j_{k,i},i}^{(k)} = 1
\end{align}
such that $\tensor{A}' = f(\vect{p}+\Delta\vect{p})$, as the last set of equations completely fixes the scaling indeterminacies. Note that the order of the rank-$1$ tensors is strict, for otherwise the rank of $\tensor{A}'$ would be strictly less than $r$. Since $\vect{p}$ is fixed, $\Delta\vect{p}$ is uniquely determined by the foregoing procedure.

Write $\Delta\tensor{A} = f(\vect{p} + \Delta\vect{p}) - f(\vect{p})$. Since $r$-identifiability holds in $\Var{N}$---but only up to the action of $\Var{T}$---it follows that 
\(
f^{\dagger}(\tensor{A}') = \Var{T} (\vect{p} + \Delta\vect{p}).
\)
Choose any
\begin{align}\label{eqn_main_proof_many_scale}
 \widehat{S} \in \Var{Z} = \arg\min_{S \in \Var{T}} \| \vect{p} - S(\vect{p} + \Delta\vect{p}) \|;
\end{align}
note that $\widehat{S}$ depends on $\tensor{A}'$, however I will not make such dependencies explicit in the notation for the sake of brevity. 
We can write \(\widehat{S}( \vect{p} + \Delta\vect{p} ) = \vect{p} + \widetilde{\Delta\vect{p}},\) so that for every $\tensor{A}' \in \Var{N}$ the following holds:
\[
 \tensor{A}' = f(\vect{p} + \Delta\vect{p}) = f(\vect{p} + \widetilde{\Delta\vect{p}}).
\]
Notice that 
\[
 \|\widetilde{\Delta\vect{p}}\| = d(\vect{p}, f^{\dagger}(\tensor{A}+\Delta\tensor{A})),
\]
so we can just analyze this norm. It is important to note that while $\widetilde{\Delta\vect{p}}$ depends on the choice of $\widehat{S}$, the norm of this vector \emph{is independent of} this choice.

In the remainder of the proof, most considered quantities, such as $\widetilde{\Delta\vect{p}}$, depend on $\Delta\tensor{A}$, which is arbitrary, and the choice of $\widehat{S}$ (which, in turn, depends on $\Delta\tensor{A}$). To be very precise, one could indicate these dependencies in the notation by writing, for example, $\widetilde{\Delta\vect{p}}( \Delta\tensor{A}, \widehat{S}(\Delta\tensor{A}) )$, but I will refrain from doing so wherever no confusion may arise for avoiding the obvious notational burden. 

\paragraph{\bf Part I: Continuity}
First, we prove that 
\begin{align} \label{eqn_main_proof_converges}
 \lim_{\substack{\|\Delta\tensor{A}\| \to 0,\\\tensor{A}+\Delta\tensor{A}\in\SecZ{r}{\Var{S}_\F}}} \| \widetilde{\Delta\vect{p}}(\Delta\tensor{A}) \| \to 0
\end{align}
by eliminating $r(d-1)$ parameters as follows. 
Let $\vect{x} = \operatorname{vecr}(\chi_1, \chi_2, \ldots, \chi_r) \in \F^{r(\Sigma+d)}$ be variables, where $\chi_i = (\sten{x}{i}{1}, \sten{x}{i}{2}, \ldots, \sten{x}{i}{d})$. We will represent the tensor $\tensor{X} = f(\vect{x})$ using the expected number of parameters $N$---provided that $\tensor{X} \in \Var{N}$ is close to $\tensor{A}$---by choosing vectorized factor matrices that lie in the subspace\footnote{This can be thought of as a normalized representation of the vectorized factor matrices.} of $\F^{r(\Sigma+d)}$ given by the system of $r(d-1)$ equations
\begin{align*} 
U_{\Var{C},\vect{p}} = 
 \bigl\{
 x_{j_{k,i},i}^{(k)} = a_{j_{k,i},i}^{(k)} \;\;|\;\; j_{k,i} \in \Var{C} 
 \bigr\},
\end{align*}
where $\Var{C}$ is as in \refeqn{eqn_main_proof_def_C}, and $x_{\ell,i}^{(k)}$ and $a_{\ell,i}^{(k)}$ denote the $\ell$th element of $\sten{x}{i}{k}$ and $\sten{a}{i}{k}$ respectively. The change of variables 
\[
\eta: \vect{x} \mapsto \vect{\dot{x}} = \operatorname{vecr}(\dot{\chi}_1,\dot{\chi}_2,\ldots,\dot{\chi}_r),
\] 
where $\dot{\chi}_i = (\sten{\dot{x}}{i}{1}, \sten{\dot{x}}{i}{2}, \ldots, \sten{\dot{x}}{i}{d})$, $\sten{\dot{x}}{i}{k} = \left[\begin{smallmatrix} \dot{x}_{1,i}^{(k)} & \cdots & \dot{x}_{n_k,i}^{(k)} \end{smallmatrix}\right]^T$, and
\begin{subequations} \label{eqn_main_proof_variables}
\begin{align}
\dot{x}_{\ell,i}^{(1)} &= x_{\ell,i}^{(1)} \prod_{k=2}^d \frac{x_{j_{k,i},i}^{(k)}}{a_{j_{k,i},i}^{(k)}},  &&\ell = 1,2,\ldots,n_1,\; i=1,2,\ldots,r, \\
\dot{x}_{\ell,i}^{(k)} &= a_{j_{k,i},i}^{(k)} \frac{ x_{\ell,i}^{(k)} }{ x_{j_{k,i},i}^{(k)} }, &&\ell=1,2,\ldots,n_k,\; i=1,2,\ldots,r,\; k=2,3,\ldots,d,
\end{align}
\end{subequations}
allows one to represent every tensor $f(\vect{x})$ with $x_{j_{k,i},i}^{(k)} \ne 0$, $j_{k,i}\in\Var{C}$ on $U_{\Var{C},\vect{p}}$ as  $\vect{\dot{x}}$. 
In particular, we may assume that all tensors $\tensor{A}' \in \Var{N}$ can be represented on $U_{\Var{C},\vect{p}}$. Otherwise, one should shrink the open neighborhood $\Var{N}$ until $\Var{N} \subset f(U_{\Var{C},\vect{p}})$ is satisfied; this is always possible due to the particular choice of $\Var{C}$ which ensures that $|a_{j_{k,i},i}^{(k)}| = \|\sten{a}{i}{k}\|_\infty > 0$ if $j_{k,i} \in \Var{C}$---the $\infty$-norm is nonzero because $\tensor{A}$ is robustly $r$-identifiable and hence of rank $r$.
Note that the definition of $\eta$ is such that it simply chooses an alternative representative for each of the rank-$1$ tensors $\chi_i$, so that 
\[
 \exists \overline{S} \in \Var{T}: \eta(\vect{x}) = \overline{S} \vect{x};
\]
note that $\overline{S}$ depends on $\vect{x}$, but I will not indicate this in the notation.
Define for $k=2,3,\ldots,d$ and $i=1,2,\ldots,r$, the functions
\begin{align*}
&\pi_{j_{k,i},\tensor{A}} : \F^{n_k} \to \F^{n_k-1},\; 
\vect{z} \mapsto 
\begin{bmatrix}
z_1 &
\cdots &
z_{j_{k,i}-1} &
z_{j_{k,i}+1} &
\cdots &
z_{n_k}
\end{bmatrix}^T, \\
&\pi_{j_{k,i},\tensor{A}}^{-1} : \F^{n_k-1} \to \F^{n_k},\;
\vect{z} \mapsto
 \begin{bmatrix}
z_{1} &
\cdots &
z_{j_{k,i}-1} &
a_{j_{k,i}}^{(k)} &
z_{j_{k,i}} &
\cdots &
z_{n_k-1}
\end{bmatrix}^T,
\end{align*}
where $j_{k,i} \in {\Var{C}}$. 
Define also
\begin{align*}
\pi_{\tensor{A}} : \F^{M} &\to \F^{N} 
&\pi_{\tensor{A}}^{-1} : \F^{N} &\to \F^{M} \\
\begin{bmatrix} 
\sten{z}{i}{1} \\ 
\sten{z}{i}{2} \\ 
\vdots \\ 
\sten{z}{i}{d} 
\end{bmatrix}_{i=1,\ldots,r} 
&\mapsto 
\begin{bmatrix} 
\sten{z}{i}{1} \\ 
\pi_{j_{2,i},\tensor{A}}( \sten{z}{i}{2} ) \\ 
\vdots \\ 
\pi_{j_{d,i},\tensor{A}}( \sten{z}{i}{d} ) 
\end{bmatrix}_{i=1,\ldots,r},
\text{ and }
&
\begin{bmatrix} 
\sten{z}{i}{1} \\ 
\sten{z}{i}{2} \\ 
\vdots \\ 
\sten{z}{i}{d} 
\end{bmatrix}_{i=1,\ldots,r} 
&\mapsto 
\begin{bmatrix} 
\sten{z}{i}{1} \\ 
\pi_{j_{2,i},\tensor{A}}^{-1}( \sten{z}{i}{2} ) \\ 
\vdots \\ 
\pi_{j_{d,i},\tensor{A}}^{-1}( \sten{z}{i}{d} ) 
\end{bmatrix}_{i=1,\ldots,r}.
\end{align*}
With the definition of these analytic functions in place, we define 
\begin{align*}
 f_{\tensor{A}} : \F^{N} &\to \F^{\Pi} \\ 
\vect{z} &\mapsto (f \circ \pi_{\tensor{A}}^{-1})( \vect{z} ).
\end{align*}
Now it follows that for an arbitrary choice of $\tensor{A}' = f(\vect{p} + \Delta\vect{p}) \in \Var{N}$ there exists an $\overline{S} \in \Var{T}$ such that 
\[
 \tensor{A}' = \tensor{A} + \Delta\tensor{A} = f(\vect{p}+\Delta\vect{p}) = f_\tensor{A}( \pi_\tensor{A}( \eta(\vect{p}+\Delta\vect{p}) ) ) = f_\tensor{A}\bigl( \pi_\tensor{A} \bigl( \overline{S} (\vect{p}+\Delta\vect{p}) \bigr) \bigr).
\] 
Let $\overline{\vect{p}} = \pi_\tensor{A}( \eta(\vect{p}) ) = \pi_\tensor{A}( \vect{p} )$ and $\overline{\Delta\vect{p}} = \pi_\tensor{A}(\overline{S}(\vect{p}+\Delta\vect{p})) - \overline{\vect{p}}$.
As ${f}_\tensor{A}$ is an analytic function over $\F$ in $N$ variables, it is continuously differentiable with Jacobian $\overline{J}$ at $\overline{\vect{p}}$.
Its rank is the maximum value $N$. This can be understood by observing 
that the Jacobian of $f \circ \pi_{\tensor{A}}^{-1}$ is---by the chain rule---given by
\[
 \overline{J} = J_{f \circ \pi_{\tensor{A}}^{-1}} (\overline{\vect{p}}) 
 = J_f( \pi_{\tensor{A}}^{-1}(\overline{\vect{p}}) ) J_{\pi_{\tensor{A}}^{-1}} (\overline{\vect{p}}) 
 = J_f( \vect{p} ) \widehat{I} = T_\vect{p} \widehat{I},
\]
where $\widehat{I} \in \R^{r(\Sigma+d) \times r(\Sigma+1)}$ is an identity matrix from which certain columns have been removed.
The implication is that the Jacobian $\overline{J}$ is formed by a subset of the columns of $f$'s Jacobian, i.e., Terracini's matrix $T_\vect{p}$. In fact, from the definition of $\pi_\tensor{A}^{-1}$ one verifies that $\overline{J}$ lacks the columns 
\begin{align}\label{eqn_main_proof_missing}
 \sten{a}{i}{1} \otimes \cdots \otimes \sten{a}{i}{k-1} \otimes \vect{e}_{j_{k,i}} \otimes \sten{a}{i}{k+1} \otimes \cdots \otimes \sten{a}{i}{d}, \quad j_{k,i} \in \Var{C},
\end{align}
relative to $T_\vect{p}$. I claim that these columns in \refeqn{eqn_main_proof_missing} are linearly dependent on the other columns of $T_\vect{p}$. Indeed, fix $k=2,3,\ldots,d$ and $i=1,2,\ldots,r$, and then 
\begin{align*}
T_\vect{p} (\vect{k}_i^{k} \otimes \vect{e}_i)
= 
\sum_{j=1}^{n_1} a_{j,i}^{(1)} \vect{e}_j \otimes \sten{a}{i}{2}\otimes\cdots\otimes\sten{a}{i}{d} - \sum_{j=1}^{n_k} a_{j,i}^{(k)} \sten{a}{i}{1} \otimes \cdots \otimes \sten{a}{i}{k-1} \otimes \vect{e}_{j} \otimes \sten{a}{i}{k+1} \otimes \cdots \otimes \sten{a}{i}{d} = 0.
\end{align*}
As $a_{j_{k,i}}^{(k)} \ne 0$, $j_{k,i}\in\Var{C}$, by construction of $\Var{C}$, it follows that \refeqn{eqn_main_proof_missing} is linearly dependent on the other columns in the above expression, and may hence be removed from $T_\vect{p}$ without affecting its range. By inspection it is clear that this procedure drops exactly $r(d-1)$ linearly independent columns of $T_\vect{p}$, resulting in $\overline{J}$ whose column span coincides with the column span of the original Terracini's matrix. As the latter had dimension $N$,\footnote{The condition number is defined to be $\infty$ if this rank condition is not satisfied.} it follows that $\overline{J} \in \F^{\Pi \times N}$ has linearly independent columns. 

I claim that $f_\tensor{A}$ is a local diffeomorphism. By our assumption on robust $r$-identifiability, $\Var{N}$ is an open neighborhood of a smooth manifold of dimension $N$. Hence, there exists a chart $(U,\phi^{-1})$ where $U \subset \Var{N}$ is an open neighborhood of $\tensor{A}$ and a local diffeomorphism $\phi^{-1}$ between $U$ and an open neighborhood $V = \phi^{-1}(U) \subset \F^N$. Restrict $\Var{N}$ to the open neighborhood $U \subset \Var{N}$. Let $\tensor{A} = \phi(\vect{q}) \in \Var{N}$. Since $\phi$ is a local diffeomorphism, $\phi^{-1} \circ \phi|_V = \operatorname{Id}_{\F^N}$. Hence, the Jacobian matrix at $\vect{q} \in V$ of the composition $\phi^{-1} \circ \phi|_V$, namely $J_{\phi^{-1}}(\phi(\vect{q})) J_\phi(\vect{q})$ by the chain rule, is the $N \times N$ identity matrix. 
Since $J_\phi(\vect{q})\in\F^{\Pi \times N}$ and $J_{\phi^{-1}}(\phi(\vect{q})) \in \F^{N \times \Pi}$ with $\Pi \ge N$, it follows that both matrices are of maximal rank $N$.
By definition, $J_\phi(\vect{q})$ is contained in the tangent space $\Tang{\tensor{A}}{\Sec{r}{\Var{S}_\F}}$. By our assumption on robust $r$-identifiability, $\cspan(\overline{J}) = \cspan(T_\vect{p}) = \Tang{\tensor{A}}{\Sec{r}{\Var{S}_\F}}$. Since $\dim \cspan\bigl( J_\phi(\vect{q}) \bigr) = N = \dim \cspan(\overline{J})$ it follows that $J_\phi(\vect{q}) Z = \overline{J}$ for some nonsingular matrix $Z \in \F^{N \times N}$. Hence, the Jacobian matrix of the composition $g : \phi^{-1} \circ f_\tensor{A}$ at $\overline{\vect{p}}$ is given by $J_{\phi^{-1}}(f_\tensor{A}(\overline{\vect{p}})) \overline{J} = J_{\phi^{-1}}(\phi(\vect{q})) J_\phi(\vect{q}) Z = Z$, where the first equality is because $\tensor{A} = f_\tensor{A}(\overline{\vect{p}}) = \phi(\vect{q})$. Now, the smooth function $g : \F^N \to \F^N$ has nonsingular Jacobian matrix $Z$, hence by the Inverse Function Theorem \cite[Theorem 1A.1]{DR2009} it has a smooth inverse function $g^{-1}$. Thus, $g$ is a local diffeomorphism between a neighborhood of $\vect{p} \in \F^N$ and a neighborhood of $\vect{q} \in \F^N$. Letting now $f_\tensor{A}^{-1}$ be the composition $g^{-1} \circ \phi^{-1}$, which is a composition of local diffeomorphisms, it follows that $f_\tensor{A}$ is a local diffeomorphism between an open neighborhood of $\vect{p} \in \F^N$ and an open neighborhood of $\tensor{A} \in \Var{N}$. Further restrict $\Var{N}$ to the neighborhood where this local diffeomorphism is defined. 

As $f_\tensor{A}^{-1}$ is locally smooth, it follows immediately that $\Delta\tensor{A} \to 0$ implies $\overline{\Delta\vect{p}} \to 0$. Both $\vect{p}$ and $\overline{S}(\vect{p} + \Delta\vect{p})$ live in $U_{\Var{C},\vect{p}}$, so that their values at all positions $j_{k,i} \in \Var{C}$ agree. As $\pi_\tensor{A}^{-1}(\vect{x})$ is analytic, it follows that $\overline{\Delta\vect{p}} \to 0$ implies $\overline{S}(\vect{p} + \Delta\vect{p}) - \vect{p} \to 0$ as well. Since 
\[
 \|\widetilde{\Delta\vect{p}}\| = \min_{S \in \Var{T}}\| \vect{p} - S(\vect{p} + \Delta\vect{p}) \| \le \| \vect{p} - \overline{S}(\vect{p} + \Delta\vect{p}) \|,
\]
equation \refeqn{eqn_main_proof_converges} follows.

\newcommand{\maxitout}{\max_{\Delta\tensor{A} \in \Var{G}_\epsilon}}
\paragraph{\bf Part II: Sandwiching}
For convenience, let us denote the set 
\[
\Var{G}_\epsilon = \{ \Delta\tensor{A} \in \F^\Pi\setminus\{0\} \;|\; \|\Delta\tensor{A}\|\le \epsilon \text{ and } \tensor{A}+\Delta\tensor{A}\in\SecZ{r}{\Var{S}_\F} \}.
\]
Then, the second part consists of demonstrating that 
\[
 \lim_{\epsilon\to0} \maxitout \frac{\|\widetilde{\Delta\vect{p}}(\Delta\tensor{A})\|}{\|\Delta\tensor{A}\|} = \lim_{\epsilon\to0} \maxitout \frac{\|T_\vect{p} \vect{x}(\Delta\tensor{A})\|}{\|\vect{x}(\Delta\tensor{A})\|} = \bigl(\varsigma_N(T_\vect{p})\bigr)^{-1}
\]
whereby $\vect{x}(\Delta\tensor{A}) \in \F^{r(\Sigma+d)}$. The key consists of exploiting the Iterated Scaling Lemma for transforming any $\widetilde{\Delta\vect{p}}$ into a new vector $\vect{x}(\Delta\tensor{A})$ that is orthogonal to $T_\vect{p}$'s kernel.

Since Terracini's matrix $T_\vect{p}$ has a kernel $\Var{K} = \cspan(K)$ of dimension exactly equal to $r(d-1)$, we may factorize $\widetilde{\Delta\vect{p}}$ as
\[
 \widetilde{\Delta\vect{p}} = \forcebold{\Delta} + \forcebold{\nabla}, \quad\text{where } \vect{\Delta} \in \Var{K}^\perp \text{ and } \forcebold{\nabla} \in \Var{K}.
\]
Notice that $\langle \forcebold{\nabla}, \vect{\Delta} \rangle = 0$, so that $\|\forcebold{\nabla}\| \le \|\widetilde{\Delta\vect{p}}\|$ and $\|\vect{\Delta}\| \le \|\widetilde{\Delta\vect{p}}\|$.
From \reflem{lem_isl} it follows that for sufficiently small $\|\Delta\tensor{A}\| > 0$, and, hence, for sufficiently small $\| \widetilde{\Delta\vect{p}} \| > 0$, we have that
\[
 \vect{p} + \forcebold{\nabla} = \dot{\vect{p}} + \vect{\Delta}' 
 \quad\text{with } 
 \dot{\vect{p}} \sim \vect{p},\; \vect{\Delta}' \in \Var{K}^\perp,\; \text{and}\; \|\vect{\Delta}'\| \le \|\forcebold{\nabla}\|.
\]
Let $\dot{\vect{\Delta}} = \vect{\Delta}' + \vect{\Delta} \in \Var{K}^\perp$.
As particular consequences, we have that
$f(\vect{p}) = f(\dot{\vect{p}})$ and 
$f(\vect{p} + \widetilde{\Delta\vect{p}}) = f(\dot{\vect{p}} + \dot{\vect{\Delta}})$.
The Taylor series expansion about the vectorized factor matrices $\dot{\vect{p}}$ is 
\begin{align*}
f(\vect{p} + \widetilde{\Delta\vect{p}})
= f(\dot{\vect{p}} + \dot{\vect{\Delta}}) 
= f(\dot{\vect{p}}) + T_{\dot{\vect{p}}} \dot{\vect{\Delta}} + \mathcal{O}(\| \dot{\vect{\Delta}} \|^2 ) 
= f(\vect{p}) + T_{\dot{\vect{p}}} \dot{\vect{\Delta}} + \mathcal{O}(\| \dot{\vect{\Delta}} \|^2 ) ;
\end{align*}
thus, there exist a constant $C > 0$ such that for all sufficiently small $\widetilde{\Delta\vect{p}}$ it holds that 
\[
\| T_{\dot{\vect{p}}} \dot{\vect{\Delta}} \| - C \| \dot{\vect{\Delta}} \|^2 \le \| f(\vect{p} + \widetilde{\Delta\vect{p}}) - f(\vect{p}) \| \le \| T_{\dot{\vect{p}}} \dot{\vect{\Delta}}\| + C \| \dot{\vect{\Delta}} \|^2.
\]
It follows from the Iterated Scaling Lemma that if $f(\vect{p} + \widetilde{\Delta\vect{p}}) \ne f(\vect{p})$ then $\|\dot{\vect{\Delta}}\| \ne 0$.\footnote{At this point where we invoke separability, the argument for the premetric $\widehat{d}(\cdot,\cdot)$ fails.} Indeed, the contrapositive $\|\dot{\vect{\Delta}}\| = 0$ yields
\(\dot{\vect{\Delta}} = \vect{p} + \widetilde{\Delta\vect{p}} - \dot{\vect{p}} = 0\), so that $f(\vect{p} + \widetilde{\Delta\vect{p}}) = f(\dot{\vect{p}}) = f(\vect{p}).$
Thus, if $\|\Delta\tensor{A}\| > 0$ then $\|\dot{\vect{\Delta}}\| > 0$, so that we may write
\begin{align*} 
-C \| \dot{\vect{\Delta}} \| \le \frac{\| f(\vect{p}+\widetilde{\Delta\vect{p}}) - f(\vect{p}) \|}{ \| \dot{\vect{\Delta}} \| } - \frac{ \| T_{\dot{\vect{p}}} \dot{\vect{\Delta}} \| }{ \| \dot{\vect{\Delta}} \| } \le C \| \dot{\vect{\Delta}} \|.
\end{align*}
Hence, for sufficiently small $\epsilon$ and $\|\dot{\vect{\Delta}}\|$, we obtain the useful bound
\begin{align} \label{eqn_main_proof_seq1}
\Biggl| \frac{\|\Delta\tensor{A}\|}{\|\dot{\vect{\Delta}}\|} - \frac{\|T_{\dot{\vect{p}}} \dot{\vect{\Delta}}\|}{\|\dot{\vect{\Delta}}\|} \Biggr| 
\le C \cdot \|\dot{\vect{\Delta}}\| 
\le C \cdot \max_{\Delta\tensor{A} \in \Var{G}_\epsilon} \max_{\widehat{S}(\Delta\tensor{A})\in\Var{Z}} \|\dot{\vect{\Delta}}(\Delta\tensor{A},\widehat{S}(\Delta\tensor{A}))\|.
\end{align}
Notice that all of the dotted vectors ultimately depend on $\Delta\tensor{A}$.
Consider now
\[
 \frac{\| T_\vect{p} \dot{\vect{\Delta}}\|}{\|\dot{\vect{\Delta}}\|} - \frac{\|T_\vect{p} \dot{E} \dot{\vect{\Delta}}\|}{\|\dot{\vect{\Delta}}\|}
 \le \frac{\|T_{\dot{\vect{p}}} \dot{\vect{\Delta}}\|}{\|\dot{\vect{\Delta}}\|} = \frac{\| T_\vect{p} \dot{\vect{\Delta}} + T_\vect{p} \dot{E} \dot{\vect{\Delta}}\|}{\|\dot{\vect{\Delta}}\|}
 \le \frac{\| T_\vect{p} \dot{\vect{\Delta}}\|}{\|\dot{\vect{\Delta}}\|} + \frac{\|T_\vect{p} \dot{E} \dot{\vect{\Delta}}\|}{\|\dot{\vect{\Delta}}\|},
\]
where $\dot{E}$ is as in \refcor{cor_isl_terracini}. Then, we similarly have for sufficiently small $\epsilon$ that
\begin{align}\label{eqn_main_proof_seq2}
\Biggl| \frac{\|T_{\dot{\vect{p}}} \dot{\vect{\Delta}}\|}{\|\dot{\vect{\Delta}}\|} - \frac{\| T_\vect{p} \dot{\vect{\Delta}}\|}{\|\dot{\vect{\Delta}}\|} \Biggr|
\le \frac{\|T_\vect{p} \dot{E} \dot{\vect{\Delta}}\|}{\|\dot{\vect{\Delta}}\|} 
\le \|T_\vect{p}\|_2 \|\dot{E}\|_2
 \le \|T_\vect{p}\|_2 \max_{\Delta\tensor{A} \in \Var{G}_\epsilon} \max_{\widehat{S}(\Delta\tensor{A})\in\Var{Z}} \|\dot{E}(\Delta\tensor{A},\widehat{S}(\Delta\tensor{A}))\|_2.
\end{align}
Combining \refeqn{eqn_main_proof_seq1} and \refeqn{eqn_main_proof_seq2} allows us to infer that
\begin{subequations}\label{eqn_main_proof_seq3}
\begin{align}
 \frac{\|\Delta\tensor{A}\|}{\|\dot{\vect{\Delta}}\|} = \frac{\| T_\vect{p} \dot{\vect{\Delta}}\|}{\|\dot{\vect{\Delta}}\|} + \varepsilon,
\end{align}
where
\begin{align}
|\varepsilon| \le \vartheta = \bigl( C \cdot \max_{\Delta\tensor{A}\in\Var{G}_\epsilon} \max_{\widehat{S}(\Delta\tensor{A})\in\Var{Z}} \|\dot{\vect{\Delta}}(\Delta\tensor{A},\widehat{S}(\Delta\tensor{A}))\| \bigr) + \bigl( \|T_\vect{p}\|_2  \cdot \max_{\Delta\tensor{A}\in\Var{G}_\epsilon} \max_{\widehat{S}(\Delta\tensor{A})\in\Var{Z}} \|\dot{E}(\Delta\tensor{A}, \widehat{S}(\Delta\tensor{A}))\|_2 \bigr). 
\end{align}
\end{subequations}
It remains to relate this expression to $\|\widetilde{\Delta\vect{p}}\|/\|\Delta\tensor{A}\|$. 

An upper bound on $\|\widetilde{\Delta\vect{p}}\|$ in terms of $\dot{\vect{\Delta}}$ is obtained by observing that 
\[
 \widehat{S}(\vect{p} + \Delta\vect{p}) = \vect{p} + \widetilde{\Delta\vect{p}} = \dot{\vect{p}} + \dot{\vect{\Delta}} = S \vect{p} + \dot{\vect{\Delta}}
\]
for some invertible $S \in \Var{T}$. It follows that
\[
\|\widetilde{\Delta\vect{p}}\| = \min_{D\in\Var{T}} \| D(\vect{p}+\Delta\vect{p}) - \vect{p} \| \le \|S^{-1} \widehat{S} (\vect{p} + \Delta\vect{p}) - \vect{p}\| = \|S^{-1} \dot{\vect{\Delta}}\| \le \|S^{-1}\|_2 \|\dot{\vect{\Delta}}\|,
\]
where the first inequality arises because $\Var{T}$ is a multiplicative group. 
As $S^{-1}\in\Var{T}$ so $S\in\Var{T}$, we can write $S = PD$ where $P\in\Var{P}$ and $D\in\Var{B}$, so that $\|S^{-1}\|_2 = \|D^{-1}\|_2$.
It follows from the Iterated Scaling Lemma that the diagonal entries of $D$ are the $\gamma_{j,i}$'s in \refeqn{eqn_isl_cor_diagonal}. Then, it follows immediately from \refeqn{eqn_isl_cor_gam1} and \refeqn{eqn_isl_cor_gam2}
that the largest diagonal entry in absolute value of $D^{-1}$ is smaller than $1 + 4(d-1) \|\widetilde{\Delta\vect{p}}\| \cdot \max_i \chi_i = 1 + \delta$, where $\chi_i$ is as in \refeqn{eqn_isl_proof_chi}.
It follows from Part I that $\delta\to0$ as $\epsilon\to0$. 
A lower bound on $\|\widetilde{\Delta\vect{p}}\|$ in terms of $\dot{\vect{\Delta}}$ is obtained as follows. From the triangle inequality, $\|\dot{\vect{\Delta}}\| \le \|\vect{\Delta}'\| + \|\vect{\Delta}\|$ so that
\begin{align*}
 \|\dot{\vect{\Delta}}\|^2 
 &\le \|\vect{\Delta}'\|^2 + 2 \|\vect{\Delta}'\| \|\vect{\Delta}\| + \|\vect{\Delta}\|^2 \\
 &= \|\vect{\Delta}'\| (\|\vect{\Delta}'\| + 2 \|\vect{\Delta}\|) + \|\vect{\Delta}\|^2
 \le \|\forcebold{\nabla}\|^2 ( 4 \lambda^2 \|\forcebold{\nabla}\|^2 + 4 \lambda \|\vect{\Delta}\| ) + \|\vect{\Delta}\|^2,
\end{align*}
where the last step is due to \reflem{lem_isl}. Note that $\lambda$ is a constant.
Provided that $\|\Delta\tensor{A}\|$ and, hence, $\|\widetilde{\Delta\vect{p}}\|$, are sufficiently small, i.e., so that the inequality
\[
4 \lambda^2 \|\forcebold{\nabla}\|^2 + 4 \lambda \|\vect{\Delta}\| \le 4 \lambda^2 \|\widetilde{\Delta\vect{p}}\|^2 + 4 \lambda \|\widetilde{\Delta\vect{p}}\| \le 1
\]
is satisfied, 
then the foregoing implies the lower bound 
\[
 \|\dot{\vect{\Delta}}\|^2 \le \|\forcebold{\nabla}\|^2 + \|\vect{\Delta}\|^2 = \|\forcebold{\nabla}\|^2 + 2 \cdot \mathrm{Re} \langle \forcebold{\nabla}, \vect{\Delta} \rangle  + \|\vect{\Delta}\|^2 = \|\widetilde{\Delta\vect{p}}\|^2,
\]
where the first equality is because of the orthogonality in the Hermitian inner product of $\forcebold{\nabla}$ and $\vect{\Delta}$.
Thus, by combining the lower and upper bound and dividing by $\|\Delta\tensor{A}\|$, we can conclude that for every nonzero $\Delta\tensor{A}$ of sufficiently small norm the following relations hold true
\begin{align}\label{eqn_main_proof_deltap_ineq}
 \frac{\|\dot{\vect{\Delta}}(\Delta\tensor{A}, \widehat{S}(\Delta\tensor{A}))\|}{\|\Delta\tensor{A}\|} 
 \le \frac{\|\widetilde{\Delta\vect{p}}(\Delta\tensor{A})\|}{\|\Delta\tensor{A}\|}
 \le (1 + \delta(\Delta\tensor{A})) \frac{\|\dot{\vect{\Delta}}(\Delta\tensor{A},\widehat{S}(\Delta\tensor{A}))\|}{\|\Delta\tensor{A}\|},
\end{align}
provided that $\epsilon \ge \|\Delta\tensor{A}\|$ is small enough and where I explicitly indicated the dependence on $\Delta\tensor{A}$ and $\widehat{S}$, which in turn depends on the former. 
Let
\begin{align*}
 \Delta\tensor{A}' \in \underset{\Delta\tensor{A} \in \Var{G}_\epsilon}{\arg\max} \; \max_{\widehat{S}(\Delta\tensor{A})\in\Var{Z}} \frac{\|\dot{\vect{\Delta}}(\Delta\tensor{A}, \widehat{S}(\Delta\tensor{A}))\|}{\|\Delta\tensor{A}\|}
 \quad\text{and}\quad
 \Delta\tensor{A}'' \in \underset{\Delta\tensor{A} \in \Var{G}_\epsilon}{\arg\max} \frac{\|\widetilde{\Delta\vect{p}}(\Delta\tensor{A})\|}{\|\Delta\tensor{A}\|},
\end{align*}
where $\Var{Z}$ is as in \refeqn{eqn_main_proof_many_scale}. Also let 
\begin{align*}
 \widehat{S}' \in \underset{\widehat{S}\in\Var{Z}}{\arg\max} \frac{\|\dot{\vect{\Delta}}(\Delta\tensor{A}', \widehat{S})\|}{\|\Delta\tensor{A}'\|}.
\end{align*}
It follows that
\begin{align}\label{eqn_main_proof_sandwich}
\frac{\|\dot{\vect{\Delta}}(\Delta\tensor{A}',\widehat{S}')\|}{\|\Delta\tensor{A}'\|} 
\le \frac{\|\widetilde{\Delta\vect{p}}(\Delta\tensor{A}')\|}{\|\Delta\tensor{A}'\|} 
\le \frac{\|\widetilde{\Delta\vect{p}}(\Delta\tensor{A}'')\|}{\|\Delta\tensor{A}''\|} 
&\le (1 + \delta(\Delta\tensor{A}'')) \frac{\|\dot{\vect{\Delta}}(\Delta\tensor{A}'', \bullet)\|}{\|\Delta\tensor{A}''\|} \\
&\le \frac{\|\dot{\vect{\Delta}}(\Delta\tensor{A}',\widehat{S}')\|}{\|\Delta\tensor{A}'\|} \cdot \maxitout (1 + \delta(\Delta\tensor{A}) ), \nonumber
\end{align}
where the first step is by the first inequality in \refeqn{eqn_main_proof_deltap_ineq}, the second inequality is by optimality of $\Delta\tensor{A}''$, the third inequality is because of the second inequality in \refeqn{eqn_main_proof_deltap_ineq} with the ``$\bullet$'' indicating that the inequality is valid for all $S(\Delta\tensor{A}'')\in\Var{Z}$ (so in particular the maximum), and the last step is because of the optimality of $\Delta\tensor{A}'$.
For continuing, I will drop the explicit notation indicating the dependence on $\Delta\tensor{A}$ and $\widehat{S}$ again and we will consider the sequences
\[
c_n = \min_{\Delta\tensor{A}_n \in \Var{G}_{1/n}} \min_{\widehat{S}(\Delta\tensor{A}_n)\in\Var{Z}} \frac{\|\Delta\tensor{A}_n\|}{\|\dot{\vect{\Delta}}_n\|} = \frac{\|\Delta\tensor{A}_n'\|}{\|\dot{\vect{\Delta}}_n\|}  \quad\text{and}\quad
s_n = \min_{\Delta\tensor{A}_n \in \Var{G}_{1/n}} \min_{\widehat{S}(\Delta\tensor{A}_n)\in\Var{Z}} \frac{\|T_\vect{p} \dot{\vect{\Delta}}_n\|}{\|\dot{\vect{\Delta}}_n\|},
\]
where $n$ is sufficiently large, 
and then it follows from the first part of \refeqn{eqn_main_proof_seq3} that 
\[
 c_n \ge s_n + \min_{\Delta\tensor{A}_n \in \Var{G}_{1/n}} \min_{\widehat{S}(\Delta\tensor{A}_n)\in\Var{Z}} \varepsilon_n 
\quad\text{and}\quad
 s_n \ge c_n + \min_{\Delta\tensor{A}_n \in \Var{G}_{1/n}} \min_{\widehat{S}(\Delta\tensor{A}_n)\in\Var{Z}} (-\varepsilon_n);
 \]
note that I used the subscript $n$ for indicating the dependence of $\dot{\vect{\Delta}}$ on $\Delta\tensor{A}_n$. From the second part of \refeqn{eqn_main_proof_seq3}, it follows that we can bound 
\begin{align}\label{eqn_main_proof_cinvbound}
s_n - \vartheta_n \le c_n \le s_n + \vartheta_n,
\;\text{so}\quad
(s_n - \vartheta_n)^{-1} \ge c_n^{-1} \ge (s_n + \vartheta_n)^{-1}, 
\end{align}
for sufficiently large $n$ so that $s_n > \vartheta_n$. Note that $\vartheta_n\to0$ as $n\to\infty$ because of \refeqn{eqn_main_proof_seq3}, \refeqn{eqn_main_proof_converges} and \refcor{cor_isl_terracini}.
Since $\dot{\vect{\Delta}}_n \in \Var{K}^\perp$, it follows immediately from the Courant--Fisher minimax characterization of the least singular values \cite[Theorem 4.2.11]{MatrixAnalysis} that 
\[
0 < \varsigma_N(T_\vect{p}) \le \frac{\|T_\vect{p} \dot{\vect{\Delta}}_n\|}{\|\dot{\vect{\Delta}}_n\|} \le \varsigma_{1}( T_\vect{p} ),
\]
where $N = r(\Sigma+1)$, $\varsigma_i(T_\vect{p})$ denotes the $i$th singular value of $T_\vect{p}$, and the first inequality is due to the assumption on the rank of $T_\vect{p}$. I claim that the sequence of $s_n$ converges:
\begin{align}\label{eqn_main_proof_seqconv}
 \lim_{n\to\infty} s_n =
 \lim_{\epsilon\to0} \min_{\Delta\tensor{A}\in\Var{G}_\epsilon} \min_{\widehat{S}(\Delta\tensor{A})\in\Var{Z}} \frac{\|T_\vect{p} \dot{\vect{\Delta}}(\Delta\tensor{A}, \widehat{S}(\Delta\tensor{A}))\|}{\|\dot{\vect{\Delta}}(\Delta\tensor{A}, \widehat{S}(\Delta\tensor{A})\|} = \varsigma_N(T_\vect{p}).
\end{align}
This statement would conclude the proof because then all of the following limits are well-defined:
\[
 \lim_{n\to\infty} (s_n - \vartheta_n)^{-1} = \bigl(\varsigma_N(T_\vect{p})\bigr)^{-1} \quad\text{and}\quad \lim_{n\to\infty} (s_n + \vartheta_n)^{-1} = \bigl(\varsigma_N(T_\vect{p})\bigr)^{-1},
\]
sandwiching $c^{-1}_n$ in \refeqn{eqn_main_proof_cinvbound}, so that by considering \refeqn{eqn_main_proof_sandwich} we find
\begin{align*}
\lim_{n\to\infty} c_n^{-1}
\le \lim_{n\to\infty}\max_{\Delta\tensor{A}_n \in \Var{G}_{1/n}} \frac{\|\widetilde{\Delta\vect{p}}_n\|}{\|\Delta\tensor{A}_n\|}
\le \lim_{n\to\infty} c_n^{-1} \cdot \lim_{n\to\infty} \max_{\Delta\tensor{A}_n \in \Var{G}_{1/n}} (1+\delta_n),
\end{align*}
and as $\delta_n \to 0$, it follows that
\[
 \kappa_A 
 = \lim_{\epsilon\to0} \max_{\Delta\tensor{A}\in\Var{G}_\epsilon} \frac{ d(\vect{p}, f^\dagger(\tensor{A}+\Delta\tensor{A})) }{\|\Delta\tensor{A}\|} 
 = \lim_{\epsilon\to0} \max_{\Delta\tensor{A}\in\Var{G}_\epsilon}  \frac{\|\widetilde{\Delta\vect{p}}_n\|}{\|\Delta\tensor{A}_n\|} 
 = (\varsigma_N(T_\vect{p}))^{-1},
\]
which would conclude the proof.

\paragraph{\bf Part III: The limit}
For proving \refeqn{eqn_main_proof_seqconv}, we consider the specific perturbation $\Delta\vect{p}_n' = n^{-1} \vect{w}$, where $\vect{w}$ is the right singular vector of $T_\vect{p}$ corresponding to the singular value $\varsigma_N(T_\vect{p})$. The perturbed tensor is $\tensor{A}_n = f(\vect{p} + \Delta\vect{p}_n') = f(\vect{p} + n^{-1} \vect{w}) = \tensor{A} + \Delta\tensor{A}_n$.
If $n_0$ is sufficiently large, then $\tensor{A}_n \subset \Var{N}$ for all $n > n_0$. 
Clearly, $\|\Delta\tensor{A}_n\|\to0$ as $n \to \infty$. From the Taylor series expansion at the start of the proof, it is also clear that for large $n$, $\Delta\tensor{A}_n \ne 0$, as the norms of the higher-order terms will be dominated by $n^{-1} = \|n^{-1} \vect{w}\|$. Let $\Delta\vect{p}_n$ be the specific normalization of $\vect{p}+\Delta\vect{p}_n \in f^\dagger(\tensor{A}+\Delta\tensor{A}_n)$ that satisfies the properties imposed in \refeqn{eqn_main_proof_normalization}. 
Let $\widehat{S}_n \in \arg\min_{S\in\Var{T}} \|\vect{p} - S(\vect{p} + \Delta\vect{p}_n)\|$ be arbitrary. Then, $\vect{p} + \widetilde{\Delta\vect{p}}_n = \widehat{S}_n(\vect{p} + \Delta\vect{p}_n)$. By definition of $\Delta\vect{p}_n$ there exists some $T_n \in \Var{T}$ such that $\vect{p} + \Delta\vect{p}_n = T_n(\vect{p} + \Delta\vect{p}_n')$. Hence,
\(
 \vect{p} + \widetilde{\Delta\vect{p}}_n = \widehat{S}_n T_n (\vect{p} + n^{-1} \vect{w}), 
\)
so that
\begin{align}\label{eqn_main_proof_tildedeltap}
\|\widetilde{\Delta\vect{p}}_n\|  
= \min_{S \in \Var{T}} \| \vect{p} - S (\vect{p} + \Delta\vect{p}_n) \| 
&= \min_{S \in \Var{T}} \| \vect{p} - S T_n (\vect{p} + n^{-1}\vect{w}) \| \\
&= \min_{S \in \Var{T}} \| \vect{p} - S (\vect{p} + n^{-1}\vect{w})\| 
\le \| \vect{p} - I(\vect{p} + n^{-1} \vect{w}) \| 
= n^{-1}; \nonumber
\end{align}
the third equality is because $\Var{T}$ is a group.

Let $n$ be sufficiently large. Then, I claim that $\widehat{S}_n T_n = I + E_n$ where $E_n$ is a diagonal matrix whose diagonal entries are bounded by $\mathcal{O}(n^{-1})$. 
Indeed, the proof of \refprop{prop_distance_measure} shows that \refeqn{eqn_main_proof_tildedeltap} may only be true if for all $k =2, \ldots, d$ and all $i=1,\ldots,r$ simultaneously we have that
\begin{align}\label{eqn_main_proof_somebound}
 \|\sten{a}{i}{k} - \theta_{k,\pi_i} (\sten{a}{\pi_i}{k} + n^{-1}\sten{w}{\pi_i}{k})\| \le n^{-1}, \quad \theta_{k,\pi_i} \in \F\setminus\{0\},
\end{align}
where $\vect{w} = \operatorname{vecr}(\omega_1,\ldots,\omega_r)$, $\omega_i = (\sten{w}{i}{1},\ldots,\sten{w}{i}{d})$, and $\pi$ is a permutation of $\{1,2,\ldots,r\}$. Since $\sten{a}{i}{k}$ are constant vectors, it follows that this inequality can only be satisfied if
\[
 \sten{a}{i}{k} = \theta_{k,\pi_i} \sten{a}{\pi_i}{k} \quad\text{for all } k = 2, 3, \ldots, d \text{ and } i = 1, 2, \ldots, r.
\]
Suppose that $\pi$ is not the identity permutation, then there exists a $j \ne \pi_j$ for which it holds that
\begin{align*}
 \sten{a}{j}{1} \otimes \sten{a}{j}{2} \otimes \cdots \otimes \sten{a}{j}{d} + \sten{a}{\pi_j}{1} \otimes \sten{a}{\pi_j}{2} \otimes \cdots \otimes \sten{a}{\pi_j}{d} 
&= \sten{a}{j}{1} \otimes \sten{a}{j}{2} \otimes \cdots \otimes \sten{a}{j}{d} +  \sten{a}{\pi_j}{1} \otimes (\theta_{2,\pi_j}^{-1}\sten{a}{j}{2}) \otimes \cdots \otimes (\theta_{d,\pi_j}^{-1}\sten{a}{j}{d})\\
&= \bigl( \sten{a}{j}{1} + \theta_{2,\pi_j}^{-1}\cdots\theta_{d,\pi_j}^{-1}\sten{a}{\pi_j}{1} \bigr) \otimes \sten{a}{j}{2} \otimes \cdots \otimes \sten{a}{j}{d},
\end{align*}
showing that $\tensor{A}$ really has a decomposition of length at most $r-1$, which is a contradiction. So $\pi$ is the identity; hence \refeqn{eqn_main_proof_somebound} may be simplified to
\begin{align}\label{eqn_main_proof_theta_ineq}
\|(1-\theta_{k,i})\sten{a}{i}{k} - n^{-1} \theta_{k,i} \sten{w}{i}{k}\|^2 \le n^{-2},
\end{align}
which should hold for all $i=1,2,\ldots,r$ and $k=2,3,\ldots,d$.
It can be satisfied only if
\begin{align}\label{eqn_main_proof_theta_bound}
1 - 3 n^{-1} \|\sten{a}{i}{k}\|^{-1} \le \theta_{k,i} \le 1 + 3 n^{-1} \|\sten{a}{i}{k}\|^{-1}.
\end{align}
Note that \refeqn{eqn_main_proof_theta_ineq} is a quadratic equation in $\theta_{k,i}$ with positive coefficient for $\theta_{k,i}^2$. 
So it suffices proving that inequality \refeqn{eqn_main_proof_theta_ineq} is satisfied for $\theta_{k,i}=1$ while it is no longer satisfied for $\theta_{k,i}= 1 \pm 3 n^{-1} \|\sten{a}{i}{k}\|^{-1}$ to conclude that any $\theta_{k,i}$ satisfying \refeqn{eqn_main_proof_theta_ineq} must be contained in the interval \refeqn{eqn_main_proof_theta_bound}. For $\theta_{k,i}=1$ one has $\|(1-1)\sten{a}{i}{k} - n^{-1} \sten{w}{i}{k}\| = n^{-1} \|\sten{w}{i}{k}\| \le n^{-1}$, so \refeqn{eqn_main_proof_theta_ineq} is satisfied. Pugging $\theta_{k,i} = 1\pm 3 n^{-1} \|\sten{a}{i}{k}\|^{-1}$ into \refeqn{eqn_main_proof_theta_ineq}, we get from the triangle inequality
\begin{align*}
\|(1-\theta_{k,i})\sten{a}{i}{k} - n^{-1} \theta_{k,i} \sten{w}{i}{k}\|
&\ge 3 n^{-1} \|\sten{a}{i}{k}\|^{-1} \|\sten{a}{i}{k}\| - n^{-1} |1 \pm 3 n^{-1} \|\sten{a}{i}{k}\|^{-1}| \cdot \|\sten{w}{i}{k}\| \\
&\ge n^{-1} (3 - |1 \pm 3n^{-1}\|\sten{a}{i}{k}\|^{-1}|)
= n^{-1} (2 \mp 3 n^{-1}\|\sten{a}{i}{k}\|^{-1}) > n^{-1}
\end{align*}
where the last equality and inequality are valid only for large $n$, namely $n > 3 \|\sten{a}{i}{k}\|^{-1}$. 
Hence, $|\theta_{k,i} -1| \le 3 n^{-1} \|\sten{a}{i}{k}\|^{-1}$ for $k=2,\ldots,d$. For $\theta_{1,i} = (\theta_{2,i}\cdots\theta_{d,i})^{-1}$ we get the usual bound
\begin{align*}
|\theta_{1,i}-1| 
&= \Biggl|-1 + \prod_{k=2}^d (1 \pm 3n^{-1}\|\sten{a}{i}{k}\|^{-1})^{-1}\Biggr| 
\le -1 + 1 + \sum_{\kappa=1}^\infty \sum_{\|\forcebold{\ell}\|_1=\kappa} \prod_{k=2}^d \Bigl| 3 n^{-1} \|\sten{a}{i}{k}\|^{-1} \Bigr|^{\ell_k} \\
&\le \sum_{\kappa=1}^\infty (3n^{-1}\chi_i)^\kappa \|\vect{1}\otimes\cdots\otimes\vect{1}\|_1
\le \sum_{\kappa=1}^{\infty} (3(d-1)n^{-1}\chi_i)^\kappa
\le 6(d-1)n^{-1}\chi_i,
\end{align*}
where in the third step the same argument surrounding the derivation of \refeqn{eqn_isl_proof_sigpi1} was employed, and where in the last step we assumed $n \ge 6(d-1) \chi_i$.
Letting $\mu_{k,i}$ be defined as
\begin{subequations}\label{eqn_main_proof_mu}
\begin{align}
\theta_{k,i} &= 1 - \mu_{k,i}, &&k=2,3,\ldots,d, \; i=1,2,\ldots,r \\
\theta_{1,i} &= \theta_{2,i}^{-1} \cdots \theta_{d,i}^{-1} = 1 - \mu_{1,i} &&i=1,2,\ldots,r
\end{align}
\end{subequations}
we may write $E_n$ explicitly as
\[
 E_n = \diag(E_{n,1}, \ldots, E_{n,r}) 
 \quad\text{with }
 E_{n_i} = \diag( -\mu_{1,i} I_{n_1} , -\mu_{2,i} I_{n_2}, \ldots, -\mu_{d,i} I_{n_d}).
\]
By the foregoing derivation $|\mu_{k,i}| = \mathcal{O}(n^{-1})$.

We now consider the following factorization
\begin{align} \label{eqn_main_proof_lastfactor}
 \vect{p} + \widetilde{\Delta\vect{p}}_n 
 = \widehat{S}_n (\vect{p} + \Delta\vect{p}_n) 
 = \widehat{S}_n T_n (\vect{p} + \Delta\vect{p}_n') 
 &= (I + {E}_n) (\vect{p} + \Delta\vect{p}_n') \\
 &= \vect{p} + \Delta\vect{p}_n' + E_n \vect{p} + E_n \Delta\vect{p}_n' \nonumber\\
 &= \vect{p} + (\forcebold{\nabla}_{n,1} + \forcebold{\nabla}_{n,2}) + (\vect{\Delta}_{n,1} + \vect{\Delta}_{n,2} + \Delta\vect{p}_n' ) \nonumber
\end{align}
where $E_n \vect{p} = \forcebold{\nabla}_{n,1} + \vect{\Delta}_{n,1}$ and $E_n\Delta\vect{p}_n' = \forcebold{\nabla}_{n,2} + \vect{\Delta}_{n,2}$ with $\forcebold{\nabla}_{n,i} \in \Var{K}$ and $\vect{\Delta}_{n,i} \in \Var{K}^\perp$. One sees immediately that $\|\forcebold{\nabla}_{n,1}\| \le \|E_n\|_2 \|\vect{p}\| = \mathcal{O}(n^{-1})$, $\|\forcebold{\nabla}_{n,2}\| \le \|E_n\|_2 \|\Delta\vect{p}_n'\| = \mathcal{O}(n^{-2})$, and $\|\forcebold{\Delta}_{n,2}\| \le \|E_n\|_2 \|\Delta\vect{p}_n'\| = \mathcal{O}(n^{-2})$.
The key difficulty lies in proving that $\|\vect{\Delta}_{n,1}\| = \mathcal{O}(n^{-2})$. By definition,
\[
 \vect{\Delta}_{n,1} = (I - K K^\dagger) E_n\vect{p} = (I - K \diag(K_1^H K_1,\ldots,K_r^H K_r)^{-1} K^H) E_n \vect{p}
\]
so it suffices to prove that $\|(I - K_i K_i^\dagger) \vect{q}_i\| = \mathcal{O}(n^{-2})$, where 
\[
 \vect{q}_i = \vect{q}^{(1)}_i = 
 \begin{bmatrix}
  -\mu_{1,i} \sten{a}{i}{1} \\
  -\mu_{2,i} \sten{a}{i}{2} \\
  \vdots \\
  -\mu_{d,i} \sten{a}{i}{d}
 \end{bmatrix}
\]
with $\mu_{k,i}$ as in \refeqn{eqn_main_proof_mu}. By definition and some Maclaurin series, we know that
\begin{align}\label{eqn_main_proof_gamma}
 -\mu_{1,i} = -1 + \prod_{k=2}^d (1-\mu_{k,i})^{-1} = \sum_{k=2}^d \mu_{k,i} + \sum_{\kappa=2}\sum_{\|\ell\|_1=\kappa} \prod_{k=2}^d \mu_{k,i}^{\ell_k} = \sum_{k=2}^d \mu_{k,i} + M_i;
\end{align}
it can be shown in the usual way that $|M_i| = \mathcal{O}(n^{-2})$.
Let $\alpha_{k,i} = \|\sten{a}{i}{k}\|$.
Then, we find \(\vect{q}^{(2)}_i = K_i^H \vect{q}^{(1)}_i.\) The entries of $\vect{q}_i^{(2)}$ are
\[
q^{(2)}_{k-1,i} = -\mu_{1,i}\alpha_{1,i}^2 + \mu_{k,i} \alpha_{k,i}^2, \quad k=2,\ldots,d.
\]
For the next step $\vect{q}^{(3)}_i = (K_i^H K_i)^{-1} \vect{q}^{(2)}_i$, we recall from the Sherman--Morrison formula that
\[
 (K_i^H K_i)^{-1} = (\diag(\alpha_{2,i}^2,\ldots,\alpha_{d,i}^2) + \alpha_{1,i}^2 \vect{1}\vect{1}^T)^{-1} 
 = \diag(\alpha_{2,i}^{-2},\ldots,\alpha_{d,i}^{-2}) - \frac{\alpha_{1,i}^2}{1 + \alpha_{1,i}^2 \sum_{j=2}^d \alpha_{j,i}^{-2}} \vect{h}\vect{h}^T,
\]
where $\vect{h}^T_i = \left[\begin{smallmatrix}\alpha_{2,i}^{-2} & \cdots & \alpha_{d,i}^{-2}\end{smallmatrix}\right]$. Let $\widehat{\alpha}_i = \alpha_{1,i}^2 \sum_{j=2}^d \alpha_{j,i}^{-2}$. It is then straightforward to establish that
\[
 \vect{h}^T_i \vect{q}^{(2)}_i = - \mu_{1,i} \widehat{\alpha}_i + \sum_{k=2}^d \mu_{k,i},
\]
so that the elements of $\sten{q}{i}{(3)}$ are given by
\begin{align*}
 q_{k-1,i}^{(3)} 
 &= -\mu_{1,i}\alpha_{1,i}^2 \alpha_{k,i}^{-2} + \mu_{k,i} - \frac{\alpha_{1,i}^{2} \alpha_{k,i}^{-2} ( -\mu_{1,i}\widehat{\alpha}_i + \sum_{k=2}^d \mu_{k,i}) }{1 + \widehat{\alpha}_i}\\
 &= -\mu_{1,i}\alpha_{1,i}^2 \alpha_{k,i}^{-2} + \mu_{k,i} - \frac{\alpha_{1,i}^{2} \alpha_{k,i}^{-2} ( 1+\widehat{\alpha}_i) \sum_{k=2}^d \mu_{k,i} + \alpha_{1,i}^{2} \alpha_{k,i}^{-2} \widehat{\alpha}_i M_i }{1 + \widehat{\alpha}_i} \\
 &= \mu_{k,i} + \alpha_{1,i}^2 \alpha_{k,i}^{-2} M_i - \alpha_{1,i}^2 \alpha_{k,i}^{-2} \widehat{\alpha}_i (1+\widehat{\alpha}_i)^{-1} M_i \\
 &= \mu_{k,i} + M_i \alpha_{1,i}^2 \alpha_{k,i}^{-2} (1 - \widehat{\alpha}_i(1+\widehat{\alpha}_i)^{-1}) 
 = \mu_{k,i} + M_{i,k},
\end{align*}
where in the second and third step \refeqn{eqn_main_proof_gamma} was used.
Finally $\vect{q}^{(4)}_i = \vect{q}^{(1)}_i - K_i \vect{q}^{(3)}_i = (I - K_i K_i^\dagger)\vect{q}_i$ where
\begin{align*}
 \vect{q}^{(4)}_i = 
 \begin{bmatrix}
  \bigl( \bigl(-\mu_{1,i} - \sum_{k=2}^d (\mu_{k,i} + M_{i,k}) \bigr) \sten{a}{i}{1} \\
  (-\mu_{2,i} + \mu_{2,i} + M_{i,2}) \sten{a}{i}{2} \\
  \vdots \\
  (-\mu_{d,i} + \mu_{d,i} + M_{i,d}) \sten{a}{i}{d}
 \end{bmatrix}
 = \begin{bmatrix}
   (M_i - \sum_{k=2}^d M_{i,k}) \sten{a}{i}{1} \\
   M_{i,2} \sten{a}{i}{2} \\
   \vdots \\
   M_{i,d} \sten{a}{i}{d}
   \end{bmatrix},
\end{align*}
from which it is obvious that $\|\vect{\Delta}_{n,1}\| = \mathcal{O}(n^{-2})$.

Continuing from \refeqn{eqn_main_proof_lastfactor}, we can apply the Iterated Scaling Lemma to $\vect{p} + \forcebold{\nabla}_n$, where $\forcebold{\nabla}_n = \forcebold{\nabla}_{n,1} + \forcebold{\nabla}_{n,2}$ and $\|\forcebold{\nabla}_n\|=\mathcal{O}(n^{-1})$, thusly obtaining 
\[
\vect{p} + \widetilde{\Delta\vect{p}}_n = \dot{\vect{p}}_n + \vect{\Delta}_n' + \vect{\Delta}_{n,1} + \vect{\Delta}_{n,2} + \Delta\vect{p}_n' = \dot{\vect{p}}_n + n^{-1} \vect{w} + n^{-2} \vect{x}_n,
\]
where $\vect{\Delta}_n' \in \Var{K}^\perp$ is of norm $\mathcal{O}(n^{-2})$ for sufficiently large $n$, and $\vect{x}_n$ is a vector whose norm can be uniformly bounded by a constant $C$. By definition, $\dot{\vect{\Delta}}_n = n^{-1} \vect{w} + n^{-2} \vect{x}_n$ and
$T_\vect{p} \vect{w} = \varsigma_N \vect{w}$.
From the sandwiching
\[
\frac{\varsigma_N - n^{-1} \|T_\vect{p}\|_2\|\vect{x}_n\|}{1 + n^{-1}\|\vect{x}_n\|}
\le \frac{n^{-1} \|T_\vect{p} (\vect{w} + n^{-1}\vect{x}_n)\|}{n^{-1} \|\vect{w} + n^{-1}\vect{x}_n\|} 
\le \frac{\varsigma_N + n^{-1} \|T_\vect{p}\|_2\|\vect{x}_n\|}{1 - n^{-1}\|\vect{x}_n\|},
\]
it follows that $\|T_\vect{p} \dot{\vect{\Delta}}(\Delta\tensor{A}_n, \bullet)\|/\|\dot{\vect{\Delta}}(\Delta\tensor{A}_n, \bullet))\|$, where ``$\bullet$'' indicates an arbitrary $\widehat{S}(\Delta\tensor{A}_n)\in\Var{Z}$, tends to $\varsigma_N( T_\vect{p} )$ as $n\to\infty$, concluding the proof.
\end{proof}

%
%
%
%
%
%
%
%

A relative condition number may then be defined as follows.

\begin{definition}[Condition number]
Let $\vect{p}$ and $\tensor{A}$ be as in \refthm{thm_condition}. Then, the relative condition number of the tensor rank decomposition problem $\tensor{A} = f(\vect{p})$ at $\vect{p}$ is 
\[
 \kappa = \kappa_A \cdot \frac{\| \tensor{A} \|}{ \| \vect{p} \| }.
\]
\end{definition}

\section{The norm-balanced condition number}\label{sec_norm_balanced}
It should be remarked that the absolute and relative condition numbers considered in the previous section are defined at $\vect{p}$ rather than at $\tensor{A}$. The reason is that their value depends on the particular representative $\vect{p} \in f^\dagger(\tensor{A})$ that was chosen. With the foregoing definition of an absolute and a relative condition number, we could now define the condition number of the entire fiber $f^\dagger(\tensor{A})$ in several ways, each with its particular interpretation. For instance, one could define the condition number of $f^\dagger$ at $\tensor{A}$ to be the maximum, minimum or (weighted) average condition number of all $\vect{p}\in f^\dagger$, which would respectively measure the conditioning of the worst possible representative, the best possible representative, and the (weighted) average representative. I believe that all of the foregoing definitions are very interesting from a mathematical viewpoint, however I am skeptical that such definitions correspond naturally with the current practice in data-analytic applications involving the tensor rank decomposition. Through private conversation with G.~Tomasi \cite{TomasiPhD}, it appeared to me that in many applications, especially originating from physical measurements, there is no a priori reason to prefer one scaling of the factors over another. Therefore, I believe that a natural choice for the condition number of $f^\dagger$ at $\tensor{A}$ is the condition number of the norm-balanced case, which is defined next.

\begin{definition}(Norm-balanced condition number) \label{def_norm_balanced_cond}
 Let $\alpha_i \in \R_+$ be strictly positive,\footnote{Every rank-$r$ tensor can be represented as such, by absorbing the sign into, e.g., $\sten{u}{i}{1}$.} and let
 \[
  \tensor{A} = \sum_{i=1}^r \alpha_i \sten{u}{i}{1} \otimes \cdots \otimes \sten{u}{i}{d} \quad\text{with } \|\sten{u}{i}{k}\|=1
 \]
be a robustly $r$-identifiable tensor rank decomposition in $\F^{n_1} \otimes \cdots \otimes \F^{n_d}$. Then, its condition number is defined to be the condition number of the norm-balanced representatives 
\[
p_i = (\alpha_i^{1/d} \sten{u}{i}{1}, \alpha_i^{1/d}\sten{u}{i}{2},\ldots,\alpha_i^{1/d}\sten{u}{i}{d}), \quad i=1,2,\ldots,r.
\]
That is, the norm-balanced absolute and relative condition numbers of $f^\dagger$ at $\tensor{A}$ are respectively
\[
 \kappa_A = \varsigma_N^{-1} \quad\text{and}\quad \kappa = \varsigma_N^{-1} \frac{\|\tensor{A}\|}{\|\vect{p}\|},
\]
where $\vect{p} = \operatorname{vecr}(p_1,p_2,\ldots,p_r)$ and $\varsigma_N$ is the $N=r(\Sigma+1)$th largest singular value of $T_\vect{p}$.
\end{definition}

Whenever I refer to respectively \emph{the} absolute and relative condition number of a tensor, I mean the norm-balanced absolute and relative condition number, respectively. 

\begin{center}
\emph{In the remainder of the paper, only the norm-balanced condition number of $f^\dagger$ will be considered.}
\end{center}

The norm-balanced condition number has another practical advantage over the other choices: it is very easy to compute when a decomposition of the tensor is given, requiring only the least singular value of one Terracini's matrix; for completeness, such a procedure is presented as \refalg{alg_norm_balanced_cond}. Lines $1$ to $9$ implement the norm-balancing of the input tensor rank decomposition. The required Terracini matrix is constructed in lines $10$ to $13$. Finally, the condition number is computed in line $15$. A concrete implementation of this algorithm in Matlab/Octave is included in \refapp{app_code}.

\begin{algorithm}
\caption{Computing the norm-balanced relative condition number}
\label{alg_norm_balanced_cond}

\SetAlCapSty{textsc}

\SetKwInOut{Input}{input}
\SetKwInOut{Output}{output}
\DontPrintSemicolon

\Input{A robustly $r$-identifiable tensor $\tensor{A}=\sum_{i=1}^r \sten{a}{i}{1}\otimes\cdots\otimes\sten{a}{i}{d} \in \F^{n_1}\otimes\cdots\otimes\F^{n_d}$ }
\Output{The norm-balanced condition number $\kappa \in \R_+$ of $\tensor{A}$.}

$\zeta \leftarrow 1$\;
\For{$i = 1, 2, \ldots, r$}{
$\gamma \leftarrow 1$\;
\For{$k = 1, 2, \ldots, d$}{
$\eta \leftarrow \|\sten{a}{i}{k}\|$;\quad $\gamma \leftarrow \gamma \cdot \eta^{1/d}$;\quad $\sten{a}{i}{k} \leftarrow \sten{a}{i}{k} / \eta$\;
}
$\sten{a}{i}{k} \leftarrow \gamma \sten{a}{i}{k},\quad k=1,2,\ldots,d$\;
$\zeta \leftarrow d\gamma^2$\;
}

$T \leftarrow []$\;
\For{$i = 1, 2, \ldots, r$}{
  $T \leftarrow \begin{bmatrix} T & I_{n_1} \otimes \sten{a}{i}{2} \otimes \cdots \otimes \sten{a}{i}{d} & \cdots & \sten{a}{i}{1} \otimes \cdots \otimes \sten{a}{i}{d-1} \otimes I_{n_d} \end{bmatrix}$\;
}
$N \leftarrow r (n_1 + \cdots + n_d - d + 1)$\;
$\kappa \leftarrow (\varsigma_N(T) \sqrt{\zeta})^{-1} \cdot \|\tensor{A}\|$\;
\end{algorithm}

The computational complexity of computing the norm-balanced condition number is determined as follows. Note that we may always assume that $r\Sigma < \Pi$, because otherwise the input cannot be robustly $r$-identifiable. Since the condition number is $\infty$ in this case no involved computations are necessary. It is straightforward to verify that the computation of the least singular value in line $15$ of \refalg{alg_norm_balanced_cond} contributes the dominant factor to the time complexity analysis. For a general $m \times n$ matrix with $n \le m$, computing the least singular value costs $\mathcal{O}(m n^2)$ operations. The asymptotic space complexity is determined by the cost of storing Terracini's matrix, which is of size $\Pi \times r(\Sigma+d)$. This proves the following elementary result.

\begin{proposition}
If a length $r$ decomposition of a tensor in $\F^{n_1} \otimes \cdots \otimes \F^{n_d}$ is supplied to \refalg{alg_norm_balanced_cond} as input, then its time and space complexity are respectively
\[
  \mathcal{O}( \Pi (r (\Sigma+d))^2 ) = \mathcal{O}(\Pi^3) \quad\text{and}\quad \mathcal{O}(r (\Sigma+d) \Pi) = \mathcal{O}(\Pi^2), 
\]
where $\Pi = \prod_{i=1}^d n_i$ and $\Sigma + d = \sum_{i=1}^d n_i$. 
\end{proposition}

\section{Basic properties of the norm-balanced condition number}\label{sec_basic_properties}
We proceed by investigating several basic properties of the norm-balanced condition number, starting with three desirable properties that are also exhibited by the familiar condition number of a matrix, namely its continuity (under mild conditions), and its invariance under scaling and orthogonal change of bases.

\begin{proposition}
The absolute and relative condition numbers are continuous in a neighborhood of a robustly $r$-identifiable tensor. 
\end{proposition}
\begin{proof}
 This follows immediately from the continuity of Terracini's matrix $T_\vect{p}$ in the parameters $\vect{p}$, which is a consequence from the assumption of regularity in the definition of robust $r$-identifiability.
\end{proof}

\begin{proposition}
The relative condition number is scale-invariant. 
\end{proposition}
\begin{proof}
The claim is that
\(
 \kappa(\tensor{A}) = \kappa(\beta \tensor{A})
\)
for all $\beta \in \F\setminus\{0\}$. Note that $\beta\tensor{A} = \alpha^{d}\tensor{A} = f(\alpha\vect{p})$, so $\|\alpha^d\tensor{A}\|/\|\alpha\vect{p}\| = |\alpha|^{d-1} \|\tensor{A}\|/\|\vect{p}\|$. Terracini's matrix corresponding to the points represented by $\vect{q}=\alpha\vect{p}$ is verified to be $T_\vect{q} = \alpha^{d-1} T_\vect{p}$, where $T_\vect{p}$ is Terracini's matrix corresponding to $\vect{p}$. Hence, $\varsigma_N(T_\vect{q})^{-1} = (|\alpha|^{d-1} \varsigma_N(T_\vect{p}))^{-1} = |\alpha|^{1-d}\kappa_A(T_\vect{p})$. This concludes the proof for $\F = \C$. In the real case, whenever $d$ is even, one should exploit for $\beta < 0$ the equality $\beta\tensor{A} = -\alpha^d\tensor{A} = f(\alpha\vect{p}')$, where $\vect{p}'=\operatorname{vecr}(p_1',\ldots,p_r')$ with $p_i' = (-\sten{a}{i}{1},\sten{a}{i}{2},\ldots,\sten{a}{i}{d})$. Then, Terracini's matrix corresponding to $\vect{q}' = \alpha\vect{p}'$ is $T_{\vect{q}'} = \alpha^{d-1}T_\vect{p} S$, where $S$ is a diagonal matrix whose diagonal entries are $\pm1$. Since $S$ is an orthogonal matrix, $\varsigma_N(T_\vect{p}) = \varsigma_N(T_\vect{p} S)$ and thusly is the proof concluded.
\end{proof}

\begin{proposition}
The absolute and relative condition number are orthogonally invariant.
\end{proposition}
\begin{proof}
Let $\vect{p}$ and $\tensor{A}$ be as in \refthm{thm_condition}. The claim is that $\kappa_A$ and $\kappa$ are invariants of the $(O(n_1) \times \cdots \times O(n_d))$-orbit of $\tensor{A}$. In other words, considering $\tensor{A} \in \F^{\Pi}$, the statement is that 
\[
\kappa(\tensor{A}) = \kappa( (Q_1\otimes\cdots\otimes Q_d) \tensor{A})
\]
where $Q_i \in O(n_i)$ are orthogonal matrices in $\F^{n_i \times n_i}$ with respect to the Hermitian inner product. 
As $Q = Q_1\otimes\cdots\otimes Q_d$ is orthogonal, $\|\tensor{A}\|=\|Q\tensor{A}\|$. Since
\[
 Q\tensor{A} = (Q_1 \otimes \cdots \otimes Q_d)\sum_{i=1}^r \sten{a}{i}{1} \otimes \cdots \otimes \sten{a}{i}{d} = \sum_{i=1}^r Q_1\sten{a}{i}{1} \otimes \cdots \otimes Q_d\sten{a}{i}{d},
\]
it follows that $Q\tensor{A} = f(\diag(Q_1,\ldots,Q_d,\ldots,Q_1,\ldots,Q_d) \vect{p}) = f(U\vect{p})$. As $U$ is orthogonal, $\|\vect{p}\| = \|U\vect{p}\|$. Finally, Terracini's matrix corresponding to $\vect{p}' = U\vect{p}$ is given explicitly by
\[
 T_{\vect{p}'} = \begin{bmatrix} T_1' & \cdots & T_r' \end{bmatrix},
 \quad\text{where }
 T_i' = \begin{bmatrix} I\otimes Q_2\sten{a}{i}{2} \otimes \cdots \otimes Q_d\sten{a}{i}{d} & \cdots & Q_1\sten{a}{i}{1}\otimes\cdots\otimes Q_{d-1}\sten{a}{i}{d-1}\otimes I \end{bmatrix}.
\]
Multiplying $T_i'$ on the right by $D = \diag(Q_1,Q_2,\ldots,Q_d)$ results in
\(
T_i' D = Q T_i,
\)
so that 
\(
 Q T_\vect{p} = T_{\vect{p}'} \diag(D,\ldots,D).
\)
Since $Q$ and $\diag(D,\ldots,D)$ are orthogonal, one obtains 
\[
\varsigma_N(T_\vect{p}) = \varsigma_N(Q T_\vect{p}) = \varsigma_N(T_{\vect{p}'} \diag(D,\ldots,D)) = \varsigma_N(T_{\vect{p}'}), 
\]
concluding the proof.
\end{proof}

The condition number of a rank-$1$ tensor has an explicit expression; in fact, the relative condition number depends only on the order $d$ of the tensor, and \emph{decreases} as $d$ increases. As rank-$1$ tensors are robustly $r$-identifiable for all $d\ge2$, the result even applies to matrices.

\begin{proposition}\label{prop_rank1}
A rank-$1$ tensor $\tensor{A} = \alpha \sten{a}{}{1} \otimes \cdots \otimes \sten{a}{ }{d} \in \F^{n_1}\otimes\cdots\otimes\F^{n_d}$ with $\|\sten{a}{ }{k}\|=1$ and $\alpha \in \R_+$ strictly positive has absolute condition number $\kappa_A(\tensor{A}) = \alpha^{1/d-1}$ and relative condition number $\kappa(\tensor{A}) = \sqrt{d^{-1}}$. 
\end{proposition}
\begin{proof}
Let $\vect{p}' = \operatorname{vec}(\alpha^{1/d}\sten{a}{}{1}, \alpha^{1/d}\sten{a}{}{2},\ldots,\alpha^{1/d}\sten{a}{}{d})$ and $\vect{p} = \operatorname{vec}(\sten{a}{}{1},\sten{a}{ }{2}, \ldots, \sten{a}{}{d})$.
Since the decomposition is norm-balanced, one verifies that $ T_{\vect{p}'} = \alpha^{1-1/d} T_\vect{p}$. 
Hence, finding the singular values of $T_\vect{p}$ suffices. 
The largest singular value of $T_\vect{p}$ is given by
\begin{align*}
 \varsigma_1(T_\vect{p}) = \max_{\|\vect{v}\| = 1} \|T_\vect{p} \vect{v}\|
 &= \max_{\|\vect{v}\| = 1} \| \sten{v}{ }{1} \otimes \sten{a}{ }{2}\otimes\cdots\otimes \sten{a}{ }{d} + \cdots + \sten{a}{ }{1}\otimes\cdots\otimes\sten{a}{ }{d-1}\otimes\sten{v}{ }{d} \| \\
 &\le \max_{\|\vect{v}\| = 1} \bigl( \| \sten{v}{ }{1} \otimes \sten{a}{ }{2}\otimes\cdots\otimes \sten{a}{ }{d} \|  + \cdots + \| \sten{a}{ }{1}\otimes\cdots\otimes\sten{a}{ }{d-1}\otimes\sten{v}{ }{d} \| \bigr) \\
 &= \max_{\|\vect{v}\| = 1} \bigl( \| \sten{v}{ }{1} \| + \cdots + \| \sten{v}{ }{d} \| \bigr) 
 = \max_{\|\sten{v}{}{1}\|^2 + \cdots + \|\sten{v}{ }{d}\|^2 = 1} \bigl( \| \sten{v}{ }{1} \| + \cdots + \| \sten{v}{ }{d} \| \bigr) \\
 &= \max_{\nu_1^2 + \cdots + \nu_d^2 = 1} ( |\nu_1| + \cdots + |\nu_d| ) 
 = \max_{\|\forcebold{\nu}\|=1} \|\forcebold{\nu}\|_1 \le \sqrt{d} \|\forcebold{\nu}\|_2 = \sqrt{d},
\end{align*}
where $\vect{v}^T = \left[\begin{smallmatrix}(\sten{v}{}{1})^T \, \cdots\, (\sten{v}{ }{d})^T \end{smallmatrix}\right]$ and $\forcebold{\nu} = \left[\begin{smallmatrix}\nu_1 & \cdots & \nu_d\end{smallmatrix}\right]$. If we take $\vect{v}_1 = \sqrt{d^{-1}} \operatorname{vec}(\sten{a}{ }{1},\cdots,\sten{a}{ }{d})$, then it follows that $\|T_\vect{p}\vect{v}\| = \|d\sqrt{d^{-1}}\sten{a}{ }{1}\otimes\cdots\otimes\sten{a}{ }{d}\|= \sqrt{d}$. That is, $\vect{v}_1$ is a right singular vector corresponding to the largest singular value $\varsigma_1(T_\vect{p}) = \sqrt{d}$. Since $T_\vect{p} \vect{v}_1 = \sqrt{d} \sten{a}{ }{1} \otimes \cdots \otimes \sten{a}{ }{d} = \sqrt{d} \vect{u}_1$, it follows that $\vect{u}_1$ is a left singular vector corresponding to $\varsigma_1(T_\vect{p})$. Deflating the largest singular tuple from $T_\vect{p}$, we find
\[
 \widehat{T}_\vect{p} = T_\vect{p} - \sqrt{d} \vect{u}_1 \vect{v}_1^H = 
 \begin{bmatrix}
  (I - \sten{a}{ }{1}(\sten{a}{ }{1})^H) \otimes \sten{a}{ }{2} \otimes \cdots \otimes \sten{a}{ }{d} & \cdots &
  \sten{a}{ }{1} \otimes \cdots \otimes \sten{a}{ }{d-1} \otimes (I - \sten{a}{ }{d} (\sten{a}{ }{d})^H)
 \end{bmatrix}.
\]
The singular values of the above matrix are the square roots of the eigenvalues of its Gram matrix, which has a particularly pleasing structure as it is block diagonal,
\[
 \widehat{T}_\vect{p}^H \widehat{T}_\vect{p} = 
 \diag\Bigl( (I - \sten{a}{ }{1}(\sten{a}{ }{1})^H)^2, (I - \sten{a}{ }{2} (\sten{a}{ }{2})^H)^2, \ldots ( I - \sten{a}{ }{d}(\sten{a}{ }{d})^H )^2 \Bigr),
\]
which is a consequence from the fact that $I - \sten{a}{ }{k}(\sten{a}{ }{k})^H$ is a projector onto the orthogonal complement of $\sten{a}{ }{k}$, so that applying it to the latter yields $0$. Since $I - \sten{a}{ }{k}(\sten{a}{ }{k})^H$ is a projector, it is idempotent, and so the eigenvalues of $\widehat{T}_\vect{p}^H \widehat{T}_\vect{p}$ are the union of the eigenvalues of the $I - \sten{a}{ }{k}(\sten{a}{ }{k})^H$ for $k=1,\ldots,d$. The only possible eigenvalues of a projector are $1$ and $0$; so eigenvalue $1$ has multiplicity equal to the rank of the projector, which is readily verified to be $n_k-1$. Hence, the singular values of $\widehat{T}_\vect{p}$ are $1$ with multiplicity $\Sigma$ and $0$ with multiplicity $d$. Then, the singular values of $T_\vect{p}$ are $\sqrt{d}$, $1$ with multiplicity $\Sigma$, and $0$ with multiplicity $d-1$. We can conclude that the $N = (\Sigma+1)$th singular value of $T_{\vect{p}'}$ is $\alpha^{1-1/d}$, so that the absolute condition number is $\kappa_A = \alpha^{1/d-1}$. Finally,
\[
 \kappa = \kappa_A \frac{\|\alpha \sten{a}{ }{1}\otimes\cdots\otimes\sten{a}{ }{d}\|}{\|\vect{p}'\|} = \alpha^{1/d-1} \cdot \alpha \cdot ( \sqrt{d} \alpha^{1/d} )^{-1} = \sqrt{d^{-1}},
\]
concluding the proof.
\end{proof}

The following result analytically characterizes the singular values of Terracini's matrix corresponding to a norm-balanced rank-$1$ tensor; it is an immediate consequence from the foregoing proof.
\begin{corollary}\label{cor_terracini_sings}
 Let $p = (\sten{a}{ }{1}, \sten{a}{ }{2}, \ldots, \sten{a}{ }{d})$ with $\|\sten{a}{ }{1}\|=\cdots=\|\sten{a}{}{d}\| = \alpha^{1/d} > 0$, and let $\vect{p} = \vecc{p}$. Then, the singular values of Terracini's matrix are 
 \[
  \varsigma( T_\vect{p} ) = \{ \sqrt{d} \alpha^{1-1/d}, \underset{\Sigma}{\underbrace{\alpha^{1-1/d}, \ldots, \alpha^{1-1/d}}}, \underset{d-1}{\underbrace{0, \ldots, 0}} \}.
 \]
\end{corollary}

Based on the foregoing characterization of the singular values of Terracini's matrix associated with one rank-$1$ tensor, we can determine explicitly the condition number of the class of \emph{weak $k$-orthogonal tensors} with $k\ge3$. Both strong and weak $2$-orthogonality were considered in \cite{VNVM2014} from whence it name derives. A tensor rank decomposition is said to be weak $k$-orthogonal if for every pair of rank-$1$ tensors, there is orthogonality in $k$ factors, however the factors in which orthogonality occurs need not be the same for different pairs. It was shown in \cite{VNVM2014} that a weak $2$-orthogonal tensor rank decomposition is a necessary condition for the existence of a certain generalization of the Eckart--Young theorem for matrices, and a particular subclass of weak $2$-orthogonal tensors, i.e., those with a strong $2$-orthogonal decomposition, in fact admits a tensor-equivalent of the Eckart--Young theorem. In the next result, it is shown that tensors with a weak $3$-orthogonal decomposition admit an easy expression for their condition number.

\begin{proposition}\label{prop_weak3orth}
Let $\alpha_i \in \R_+$ be sorted as $\alpha_1 \ge \alpha_2 \ge \cdots \ge \alpha_r > 0$, and let
\[
 \tensor{A} = \sum_{i=1}^r \alpha_i \sten{v}{i}{1}\otimes \cdots \otimes \sten{v}{i}{d} \quad\text{with } \|\sten{v}{i}{k}\|=1
\]
be a robustly $r$-identifiable weak $3$-orthogonal tensor:
\[
 \forall i < j: \exists 1 \le k_1 < k_2 < k_3 \le d: \langle \sten{v}{i}{k_1}, \sten{v}{j}{k_1} \rangle = \langle \sten{v}{i}{k_2}, \sten{v}{j}{k_2} \rangle = \langle \sten{v}{i}{k_3}, \sten{v}{j}{k_3} \rangle = 0,
\]
where $\langle \cdot,\cdot \rangle$ is the Hermitian inner product.
Then, 
\[
 \kappa_A(\tensor{A}) = \alpha_r^{-1+1/d} 
 \quad\text{and}\quad \kappa(\tensor{A}) = \alpha_r^{-1+1/d} \sqrt{ \sum_{i=1}^r \alpha_i^2 } \Bigg/ \sqrt{ \sum_{i=1}^r d \alpha_i^{2/d} }.
\]
\end{proposition}
\begin{proof}
Let $\vect{p}=\operatorname{vecr}(p_1,\ldots,p_r)$ with $p_i$ the norm-balanced representatives as in \refdef{def_norm_balanced_cond}.
By weak $3$-orthogonality, we have for every $1\le i < j \le r$ that there exist $k_1 < k_2$ so that
\begin{align*}
 (\sten{v}{i}{1} \otimes \cdots \otimes \sten{v}{i}{k_1-1} \otimes I \otimes \sten{v}{i}{k_1+1} \otimes \cdots \otimes \sten{v}{i}{d})^H
 (\sten{v}{j}{1} \otimes \cdots \otimes \sten{v}{j}{k_2-1} \otimes I \otimes \sten{v}{j}{k_2+1} \otimes \cdots \otimes \sten{v}{j}{d})\qquad\qquad& \\
  = \sten{v}{j}{k_1} (\sten{v}{i}{k_2})^H \cdot \prod_{\substack{k=1,\\k\ne k_1, k_2}}^d \langle \sten{v}{i}{k}, \sten{v}{j}{k} \rangle = 0,&
\end{align*}
because the product excludes only $2$ inner products so that there is always one inner product left with orthogonality. Then the foregoing entails that $T_i^H T_j = 0$ for all $i \ne j$, so that
\[
 T_\vect{p}^H T_\vect{p} = \diag( T_1^H T_1, T_2^H T_2, \ldots, T_r^H T_r ),
\]
where $T_i$ is as in \refeqn{eqn_terracini}. The singular values of $T_\vect{p}$ thus coincide with the union of the set of singular values of each of the $T_i$'s. From \refcor{cor_terracini_sings} it follows that 
\[
 \varsigma(T_\vect{p}) = \bigcup_{i=1}^r \{ \sqrt{d} \alpha_i^{1-1/d}, \alpha_i^{1-1/d},\ldots,\alpha_i^{1-1/d}, 0,\ldots,0 \};
\]
the $N$th singular value is thus $\varsigma_N(T_\vect{p}) = \alpha_r^{1-1/d}$ by our assumption on the order of the $\alpha_i$. The norm $\|\tensor{A}\|^2 = \sum_{i=1}^r \alpha_i^2$ because of the orthogonality of the $\sten{v}{i}{1}\otimes\cdots\otimes\sten{v}{i}{d}$'s. The norm of $\vect{p}$ is clear.
\end{proof}

\begin{remark}
The foregoing class of weak $3$-orthogonal tensors includes rank-$1$ tensors and \emph{orthogonally decomposable tensors} (odeco) of order $d\ge3$ \cite{Zhang2001,Kolda2001,Chen2009,Robeva16,BDHR16,VNVM2014}. 
\end{remark}

As the relative condition number of order-$d$ rank-$1$ tensors is $d^{-1/2}$, we can safely call such tensors \emph{perfectly conditioned}: the relative error in the parameters of the decomposition is \emph{less} than the relative error of the tensor. In fact, it is moderately decreasing with the order of the tensor. Recall that the relative condition number is defined in such a way that if $\kappa = 10^{k}$, then approximately $k$ significant digits of accuracy are lost in the parameters of the decompositions $f^\dagger(\tensor{A})$ and $f^\dagger(\tensor{A}+\Delta\tensor{A})$ relative to the error between the tensors $\tensor{A}$ and $\tensor{A}+\Delta\tensor{A}$ for sufficiently small $\|\Delta\tensor{A}\|$. Whenever $\log_{10}(\kappa)$ is small relative to $-\log_{10}(\|\Delta\tensor{A}\|/\|\tensor{A}\|)$, one may call $\tensor{A}$ \emph{well-conditioned}. If the former is large relative to the former, then $\tensor{A}$ is called \emph{ill-conditioned}. If $\kappa \|\Delta\tensor{A}\| > \|\tensor{A}\|$, then all significant digits may be lost in the output $f^\dagger(\tensor{A})$. 
The next result bounds the relative condition number of any robustly $r$-identifiable tensor from below.

\begin{proposition}
The relative condition number of robustly $r$-identifiable tensors in $\F^{n_1} \otimes \cdots \otimes \F^{n_d}$ is bounded from below by $d^{-1}$. 
\end{proposition}
\begin{proof}
 Let $\vect{p}$ and $\tensor{A}$ be as in \refthm{thm_condition}. Then,
 \[
  \kappa(\tensor{A}) = \varsigma_N(T_\vect{p})^{-1} \frac{\|\tensor{A}\|}{\|\vect{p}\|} = \varsigma_N(T_\vect{p})^{-1} \frac{\|d^{-1} T_\vect{p} \vect{p}\|}{\|\vect{p}\|} = d^{-1} \varsigma_N(T_\vect{p})^{-1} \frac{\|T_\vect{p} \vect{p}\|}{\|\vect{p}\|}.
 \]
Let $K$ be as in \refeqn{eqn_terracini_kernel}. Then by assumption on the rank of $T_\vect{p}$, $K K^\dagger = K(K^H K)^{-1} K^H$ is a projector onto the kernel of $T_\vect{p}$. Let $\vect{p}_i^T = \left[\begin{smallmatrix} (\sten{a}{i}{1})^T & \cdots & (\sten{a}{i}{d})^T \end{smallmatrix}\right]$, and then it follows that $K_i^H \vect{p}_i = 0$. Hence, $K^H \vect{p} = 0 = K K^\dagger \vect{p}$. So $\vect{p}$ is contained in $\cspan(K)^\perp$. Therefore, $\|T_\vect{p} \vect{p}\|/\|\vect{p}\| \ge \varsigma_N(T_\vect{p})$ and the result follows.
\end{proof}

\begin{remark}
It is unknown to me whether this lower bound may be sharp, because it is neither sharp for rank-$1$ tensors nor for weak $3$-orthogonal tensors with all $\alpha_i = \alpha \ne 0$. I cannot presently think of tensors that would be even more well-behaved than completely orthogonal tensors all of whose terms are of equal norm; these tensors have condition number $\sqrt{d^{-1}}$.
\end{remark}

%
%
%
%
\section{Numerical examples} \label{sec_numerical_experiments}
In this section, some experiments are performed with the basic Matlab/Octave implementation of the proposed algorithm for computing the norm-balanced condition number that is included in the ancillary files accompanying this paper. 
All of the experiments were performed using Matlab R2015a on a computer system consisting of an Intel Core i7-5600U CPU, clocked at 2.6GHz, and $8$GB of main memory. Some of the experiments employed Tensorlab v2.02 \cite{Tensorlab} for computing approximate tensor rank decompositions.

\subsection{The Iterated Scaling Lemma}\label{sec_illustration_main_thm}
Let us start by experimentally investigating \reflem{lem_isl}. It claims that a perturbation of the vectorized factor matrices $\vect{p}$ in the direction of the kernel of Terracini's matrix $T_\vect{p}$ may be replaced with a new set of vectorized factor matrices representing the same tensor plus a much smaller perturbation perpendicular to the kernel of $T_\vect{p}$. For simplicity, take the following rank-$2$ tensor of size $3 \times 3 \times 2$, which is robustly $r$-identifiable by Kruskal's theorem:
\begin{verbatim}
A = [[2 0]; [-1 1]; [0 2]];  B = [[-1 -2]; [2 0]; [0 1]];  C = [[1 -2]; [2 1]];
\end{verbatim}
These factor matrices are norm-balanced, and the corresponding Terracini's matrix can be constructed with the \texttt{cpdcond} algorithm that is included in the ancillary files:
\begin{verbatim}
[k,F,Tr] = cpdcond({A,B,C});
\end{verbatim}
The relative condition number is approximately $\kappa = 0.769$. Next, a random perturbation of unit norm in the direction of the kernel of \texttt{Tr} is constructed as follows:
\begin{verbatim}
[~,~,V] = svd(Tr); v = orth( V(:,13:16)*randn(4,1) );
\end{verbatim}
Now, the vectorized factor matrices are represented by the vector
\[
 \vect{p}^T = 
 \left[\begin{array}{cccccccccccccccc}
  2 & -1 & 0 & -1 & 2 & 0 & 1 & 2 & 0 & 1 & 2 & -2 & 0 & 1 & -2 & 1
 \end{array}\right].
\]
Applying the algorithm in Part I of the proof of the Iterated Scaling Lemma is accomplished as follows:
\begin{verbatim}
nabla = 1e-2*v; [q, delta] = isl(p, nabla, [3 3 2], 2)
q =                            delta =
   1.999142573254893e+00         -1.058731468865185e-06
  -9.995712866274463e-01          5.293657344325927e-07
                       0          2.132104856281811e-22
  -1.001419523209202e+00          5.293657344325930e-07
   2.002839046418403e+00         -1.058731468865186e-06
                       0                              0
   9.990107782605554e-01         -5.293657344325924e-07
   1.998021556521111e+00         -1.058731468865185e-06
                       0                              0
   1.000440759343655e+00         -2.804966567016011e-06
   2.000881518687309e+00         -5.609933134032022e-06
  -1.993817021879831e+00          5.609933134032020e-06
                       0                              0
   9.969085109399156e-01         -2.804966567016010e-06
  -2.005318289232261e+00          5.609933134032017e-06
   1.002659144616131e+00         -2.804966567016008e-06
\end{verbatim}
Notice that the relative difference between \verb|p| and \verb|q| is of the order $10^{-3}$. The vector \verb|delta| is of norm approximately equal to $1.1055 \cdot 10^{-5}$. According to \reflem{lem_isl}, the norm of this vector should be of the order $10^{-4} = \|10^{-2} \vect{v}\|^2$, which is indeed the case. In addition, the Iterated Scaling Lemma claims that \verb|delta| should be perpendicular to the kernel of Terracini's matrix \verb|Tr|. Computing \verb|norm(V(:,13:16)'*delta)| one finds approximately $4.2999 \cdot 10^{-21}$, which is a fairly good approximation to zero. Finally, \reflem{lem_isl} states that the vector \verb|q| should correspond with a different representative of the same rank-$2$ decomposition as represented by the vectorized factor matrices \verb|p|. We can test this property by computing the norm of the difference between the individual rank-$1$ tensors represented respectively by \verb|q| and \verb|p|:
\begin{verbatim}
norm( kron(q(1:3),kron(q(4:6),q(7:8))) - kron(A(:,1),kron(B(:,1),C(:,1))) )
ans =
     0
norm( kron(q(9:11),kron(q(12:14),q(15:16))) - kron(A(:,2),kron(B(:,2),C(:,2))) )
ans =
     1.241267076623637e-15
\end{verbatim}

\begin{figure}
 \caption{Evolution of $\|\forcebold{\nabla}^{(k)}\|$ in function of the iteration number, for different initial values of $\|\forcebold{\nabla}^{(1)}\| = 10^{-q}$.}
 \label{fig_convergence_isl}
\begingroup
  \makeatletter
  \providecommand\color[2][]{%
    \GenericError{(gnuplot) \space\space\space\@spaces}{%
      Package color not loaded in conjunction with
      terminal option `colourtext'%
    }{See the gnuplot documentation for explanation.%
    }{Either use 'blacktext' in gnuplot or load the package
      color.sty in LaTeX.}%
    \renewcommand\color[2][]{}%
  }%
  \providecommand\includegraphics[2][]{%
    \GenericError{(gnuplot) \space\space\space\@spaces}{%
      Package graphicx or graphics not loaded%
    }{See the gnuplot documentation for explanation.%
    }{The gnuplot epslatex terminal needs graphicx.sty or graphics.sty.}%
    \renewcommand\includegraphics[2][]{}%
  }%
  \providecommand\rotatebox[2]{#2}%
  \@ifundefined{ifGPcolor}{%
    \newif\ifGPcolor
    \GPcolortrue
  }{}%
  \@ifundefined{ifGPblacktext}{%
    \newif\ifGPblacktext
    \GPblacktexttrue
  }{}%
  \let\gplgaddtomacro\g@addto@macro
  \gdef\gplbacktext{}%
  \gdef\gplfronttext{}%
  \makeatother
  \ifGPblacktext
    \def\colorrgb#1{}%
    \def\colorgray#1{}%
  \else
    \ifGPcolor
      \def\colorrgb#1{\color[rgb]{#1}}%
      \def\colorgray#1{\color[gray]{#1}}%
      \expandafter\def\csname LTw\endcsname{\color{white}}%
      \expandafter\def\csname LTb\endcsname{\color{black}}%
      \expandafter\def\csname LTa\endcsname{\color{black}}%
      \expandafter\def\csname LT0\endcsname{\color[rgb]{1,0,0}}%
      \expandafter\def\csname LT1\endcsname{\color[rgb]{0,1,0}}%
      \expandafter\def\csname LT2\endcsname{\color[rgb]{0,0,1}}%
      \expandafter\def\csname LT3\endcsname{\color[rgb]{1,0,1}}%
      \expandafter\def\csname LT4\endcsname{\color[rgb]{0,1,1}}%
      \expandafter\def\csname LT5\endcsname{\color[rgb]{1,1,0}}%
      \expandafter\def\csname LT6\endcsname{\color[rgb]{0,0,0}}%
      \expandafter\def\csname LT7\endcsname{\color[rgb]{1,0.3,0}}%
      \expandafter\def\csname LT8\endcsname{\color[rgb]{0.5,0.5,0.5}}%
    \else
      \def\colorrgb#1{\color{black}}%
      \def\colorgray#1{\color[gray]{#1}}%
      \expandafter\def\csname LTw\endcsname{\color{white}}%
      \expandafter\def\csname LTb\endcsname{\color{black}}%
      \expandafter\def\csname LTa\endcsname{\color{black}}%
      \expandafter\def\csname LT0\endcsname{\color{black}}%
      \expandafter\def\csname LT1\endcsname{\color{black}}%
      \expandafter\def\csname LT2\endcsname{\color{black}}%
      \expandafter\def\csname LT3\endcsname{\color{black}}%
      \expandafter\def\csname LT4\endcsname{\color{black}}%
      \expandafter\def\csname LT5\endcsname{\color{black}}%
      \expandafter\def\csname LT6\endcsname{\color{black}}%
      \expandafter\def\csname LT7\endcsname{\color{black}}%
      \expandafter\def\csname LT8\endcsname{\color{black}}%
    \fi
  \fi
  \setlength{\unitlength}{0.0500bp}%
  \scalebox{0.90}{
  \begin{picture}(8640.00,3528.00)%
    \gplgaddtomacro\gplbacktext{%
      \csname LTb\endcsname%
      \put(990,704){\makebox(0,0)[r]{\strut{} $10^{-70}$}}%
      \csname LTb\endcsname%
      \put(990,1070){\makebox(0,0)[r]{\strut{} $10^{-60}$}}%
      \csname LTb\endcsname%
      \put(990,1435){\makebox(0,0)[r]{\strut{} $10^{-50}$}}%
      \csname LTb\endcsname%
      \put(990,1801){\makebox(0,0)[r]{\strut{} $10^{-40}$}}%
      \csname LTb\endcsname%
      \put(990,2166){\makebox(0,0)[r]{\strut{} $10^{-30}$}}%
      \csname LTb\endcsname%
      \put(990,2532){\makebox(0,0)[r]{\strut{} $10^{-20}$}}%
      \csname LTb\endcsname%
      \put(990,2897){\makebox(0,0)[r]{\strut{} $10^{-10}$}}%
      \csname LTb\endcsname%
      \put(990,3263){\makebox(0,0)[r]{\strut{} $1$}}%
      \csname LTb\endcsname%
      \put(1122,484){\makebox(0,0){\strut{} 0}}%
      \csname LTb\endcsname%
      \put(1913,484){\makebox(0,0){\strut{} 1}}%
      \csname LTb\endcsname%
      \put(2704,484){\makebox(0,0){\strut{} 2}}%
      \csname LTb\endcsname%
      \put(3495,484){\makebox(0,0){\strut{} 3}}%
      \csname LTb\endcsname%
      \put(4286,484){\makebox(0,0){\strut{} 4}}%
      \csname LTb\endcsname%
      \put(5078,484){\makebox(0,0){\strut{} 5}}%
      \csname LTb\endcsname%
      \put(5869,484){\makebox(0,0){\strut{} 6}}%
      \csname LTb\endcsname%
      \put(6660,484){\makebox(0,0){\strut{} 7}}%
      \csname LTb\endcsname%
      \put(7451,484){\makebox(0,0){\strut{} 8}}%
      \csname LTb\endcsname%
      \put(8242,484){\makebox(0,0){\strut{} 9}}%
      \put(4682,154){\makebox(0,0){\strut{}$k$}}%
    }%
    \gplgaddtomacro\gplfronttext{%
      \csname LTb\endcsname%
      \put(1914,1757){\makebox(0,0)[r]{\strut{}$q = 1$}}%
      \csname LTb\endcsname%
      \put(1914,1537){\makebox(0,0)[r]{\strut{}$q = 2$}}%
      \csname LTb\endcsname%
      \put(1914,1317){\makebox(0,0)[r]{\strut{}$q = 3$}}%
      \csname LTb\endcsname%
      \put(1914,1097){\makebox(0,0)[r]{\strut{}$q = 4$}}%
      \csname LTb\endcsname%
      \put(1914,877){\makebox(0,0)[r]{\strut{}$q = 5$}}%
    }%
    \gplbacktext
    \put(0,0){\includegraphics{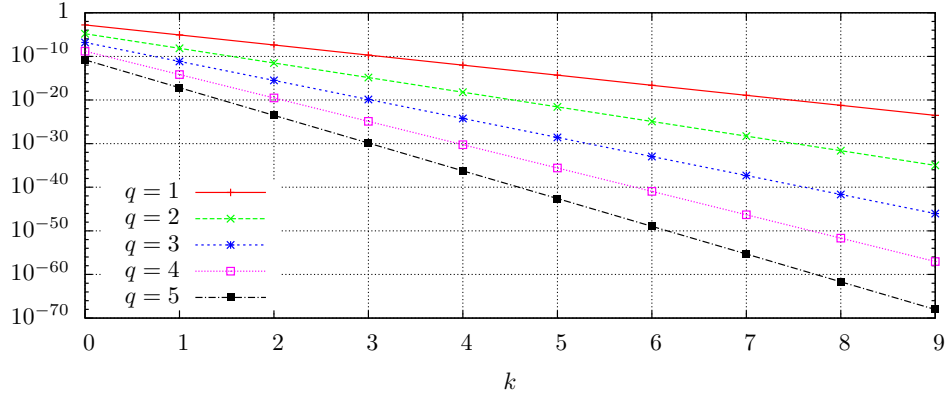}}%
    \gplfronttext
  \end{picture}%
  }
\endgroup
\end{figure}

As a final experiment involving the Iterated Scaling Lemma, I present in \reffig{fig_convergence_isl} the value of $\|\forcebold{\nabla}^{(k)}\|$ in the proof of \reflem{lem_isl} for the perturbation vector \verb|nabla = 10^(-q)*v| whereby $q = 1,2, \ldots, 5$. The figure illustrates the linear convergence claimed in \refeqn{eqn_isl_proof_h1}. Matlab's variable precision arithmetic capabilities using $100$ digits of accuracy were employed for generating this figure.

\subsection{The main theorem}
For illustrating \refthm{thm_condition}, I propose investigating the quantity on the left hand side of
\begin{align}\label{eqn_proxy_cond}
 \frac{\|\vect{p} - \widehat{\vect{p}}\|/\|\vect{p}\|}{\|f(\vect{p})-f(\widehat{\vect{p}})\|/\|f(\vect{p})\|} \ge \frac{d(\vect{p}, \widehat{\vect{p}}) / \|\vect{p}\|}{\|f(\vect{p})-f(\widehat{\vect{p}})\|/\|f(\vect{p})\|}
\end{align}
as a proxy for the right hand side. We may estimate the condition number $\kappa$ by generating a very large number of vectors $\widehat{\vect{p}}$ such that $\|f(\vect{p})-f(\widehat{\vect{p}})\| \le \varepsilon$, for some small value of $\varepsilon$, and then taking the maximum value. For instance, let us generate one random positive rank-$2$ decomposition in $\R^{3 \times 3 \times 2}$ as follows:
\begin{verbatim}
A = rand(3,2); B = rand(3,2); C = rand(2,2);
[c,F,~] = cpdcond({A,B,C}); 
\end{verbatim}
In my experiment,
\begin{verbatim}
F{1} =                        F{2} =                        F{3} = 
 5.1518e-01  8.8821e-01        1.9032e-01  7.5082e-01        7.2302e-01  6.9447e-01
 4.9802e-01  3.6941e-01        5.4218e-01  1.6653e-01        4.9879e-01  6.7487e-01
 5.0806e-01  1.1117e-01        6.6436e-01  5.8845e-01
\end{verbatim}
and the relative norm-balanced condition number was $\kappa \approx 18.410308$. Then, the factor matrices \verb|F| were perturbed randomly as follows
\begin{verbatim}
Fp = arrayfun(... 
   @(k) F{k}+1e-4*norm(F{k},'fro')*rand(size(F{k})), 1:3, 'UniformOutput', false ...
 );
\end{verbatim}
Starting from these factor matrices, I instructed the \verb|cpd_nls| algorithm from Tensorlab to compute an approximate tensor rank decomposition of the tensor $f(\vect{p})$, where $\vect{p}$ are the vectorized factor matrices corresponding to \verb|F|. Let $\widehat{\vect{p}}$ denote the set of norm-balanced vectorized factor matrices obtained from a call to \verb|cpd_nls|. The forgoing computation is repeated $10$ million times; for every instance where $\|f(\vect{p})-f(\widehat{\vect{p}})\| \le 10^{-14} \cdot \|f(\vect{p})\|$, the value on the left-hand side of \refeqn{eqn_proxy_cond} is computed. This resulted in $9,816,827$ valid samples. Taking the maximum value over all the valid samples yielded approximately $10.437311$ as estimate for the condition number. Given the very large number of trials, this is poor approximation to $\kappa$. In fact, the mean estimate of the condition number over all trials was only $0.8667234$. These results seem to suggest that most perturbations---in this example---are actually very good for the tensor rank decomposition problem, because the parameters $\vect{p}$ change relatively less than the tensor. 
However, if we apply the worst possible perturbation from Part III of the proof of \refthm{thm_condition} by executing the following code\footnote{The \texttt{flatten} function is also included in the ancillary files and essentially implements the $\operatorname{vecr}$ map.}
\begin{verbatim}
[~,F,Tr] = cpdcond({A,B,C});
[~,~,V] = svd(Tr); v = V(:,2*(3+3+2-2));
Fbad = {...
  F{1}+1e-8*[v(1:3) v(9:11)], ...
  F{2}+1e-8*[v(4:6) v(12:14)], ...
  F{3}+1e-8*[v(7:8) v(15:16)] ...
};
[~,Fbad,~] = cpdcond(Fbad);
p = flatten(F); pp = flatten(Fbad); T = cpdgen(F);
goodCondApp = (frob(pp-p)/frob(p)) / (frob(T - cpdgen(Fbad))/frob(T))
\end{verbatim}
then we get a substantially more accurate estimate of the condition number, namely $18.410307$---a relative difference of only about $5\cdot10^{-8}$ with respect to the true condition number.

\subsection{The elementary properties}
Let us numerically check \refprop{prop_rank1}. The condition numbers of rank-$1$ tensors of successively higher order behave as $\sqrt{d^{-1}}$, as evidenced by the next experiment:
\begin{verbatim}
a = orth(randn(2,1)); b = orth(randn(3,1)); c = orth(randn(4,1));
Conds = zeros(8,3);
for k = 3 : 10
  Conds(k-2,1) = cpdcond(arrayfun(@(x) 1*a, 1:k, 'UniformOutput', false));
  Conds(k-2,2) = cpdcond(arrayfun(@(x) 2*b, 1:k, 'UniformOutput', false));
  Conds(k-2,3) = cpdcond(arrayfun(@(x) 3*c, 1:k, 'UniformOutput', false));
end
norm( Conds.^(-2) - (3:10)'*ones(1,3) )
\end{verbatim}
Matlab's output was \verb|1.0272e-11|, illustrating that the relative condition number of a rank-$1$ tensor indeed only depends on the order of the tensor and, in particular, does not depend on the dimensions $n_i$ of the tensor and its norm.

From the expression in \refprop{prop_weak3orth} it is not difficult to derive that ill-conditioned weak $3$-orthogonal tensors can be constructed in two principal ways, namely either by letting $\alpha_r \to 0$ while holding the other $\alpha_i$'s constant, or by letting $\alpha_1 \to \infty$ while keeping the others constant. Observe that these two strategies are equivalent because the relative condition number is invariant under scaling. Consider for example the following weak $3$-orthogonal (and odeco) tensors
\[
 \tensor{A}_s = 10^{-s} \sten{u}{ }{1} \otimes \sten{u}{ }{2} \otimes \sten{u}{ }{3} + \sten{v}{ }{1} \otimes \sten{v}{ }{2} \otimes \sten{v}{ }{3},
\]
where $\langle \sten{u}{ }{j}, \sten{v}{ }{j} \rangle = 0$ for $j = 1,2,3$. The relative condition numbers for successive values of $s$ are
\[\small{}
\begin{array}{ccccccccccccccccc}
\toprule
s & 0 & 1 & 2 & 3 & 4 & 5 & 6 & 7 & 8 & 9 & 10 & 11 & 12 & 13 & 14 & 15 \\
\midrule
\log_{10}(\kappa(\tensor{A}_s))
 & -0.2 & 0.4 & 1.1 & 1.8 & 2.4 & 3.1 & 3.8 & 4.4 & 5.1 & 5.8 & 6.4 & 7.1 & 7.8 & 8.4 & 9.1 & 9.8 \\
\bottomrule
\end{array}
\]
Note that $\log_{10}(\kappa(\tensor{A}_s)) \approx -0.2 + 2s/3$ as can be understood by applying \refprop{prop_weak3orth} to this example. 

\subsection{An ill-conditioned example}\label{sec_illcond_example}
Consider the following sequence of tensors
\begin{align}\label{eqn_bad_sequence}
\tensor{A}_s = f(\vect{p}_s) = \sten{a}{1}{1} \otimes \sten{a}{1}{2} \otimes (\vect{x} + 2^{-s} \sten{a}{1}{3}) + \sten{a}{2}{1}\otimes\sten{a}{2}{2}\otimes(\vect{x} + 2^{-s} \sten{a}{2}{3}),
\end{align}
where $\sten{a}{i}{k} \in \F^{n_k}$ are random vectors sampled from a Gaussian distribution, $\vect{x} \in \F^{n_3}$ is any vector, and $\vect{p}_s$ are the norm-balanced vectorized factor matrices. Every tensor $\tensor{A}_s$ on this sequence is robustly $2$-identifiable by Kruskal's theorem with probability $1$. However, 
\[
 \lim_{s\to\infty} \tensor{A}_s = (\sten{a}{1}{1} \otimes \sten{a}{1}{2} + \sten{a}{2}{1}\otimes\sten{a}{2}{2}) \otimes \vect{x}
\]
which is not a $2$-identifiable tensor because 
\[
 (\sten{x}{1}{1} \otimes \sten{x}{1}{2} + \sten{x}{2}{1}\otimes\sten{x}{2}{2})\otimes \vect{x} \quad\text{with}\quad
 \begin{bmatrix}
  \sten{x}{1}{1} & \sten{x}{2}{1}
 \end{bmatrix} = 
 \begin{bmatrix}
  \sten{a}{1}{1} & \sten{a}{2}{1}
 \end{bmatrix} Z 
 \;\text{ and }\;
 \begin{bmatrix}
  \sten{x}{1}{2} & \sten{x}{2}{2}
 \end{bmatrix} =
 \begin{bmatrix}
  \sten{a}{1}{2} & \sten{a}{2}{2}
 \end{bmatrix} Z^{-1}
\]
for any invertible $Z \in \F^{2 \times 2}$ is an alternative decomposition. Hence, $\tensor{A}_\infty$ has infinitely many decompositions of length $2$, yet it is also the limit of a sequence of $2$-identifiable tensors.\footnote{There is nothing special about the coefficient $2^{-s}$ in front of $\sten{a}{i}{3}$; one could equally well have replaced it with any decreasing function of $s$---even different coefficients in front of $\sten{a}{1}{3}$ and $\sten{a}{2}{3}$ would yield the same type of behavior, i.e., a sequence of $2$-identifiable tensors converging to a tensor with an infinite number of decompositions.}

Because $\tensor{A}_s$ is $2$-identifiable and its rank is not larger than any of its multilinear ranks, it follows that a direct decomposition algorithm such as those in \cite{LRA1993,dL2006,DdL2014,DdL2015b} may be employed for computing $\tensor{A}_s$'s decomposition. 
Let $\widehat{\vect{p}}_s \in f^\dagger(\tensor{A}_s)$ be the norm-balanced set of vectorized factor matrices that were computed by some numerical algorithm for solving the tensor rank decomposition problem (with rank equal to $2$ in this case). Note that the numerical algorithm thus computed the decomposition of a nearby tensor $\widehat{\tensor{A}}_s = f(\widehat{\vect{p}}_s)$. The relative backward error is then defined as $\|\tensor{A}_s - \widehat{\tensor{A}}_s\|/\|\tensor{A}_s\|$; it is a measure of how close the rank decomposition problem that was solved by the numerical algorithm was to the true rank decomposition problem. Whenever the backward error of an algorithm is always of the order of the machine precision, then the algorithm is called \emph{backward stable}.
As mentioned in the introduction, since one does not know the true solution in a practical setting, the relative forward error $d(\vect{p}_s, f^\dagger(\tensor{A}_s)) \| / \|\vect{p}_s\|$---which is I believe to be the quantity of interest in any data-analysis application where the individual rank-$1$ tensors appearing in the tensor rank decomposition are to be interpreted---cannot be evaluated. However, the condition number proposed in this paper allows us to (asymptotically) bound the relative forward error from above, because of the relation
\begin{align}\label{eqn_forward_bound}
 \frac{d(\vect{p}_s, \widehat{\vect{p}}_s)}{\|\vect{p}_s\|} \lesssim \kappa(\tensor{A}_s) \cdot \frac{\| \tensor{A}_s - \widehat{\tensor{A}}_s\|}{\|\tensor{A}_s\|}.
\end{align}
In the present experiment, we will investigate these four quantities, i.e., the relative forward error, the condition number, the relative backward error, and the estimated upper bound on the right hand side in the last inequality.

As a concrete experiment, let us investigate $13 \times 11 \times 7$ tensors whose factor matrices \verb|F(s)| were generated as follows:
\begin{verbatim}
A = randn(13,2); B = randn(11,2); C = randn(7,2); x = randn(7,1);
F = @(s) {A, B, x*[1 1] + 2^(-s) * C};
\end{verbatim}
I will use the \verb|cpd_gevd| algorithm from Tensorlab for computing a direct decomposition of $\tensor{A}_s$ for the values $s = 1, 2, \ldots, 45$. For simplicity, I will use $\|\vect{p}_s - \widehat{\vect{p}}_s\|$ as a proxy for the forward error $d(\vect{p}_s, \widehat{\vect{p}}_s)$; by definition, the former is an upper bound on the latter. The backward error, the proxy of the forward error, and estimated forward error obtained by multiplying the condition number with the backward error are all plotted in function of $s$ in \reffig{fig_backforecond}.

\begin{figure}
\caption{The relative backward error, relative forward error, and relative condition number of a particular instance of the sequence in \refeqn{eqn_bad_sequence}.}
\label{fig_backforecond}
\begingroup
  \makeatletter
  \providecommand\color[2][]{%
    \GenericError{(gnuplot) \space\space\space\@spaces}{%
      Package color not loaded in conjunction with
      terminal option `colourtext'%
    }{See the gnuplot documentation for explanation.%
    }{Either use 'blacktext' in gnuplot or load the package
      color.sty in LaTeX.}%
    \renewcommand\color[2][]{}%
  }%
  \providecommand\includegraphics[2][]{%
    \GenericError{(gnuplot) \space\space\space\@spaces}{%
      Package graphicx or graphics not loaded%
    }{See the gnuplot documentation for explanation.%
    }{The gnuplot epslatex terminal needs graphicx.sty or graphics.sty.}%
    \renewcommand\includegraphics[2][]{}%
  }%
  \providecommand\rotatebox[2]{#2}%
  \@ifundefined{ifGPcolor}{%
    \newif\ifGPcolor
    \GPcolortrue
  }{}%
  \@ifundefined{ifGPblacktext}{%
    \newif\ifGPblacktext
    \GPblacktexttrue
  }{}%
  \let\gplgaddtomacro\g@addto@macro
  \gdef\gplbacktext{}%
  \gdef\gplfronttext{}%
  \makeatother
  \ifGPblacktext
    \def\colorrgb#1{}%
    \def\colorgray#1{}%
  \else
    \ifGPcolor
      \def\colorrgb#1{\color[rgb]{#1}}%
      \def\colorgray#1{\color[gray]{#1}}%
      \expandafter\def\csname LTw\endcsname{\color{white}}%
      \expandafter\def\csname LTb\endcsname{\color{black}}%
      \expandafter\def\csname LTa\endcsname{\color{black}}%
      \expandafter\def\csname LT0\endcsname{\color[rgb]{1,0,0}}%
      \expandafter\def\csname LT1\endcsname{\color[rgb]{0,1,0}}%
      \expandafter\def\csname LT2\endcsname{\color[rgb]{0,0,1}}%
      \expandafter\def\csname LT3\endcsname{\color[rgb]{1,0,1}}%
      \expandafter\def\csname LT4\endcsname{\color[rgb]{0,1,1}}%
      \expandafter\def\csname LT5\endcsname{\color[rgb]{1,1,0}}%
      \expandafter\def\csname LT6\endcsname{\color[rgb]{0,0,0}}%
      \expandafter\def\csname LT7\endcsname{\color[rgb]{1,0.3,0}}%
      \expandafter\def\csname LT8\endcsname{\color[rgb]{0.5,0.5,0.5}}%
    \else
      \def\colorrgb#1{\color{black}}%
      \def\colorgray#1{\color[gray]{#1}}%
      \expandafter\def\csname LTw\endcsname{\color{white}}%
      \expandafter\def\csname LTb\endcsname{\color{black}}%
      \expandafter\def\csname LTa\endcsname{\color{black}}%
      \expandafter\def\csname LT0\endcsname{\color{black}}%
      \expandafter\def\csname LT1\endcsname{\color{black}}%
      \expandafter\def\csname LT2\endcsname{\color{black}}%
      \expandafter\def\csname LT3\endcsname{\color{black}}%
      \expandafter\def\csname LT4\endcsname{\color{black}}%
      \expandafter\def\csname LT5\endcsname{\color{black}}%
      \expandafter\def\csname LT6\endcsname{\color{black}}%
      \expandafter\def\csname LT7\endcsname{\color{black}}%
      \expandafter\def\csname LT8\endcsname{\color{black}}%
    \fi
  \fi
  \setlength{\unitlength}{0.0500bp}%
  \scalebox{0.90}{
  \begin{picture}(8640.00,3528.00)%
    \gplgaddtomacro\gplbacktext{%
      \csname LTb\endcsname%
      \put(1122,704){\makebox(0,0)[r]{\strut{} $10^{-16}$}}%
      \csname LTb\endcsname%
      \put(1122,1024){\makebox(0,0)[r]{\strut{} $10^{-14}$}}%
      \csname LTb\endcsname%
      \put(1122,1344){\makebox(0,0)[r]{\strut{} $10^{-12}$}}%
      \csname LTb\endcsname%
      \put(1122,1664){\makebox(0,0)[r]{\strut{} $10^{-10}$}}%
      \csname LTb\endcsname%
      \put(1122,1984){\makebox(0,0)[r]{\strut{} $10^{-8}$}}%
      \csname LTb\endcsname%
      \put(1122,2303){\makebox(0,0)[r]{\strut{} $10^{-6}$}}%
      \csname LTb\endcsname%
      \put(1122,2623){\makebox(0,0)[r]{\strut{} $10^{-4}$}}%
      \csname LTb\endcsname%
      \put(1122,2943){\makebox(0,0)[r]{\strut{} $10^{-2}$}}%
      \csname LTb\endcsname%
      \put(1122,3263){\makebox(0,0)[r]{\strut{} $1$}}%
      \csname LTb\endcsname%
      \put(1254,484){\makebox(0,0){\strut{} 0}}%
      \csname LTb\endcsname%
      \put(2030,484){\makebox(0,0){\strut{} 5}}%
      \csname LTb\endcsname%
      \put(2807,484){\makebox(0,0){\strut{} 10}}%
      \csname LTb\endcsname%
      \put(3583,484){\makebox(0,0){\strut{} 15}}%
      \csname LTb\endcsname%
      \put(4360,484){\makebox(0,0){\strut{} 20}}%
      \csname LTb\endcsname%
      \put(5136,484){\makebox(0,0){\strut{} 25}}%
      \csname LTb\endcsname%
      \put(5913,484){\makebox(0,0){\strut{} 30}}%
      \csname LTb\endcsname%
      \put(6689,484){\makebox(0,0){\strut{} 35}}%
      \csname LTb\endcsname%
      \put(7466,484){\makebox(0,0){\strut{} 40}}%
      \csname LTb\endcsname%
      \put(8242,484){\makebox(0,0){\strut{} 45}}%
      \put(4748,154){\makebox(0,0){\strut{}$s$}}%
    }%
    \gplgaddtomacro\gplfronttext{%
      \csname LTb\endcsname%
      \put(4686,3090){\makebox(0,0)[r]{\strut{}Backward error}}%
      \csname LTb\endcsname%
      \put(4686,2870){\makebox(0,0)[r]{\strut{}Forward error}}%
      \csname LTb\endcsname%
      \put(4686,2650){\makebox(0,0)[r]{\strut{}Forward error upper bound}}%
    }%
    \gplbacktext
    \put(0,0){\includegraphics{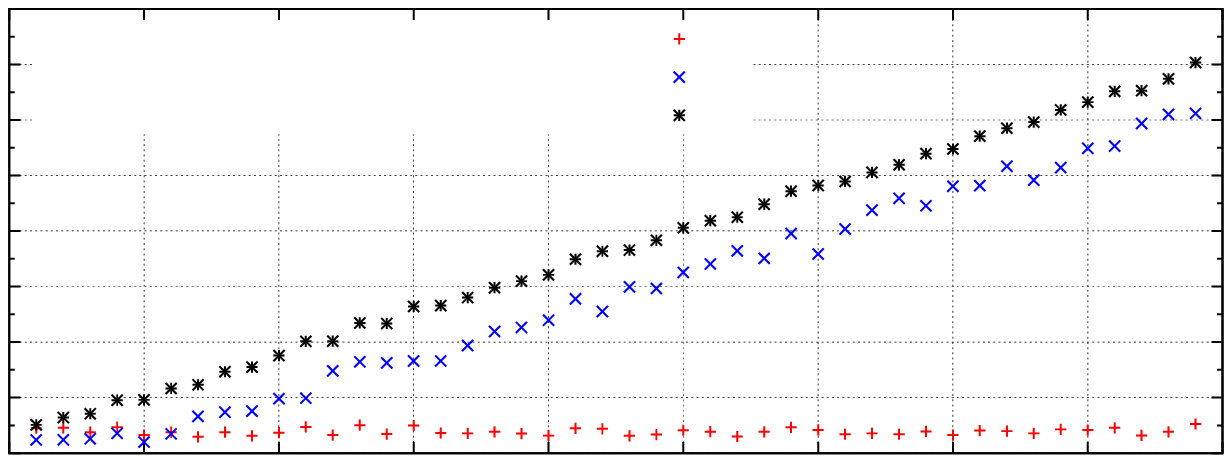}}%
    \gplfronttext
  \end{picture}%
  }
\endgroup
\end{figure}

I believe that \reffig{fig_backforecond} clearly illustrates the danger of simply investigating the relative backward error. One sees that the backward error of the generalized eigenvalue decomposition algorithm is of the order of the machine precision so long as the tensor $\tensor{A}_s$ is sufficiently\footnote{When $s=45$, Tensorlab's implementation of the ST-HOSVD \cite{Vannieuwenhoven2012a} algorithm detects that the multilinear rank of $\tensor{A}_s$ is very close to $(2,2,1)$, at which point the software prudently refuses to employ the generalized eigenvalue decomposition algorithm.} discernible from $\tensor{A}_\infty$; for the latter both the generalized eigendecomposition and the tensor rank decomposition no longer admit a unique solution. The forward error, however, may be several orders of magnitude larger. In this example, it can be observed that the forward error grows as $s$ increases because of the increase of the condition number from roughly $\mathcal{O}(1)$ to $\mathcal{O}(10^{13})$. 

This experiment seems to suggest that the condition number strongly increases as one moves closer to a tensor with $\infty$ many decompositions. An intuitive explanation is that an identifiable tensor $\tensor{A}_\epsilon$, whose norm-balanced vectorized factor matrices are given by $\vect{p}_\epsilon$, at relative distance $\epsilon$ to a tensor with $\infty$ many decompositions $\tensor{A}$ needs but a perturbation of magnitude $\epsilon$ for \emph{potentially} increasing the relative forward error to $\mathcal{O}(1)$. Then,
\[
 \kappa(\tensor{A}_\epsilon) \gtrsim \frac{ d(\vect{p}_\epsilon, f^\dagger(\tensor{A}))/\|\vect{p}\| }{\|\tensor{A}-\tensor{A}_\epsilon\|/\|\tensor{A}_\epsilon\|} = \mathcal{O}(1) \cdot \frac{\|\tensor{A}_\epsilon\|}{\|\vect{p}_\epsilon\|} \cdot \epsilon^{-1},
\]
which diverges to infinity if the fraction in the right hand side remains constant as $\epsilon\to0$.
Further research is necessary for proving this theoretically.

\subsection{Two ill-posed examples}\label{sec_illposed_example}
According to Demmel \cite{Demmel1987}, a problem that is close to an ill-posed problem is often ill-conditioned. We will investigate numerically whether this property is also admitted by the proposed condition number, starting with the well-known example of de Silva and Lim \cite{Silva2008}:
\begin{align}\label{eqn_illposed1}
 \tensor{A}_s = f(\vect{p}_s) = 2^{s/5} (\sten{a}{}{1} + 2^{-s/5} \sten{b}{}{1}) \otimes (\sten{a}{}{2} + 2^{-s/5} \sten{b}{ }{2}) \otimes (\sten{a}{ }{3} + 2^{-s/5} \sten{b}{ }{3}) - 2^{s/5} \sten{a}{ }{1} \otimes \sten{a}{ }{2} \otimes \sten{a}{ }{3},
\end{align}
where $\sten{a}{ }{k}, \sten{b}{ }{k} \in \F^{n_k}$ are linearly independent vectors. Every tensor on this sequence is robustly $2$-identifiable by Kruskal's criterion ($2 \le \tfrac{1}{2}(2+2+2-2)$), and the limit of this sequence of rank-$2$ tensors is the rank-$3$ tensor
\[
 \lim_{s\to\infty} \tensor{A}_s = \sten{b}{ }{1}\otimes\sten{a}{ }{2}\otimes\sten{a}{ }{3} + \sten{a}{ }{1}\otimes\sten{b}{ }{2}\otimes\sten{a}{ }{3} + \sten{a}{}{1}\otimes\sten{a}{ }{2}\otimes\sten{b}{ }{3}.
\]
As in the foregoing ill-conditioned example, the tensor rank decomposition of $\tensor{A}_s$ can be computed using a direct algorithm, such as \verb|cpd_gevd|. Let $\widehat{\vect{p}}_s$ denote the set of norm-balanced vectorized factor matrices obtained from applying this numerical algorithm to $\tensor{A}_s$, and let $\widehat{\tensor{A}}_s = f(\widehat{\vect{p}}_s)$.

As a particular instance, I generated vectors $\sten{a}{ }{1}, \sten{b}{ }{1} \in \R^5$, $\sten{a}{ }{2}, \sten{b}{ }{2} \in \R^4$ and $\sten{a}{ }{3}, \sten{b}{ }{3}\in\R^3$ all of whose entries were sampled from a standard normal distribution. Then, for all $s = 5, 6, \ldots, 100$, the rank-$2$ decomposition was computed with \verb|cpd_gevd|, recording the relative backward error, the proxy of the relative forward error, and the relative condition number at $\widehat{\tensor{A}}_s$ multiplied with the relative backward error. These quantities are plotted in \reffig{fig_illposed1}.

\begin{figure}
\caption{The relative backward error, relative forward error, and relative condition number of a particular instance of the sequence in \refeqn{eqn_illposed1}.}
\label{fig_illposed1}
\begingroup
  \makeatletter
  \providecommand\color[2][]{%
    \GenericError{(gnuplot) \space\space\space\@spaces}{%
      Package color not loaded in conjunction with
      terminal option `colourtext'%
    }{See the gnuplot documentation for explanation.%
    }{Either use 'blacktext' in gnuplot or load the package
      color.sty in LaTeX.}%
    \renewcommand\color[2][]{}%
  }%
  \providecommand\includegraphics[2][]{%
    \GenericError{(gnuplot) \space\space\space\@spaces}{%
      Package graphicx or graphics not loaded%
    }{See the gnuplot documentation for explanation.%
    }{The gnuplot epslatex terminal needs graphicx.sty or graphics.sty.}%
    \renewcommand\includegraphics[2][]{}%
  }%
  \providecommand\rotatebox[2]{#2}%
  \@ifundefined{ifGPcolor}{%
    \newif\ifGPcolor
    \GPcolortrue
  }{}%
  \@ifundefined{ifGPblacktext}{%
    \newif\ifGPblacktext
    \GPblacktexttrue
  }{}%
  \let\gplgaddtomacro\g@addto@macro
  \gdef\gplbacktext{}%
  \gdef\gplfronttext{}%
  \makeatother
  \ifGPblacktext
    \def\colorrgb#1{}%
    \def\colorgray#1{}%
  \else
    \ifGPcolor
      \def\colorrgb#1{\color[rgb]{#1}}%
      \def\colorgray#1{\color[gray]{#1}}%
      \expandafter\def\csname LTw\endcsname{\color{white}}%
      \expandafter\def\csname LTb\endcsname{\color{black}}%
      \expandafter\def\csname LTa\endcsname{\color{black}}%
      \expandafter\def\csname LT0\endcsname{\color[rgb]{1,0,0}}%
      \expandafter\def\csname LT1\endcsname{\color[rgb]{0,1,0}}%
      \expandafter\def\csname LT2\endcsname{\color[rgb]{0,0,1}}%
      \expandafter\def\csname LT3\endcsname{\color[rgb]{1,0,1}}%
      \expandafter\def\csname LT4\endcsname{\color[rgb]{0,1,1}}%
      \expandafter\def\csname LT5\endcsname{\color[rgb]{1,1,0}}%
      \expandafter\def\csname LT6\endcsname{\color[rgb]{0,0,0}}%
      \expandafter\def\csname LT7\endcsname{\color[rgb]{1,0.3,0}}%
      \expandafter\def\csname LT8\endcsname{\color[rgb]{0.5,0.5,0.5}}%
    \else
      \def\colorrgb#1{\color{black}}%
      \def\colorgray#1{\color[gray]{#1}}%
      \expandafter\def\csname LTw\endcsname{\color{white}}%
      \expandafter\def\csname LTb\endcsname{\color{black}}%
      \expandafter\def\csname LTa\endcsname{\color{black}}%
      \expandafter\def\csname LT0\endcsname{\color{black}}%
      \expandafter\def\csname LT1\endcsname{\color{black}}%
      \expandafter\def\csname LT2\endcsname{\color{black}}%
      \expandafter\def\csname LT3\endcsname{\color{black}}%
      \expandafter\def\csname LT4\endcsname{\color{black}}%
      \expandafter\def\csname LT5\endcsname{\color{black}}%
      \expandafter\def\csname LT6\endcsname{\color{black}}%
      \expandafter\def\csname LT7\endcsname{\color{black}}%
      \expandafter\def\csname LT8\endcsname{\color{black}}%
    \fi
  \fi
  \setlength{\unitlength}{0.0500bp}%
  \scalebox{0.90}{
  \begin{picture}(8640.00,3528.00)%
    \gplgaddtomacro\gplbacktext{%
      \csname LTb\endcsname%
      \put(1122,839){\makebox(0,0)[r]{\strut{} $10^{-16}$}}%
      \csname LTb\endcsname%
      \put(1122,1108){\makebox(0,0)[r]{\strut{} $10^{-14}$}}%
      \csname LTb\endcsname%
      \put(1122,1377){\makebox(0,0)[r]{\strut{} $10^{-12}$}}%
      \csname LTb\endcsname%
      \put(1122,1647){\makebox(0,0)[r]{\strut{} $10^{-10}$}}%
      \csname LTb\endcsname%
      \put(1122,1916){\makebox(0,0)[r]{\strut{} $10^{-8}$}}%
      \csname LTb\endcsname%
      \put(1122,2186){\makebox(0,0)[r]{\strut{} $10^{-6}$}}%
      \csname LTb\endcsname%
      \put(1122,2455){\makebox(0,0)[r]{\strut{} $10^{-4}$}}%
      \csname LTb\endcsname%
      \put(1122,2724){\makebox(0,0)[r]{\strut{} $10^{-2}$}}%
      \csname LTb\endcsname%
      \put(1122,2994){\makebox(0,0)[r]{\strut{} $1$}}%
      \csname LTb\endcsname%
      \csname LTb\endcsname%
      \put(1691,484){\makebox(0,0){\strut{} 10}}%
      \csname LTb\endcsname%
      \put(2419,484){\makebox(0,0){\strut{} 20}}%
      \csname LTb\endcsname%
      \put(3147,484){\makebox(0,0){\strut{} 30}}%
      \csname LTb\endcsname%
      \put(3875,484){\makebox(0,0){\strut{} 40}}%
      \csname LTb\endcsname%
      \put(4602,484){\makebox(0,0){\strut{} 50}}%
      \csname LTb\endcsname%
      \put(5330,484){\makebox(0,0){\strut{} 60}}%
      \csname LTb\endcsname%
      \put(6058,484){\makebox(0,0){\strut{} 70}}%
      \csname LTb\endcsname%
      \put(6786,484){\makebox(0,0){\strut{} 80}}%
      \csname LTb\endcsname%
      \put(7514,484){\makebox(0,0){\strut{} 90}}%
      \csname LTb\endcsname%
      \put(8242,484){\makebox(0,0){\strut{} 100}}%
      \put(4748,154){\makebox(0,0){\strut{}$s$}}%
    }%
    \gplgaddtomacro\gplfronttext{%
      \csname LTb\endcsname%
      \put(4686,3090){\makebox(0,0)[r]{\strut{}Backward error}}%
      \csname LTb\endcsname%
      \put(4686,2870){\makebox(0,0)[r]{\strut{}Forward error}}%
      \csname LTb\endcsname%
      \put(4686,2650){\makebox(0,0)[r]{\strut{}Forward error upper bound}}%
    }%
    \gplbacktext
    \put(0,0){\includegraphics{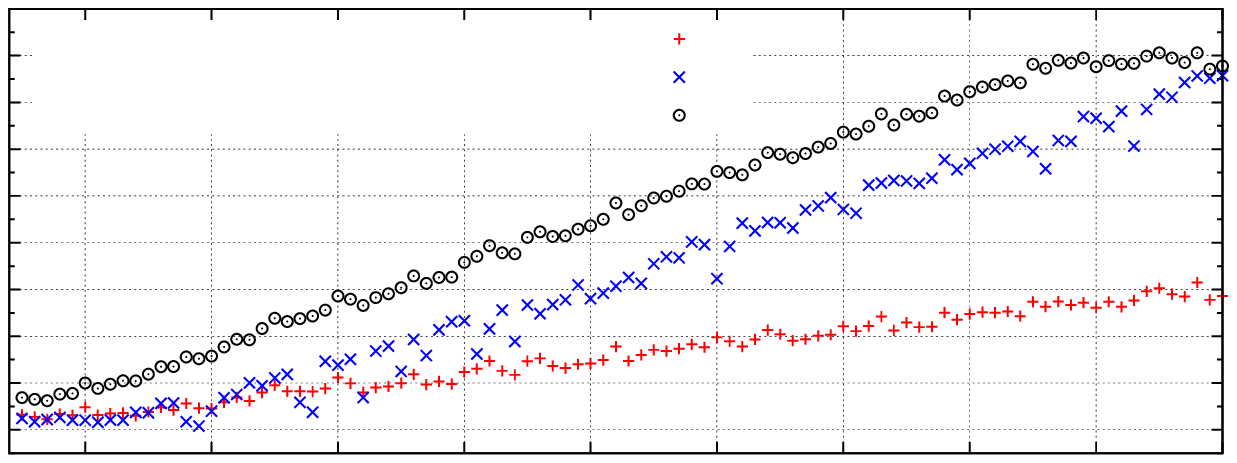}}%
    \gplfronttext
  \end{picture}%
  }
\endgroup
\end{figure}

It is again clear from \reffig{fig_illposed1} that the backward error is not a good measure of the stability of the vectorized factor matrices, as the forward error is several orders of magnitude larger for large $s$. Contrary to the ill-conditioned example from the foregoing subsection, the backward error is also increasing significantly. It turns out that the increase of the backward error is quite strongly correlated with the fraction $\|\widehat{\vect{p}}_s\|/\|\widehat{\tensor{A}}_s\|$ with a correlation coefficient of about $0.76$.

From the figure it can be deduced that the condition number increases from $\mathcal{O}(1)$ to $\mathcal{O}(10^{10})$ as $s$ increases from $5$ to $90$, suggesting that the conditioning of the problem deteriorates as one moves closer to the ill-posed tensor rank decomposition problem. Notice that the upper bound on the forward error seems to stagnate around $s = 85$. It should be stressed that this is a numerical issue in the computation of the condition number. Employing variable precision arithmetic, it can be verified that the condition number keeps increasing as $s$ increases. 
Fortunately, the occurrence of these numerical difficulties in computing the condition number can be detected by verifying that the $N$th singular value of Terracini's matrix is larger than a small constant multiple of the machine precision---$\epsilon_\text{mach} \approx 2.2\cdot10^{-16}$ using standard double precision in Matlab---multiplied with the largest singular value $\sigma_1$ of Terracini's matrix. The basic implementation that is included with this manuscript implements this test and issues a warning if the $N$th singular value is larger than $100 \cdot \epsilon_\text{mach} \cdot \sigma_1$. In this particular example, the code issues warnings starting from $s \ge 75$. Although I have no formal proof, I believe that it is reasonable to assume that when the computation of the condition number starts to suffer from numerical difficulties, the tensor rank decomposition problem is likely ill-conditioned.

It seems that the behavior of the condition number is very similar in other known examples of sequences tending to an ill-posed tensor rank decomposition problem. I present one more case, namely the example of Paatero \cite[Section 7]{Paatero2000}:
\begin{verbatim}
Ab = randn(n1,3); Bb = randn(n2,3); Cb = randn(n3,3);
A = @(e) [-Ab(:,1)/e - Ab(:,2)/e, Ab(:,1)/e + e*e*Ab(:,3)/2, Ab(:,2)/e];
B = @(e) [-Bb(:,1)/e, Bb(:,1)/e + e*e*Bb(:,2)/2, Bb(:,1)/e + e*e*Bb(:,3)/2];
C = @(e) [-Cb(:,1)/e, Cb(:,1)/e + e*e*Cb(:,2)/2, Cb(:,1)/e + e*e*Cb(:,3)/2];
F = @(s) {A(2^(-s/16)), B(2^(-s/16)), C(2^(-s/16))};
\end{verbatim}
wherein $n_1, n_2, n_3 \ge 3$. I chose \verb|n1 = 5|, \verb|n2 = 4| and \verb|n3 = 3| for the present experiment. Then, every tensor $\tensor{A}_s$ whose factor matrices are given by \verb|F(s)| will be robustly $r$-identifiable by Kruskal's theorem ($3 \le \tfrac{1}{2}(3 + 3 + 3 - 2)$), provided that the matrices \verb|Ab|, \verb|Bb| and \verb|Cb| have linearly independent columns.

\begin{figure}
\caption{The relative backward error, relative forward error, and relative condition number of a particular instance of Paatero's example \cite[Section 7]{Paatero2000}.}
\label{fig_illposed2}
\begingroup
  \makeatletter
  \providecommand\color[2][]{%
    \GenericError{(gnuplot) \space\space\space\@spaces}{%
      Package color not loaded in conjunction with
      terminal option `colourtext'%
    }{See the gnuplot documentation for explanation.%
    }{Either use 'blacktext' in gnuplot or load the package
      color.sty in LaTeX.}%
    \renewcommand\color[2][]{}%
  }%
  \providecommand\includegraphics[2][]{%
    \GenericError{(gnuplot) \space\space\space\@spaces}{%
      Package graphicx or graphics not loaded%
    }{See the gnuplot documentation for explanation.%
    }{The gnuplot epslatex terminal needs graphicx.sty or graphics.sty.}%
    \renewcommand\includegraphics[2][]{}%
  }%
  \providecommand\rotatebox[2]{#2}%
  \@ifundefined{ifGPcolor}{%
    \newif\ifGPcolor
    \GPcolortrue
  }{}%
  \@ifundefined{ifGPblacktext}{%
    \newif\ifGPblacktext
    \GPblacktexttrue
  }{}%
  \let\gplgaddtomacro\g@addto@macro
  \gdef\gplbacktext{}%
  \gdef\gplfronttext{}%
  \makeatother
  \ifGPblacktext
    \def\colorrgb#1{}%
    \def\colorgray#1{}%
  \else
    \ifGPcolor
      \def\colorrgb#1{\color[rgb]{#1}}%
      \def\colorgray#1{\color[gray]{#1}}%
      \expandafter\def\csname LTw\endcsname{\color{white}}%
      \expandafter\def\csname LTb\endcsname{\color{black}}%
      \expandafter\def\csname LTa\endcsname{\color{black}}%
      \expandafter\def\csname LT0\endcsname{\color[rgb]{1,0,0}}%
      \expandafter\def\csname LT1\endcsname{\color[rgb]{0,1,0}}%
      \expandafter\def\csname LT2\endcsname{\color[rgb]{0,0,1}}%
      \expandafter\def\csname LT3\endcsname{\color[rgb]{1,0,1}}%
      \expandafter\def\csname LT4\endcsname{\color[rgb]{0,1,1}}%
      \expandafter\def\csname LT5\endcsname{\color[rgb]{1,1,0}}%
      \expandafter\def\csname LT6\endcsname{\color[rgb]{0,0,0}}%
      \expandafter\def\csname LT7\endcsname{\color[rgb]{1,0.3,0}}%
      \expandafter\def\csname LT8\endcsname{\color[rgb]{0.5,0.5,0.5}}%
    \else
      \def\colorrgb#1{\color{black}}%
      \def\colorgray#1{\color[gray]{#1}}%
      \expandafter\def\csname LTw\endcsname{\color{white}}%
      \expandafter\def\csname LTb\endcsname{\color{black}}%
      \expandafter\def\csname LTa\endcsname{\color{black}}%
      \expandafter\def\csname LT0\endcsname{\color{black}}%
      \expandafter\def\csname LT1\endcsname{\color{black}}%
      \expandafter\def\csname LT2\endcsname{\color{black}}%
      \expandafter\def\csname LT3\endcsname{\color{black}}%
      \expandafter\def\csname LT4\endcsname{\color{black}}%
      \expandafter\def\csname LT5\endcsname{\color{black}}%
      \expandafter\def\csname LT6\endcsname{\color{black}}%
      \expandafter\def\csname LT7\endcsname{\color{black}}%
      \expandafter\def\csname LT8\endcsname{\color{black}}%
    \fi
  \fi
  \setlength{\unitlength}{0.0500bp}%
  \scalebox{0.90}{
  \begin{picture}(8640.00,3528.00)%
    \gplgaddtomacro\gplbacktext{%
      \csname LTb\endcsname%
      \put(1122,839){\makebox(0,0)[r]{\strut{} $10^{-16}$}}%
      \csname LTb\endcsname%
      \put(1122,1108){\makebox(0,0)[r]{\strut{} $10^{-14}$}}%
      \csname LTb\endcsname%
      \put(1122,1377){\makebox(0,0)[r]{\strut{} $10^{-12}$}}%
      \csname LTb\endcsname%
      \put(1122,1647){\makebox(0,0)[r]{\strut{} $10^{-10}$}}%
      \csname LTb\endcsname%
      \put(1122,1916){\makebox(0,0)[r]{\strut{} $10^{-8}$}}%
      \csname LTb\endcsname%
      \put(1122,2186){\makebox(0,0)[r]{\strut{} $10^{-6}$}}%
      \csname LTb\endcsname%
      \put(1122,2455){\makebox(0,0)[r]{\strut{} $10^{-4}$}}%
      \csname LTb\endcsname%
      \put(1122,2724){\makebox(0,0)[r]{\strut{} $10^{-2}$}}%
      \csname LTb\endcsname%
      \put(1122,2994){\makebox(0,0)[r]{\strut{} $1$}}%
      \csname LTb\endcsname%
      \csname LTb\endcsname%
      \put(1665,484){\makebox(0,0){\strut{} 20}}%
      \csname LTb\endcsname%
      \put(2487,484){\makebox(0,0){\strut{} 30}}%
      \csname LTb\endcsname%
      \put(3309,484){\makebox(0,0){\strut{} 40}}%
      \csname LTb\endcsname%
      \put(4131,484){\makebox(0,0){\strut{} 50}}%
      \csname LTb\endcsname%
      \put(4954,484){\makebox(0,0){\strut{} 60}}%
      \csname LTb\endcsname%
      \put(5776,484){\makebox(0,0){\strut{} 70}}%
      \csname LTb\endcsname%
      \put(6598,484){\makebox(0,0){\strut{} 80}}%
      \csname LTb\endcsname%
      \put(7420,484){\makebox(0,0){\strut{} 90}}%
      \csname LTb\endcsname%
      \put(8242,484){\makebox(0,0){\strut{} 100}}%
      \put(4748,154){\makebox(0,0){\strut{}$s$}}%
    }%
    \gplgaddtomacro\gplfronttext{%
      \csname LTb\endcsname%
      \put(4686,3090){\makebox(0,0)[r]{\strut{}Backward error}}%
      \csname LTb\endcsname%
      \put(4686,2870){\makebox(0,0)[r]{\strut{}Forward error}}%
      \csname LTb\endcsname%
      \put(4686,2650){\makebox(0,0)[r]{\strut{}Forward error upper bound}}%
    }%
    \gplbacktext
    \put(0,0){\includegraphics{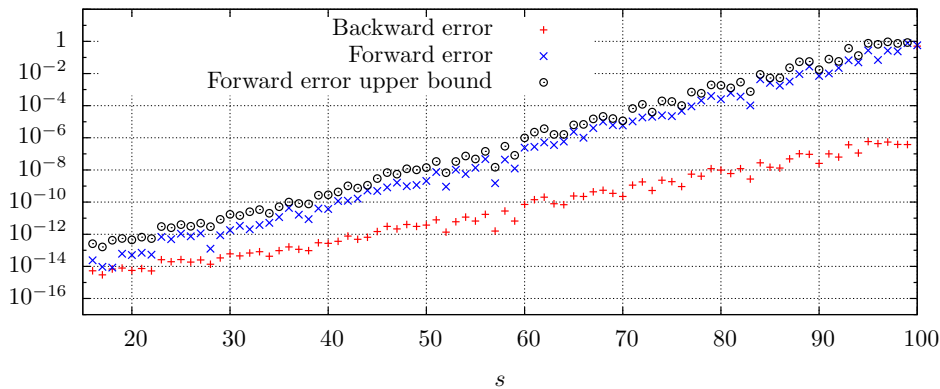}}%
    \gplfronttext
  \end{picture}%
  }
\endgroup
\end{figure}

In \reffig{fig_illposed2}, the evolution of the relative backward error, the proxy of the relative forward error, and the condition number multiplied with the backward error are plotted in function of $s$. The behavior is very similar to the case of de Silva and Lim's example. The relative backward error is increasing from $\mathcal{O}(10^{-16})$ to $\mathcal{O}(10^{-10})$, seemingly driven by the increase of the fraction $\|\widehat{\vect{p}}_s\|/\|\widehat{\tensor{A}}_s\|$---the correlation coefficient is about $0.73$ in this case. Meanwhile, the relative condition number increases from $\mathcal{O}(10^2)$ to $\mathcal{O}(10^{6})$. It is again visually clear that the condition number multiplied with the backward error provides a more accurate estimate of the forward error than the backward error does. 
For $s=100$, the code warned about a possibly inaccurate computation of the condition number. Also note that the \verb|cpd_gevd| failed to compute a reasonable decomposition in this case, for some reason unbeknown to me.

These experiments seem to suggest that the condition number of the rank-$r$ decomposition problem strongly increases as one moves closer to a tensor $\tensor{A}$ that can be approximated arbitrarily well by a rank-$r$ tensor, but which does not admit an exact rank-$r$ decomposition. It is known from \cite{Silva2008} that on such a sequence, $\|\vect{p}_s\| \to \infty$, i.e., the norm of the vectorized factor matrices grows without bound. On the other hand, $\|\tensor{A}_s\|$ is bounded by a constant for sufficiently large $s$, because the sequence $\tensor{A}_s \to \tensor{A}$. Since
\[
 \kappa(\tensor{A}_\epsilon) \gtrsim \frac{ d(\vect{p}_\epsilon, f^\dagger(\tensor{A}))/\|\vect{p}_\epsilon\| }{\|\tensor{A}-\tensor{A}_\epsilon\|/\|\tensor{A}_\epsilon\|} = d(\vect{p}_\epsilon, f^\dagger(\tensor{A})) \frac{1}{\|\vect{p}_\epsilon\| \cdot \|\tensor{A}-\tensor{A}_\epsilon\|/\|\tensor{A}_\epsilon\|},
\]
no conclusions can be drawn without understanding the relative speed with which $\|\vect{p}_s\| \to \infty$ and $\|\tensor{A}-\tensor{A}_\epsilon\|/\|\tensor{A}_\epsilon\|\to0$, and what happens to $d(\vect{p}_\epsilon, f^\dagger(\tensor{A}))$---the latter is not even well-defined, because the fiber is empty. The two examples seem to suggest that the condition number becomes unbounded.

%
%
%
%
\section{Conclusions}\label{sec_conclusions}
A condition number for the tensor rank decomposition problem was introduced. 
Provided that the input tensor satisfies a technical assumption called robust $r$-identifiability, the proposed condition number admits a particularly simple interpretation as the asymptotically worst fraction between the relative forward and relative backward error. 
The investigated condition number corresponds to a multiple of the inverse of a certain singular value of Terracini's matrix. Several basic properties of the norm-balanced condition number were investigated. The main conclusions of this paper are that 
\begin{itemize}
 \item Terracini's matrix is the key to investigating conditioning of robustly $r$-identifiable tensor rank decompositions;
 \item the condition number multiplied with the backward error yields an asymptotic upper bound on the forward error;
 \item rank-$1$ tensors are always well-conditioned;
 \item weak $3$-orthogonal tensors (which include odeco tensors) are well-conditioned when the norms of the individual rank-$1$ terms are roughly of the same order of magnitude; and 
 \item weak $3$-orthogonal tensors are ill-conditioned when several orders of magnitude of difference exists in the norms of the rank-$1$ terms.
\end{itemize}
In addition, the numerical experiments provide some preliminary evidence suggesting that robustly $r$-identifiable tensors that are close to 
\begin{itemize}
 \item tensors with $\infty$ many rank-$r$ decompositions (see \refsec{sec_illcond_example}), or 
 \item tensors with border rank $r$ but rank strictly larger (see \refsec{sec_illposed_example})
\end{itemize}
can be ill-conditioned. 

I hope that this paper opens several avenues of further research as well as practical application. In particular, I hope that it may find application in a more rigorous stability analysis of computed tensor rank decompositions. I believe that it may assist in understanding the convergence of gradient-based optimization methods for computing approximate tensor rank decompositions. Another interesting question concerns the conditioning of the tensor rank \emph{approximation} problem, whereby the sensitivity of the factor matrices is investigated with respect to a general, unstructured perturbation of a tensor. It is also an open question whether the behavior observed in Sections \ref{sec_illcond_example} and \ref{sec_illposed_example} may be explained theoretically. Finally, it may be anticipated that the analysis and techniques presented in this paper may also be applied in more structured variants of the tensor rank decomposition, such as Waring decompositions of symmetric tensors \cite{Comon2008}---which are provably generically $r$-identifiable for all subgeneric symmetric ranks, and thus generically robustly $r$-identifiable, because of \cite{COV2015}---or of tensor rank decompositions of partially symmetric tensors.

\clearpage
\appendix
\section{Matlab/Octave program code} \label{app_code}

\lstset{language=Matlab,basicstyle=\footnotesize,showstringspaces=false}
\begin{lstlisting}
function [ kappa, A, T ] = cpdcond( A, ~ )
%CPDCOND Computes the condition number. 
%  The routine estimates the norm-balanced condition number of the factor
%   matrices A. If the second input argument is specified, the absolute
%   condition number is computed. 
%  The first output argument is the norm-balanced condition number.
%  The second output argument are the norm-balanced factor matrices with
%   respect to which the condition number is computed.
%  The third output argument contains Terracini's matrix of the
%   norm-balanced representatives.

% Constants
d = length(A);
n = arrayfun(@(k) size(A{k},1), 1:d);
r = size(A{1},2);
Pi = prod(n);
Sigma = sum(n-1);

%% Apply norm-balancing
scal = ones(1,r);
for k = 1 : d
    lscal = arrayfun(@(l) norm(A{k}(:,l)), 1:r);
    scal = scal .* lscal.^(1/d);
    A{k} = A{k} * diag(1./lscal);
end
for k = 1 : d
    A{k} = A{k} * diag(scal);
end
normP = sqrt(d)*norm(scal);

% Condition number is unbounded by definition.
if r*(1+Sigma) > Pi
   warning('The number of terms is not subgeneric.')
   kappa = inf;
   T = [];
   return;
end

%% Construct Terracini's matrix
T = zeros(Pi,r*(d+Sigma));
for l = 1 : r
   lT = [];
   for k = 1 : d
       B = arrayfun(...
           @(kk) repmat(A{kk}(:,l),1,n(k)), 1:d, 'UniformOutput', false ...
       );
       B{k} = eye(n(k));
       lT = [lT krp(B)];
   end
   T(:,(d+Sigma)*(l-1)+1:(d+Sigma)*l) = lT;
end

%% Compute relative condition number
N = r*(1+Sigma);
S = svd(T,0);
kappa = 1 / S(N);
if nargin == 1
   tensNorm = norm(krp(A)*ones(r,1));
   kappa = (kappa * tensNorm) / normP;
end
if ( S(1)*100*eps > S(N) )
   warning('Computed condition number may not be accurate!')
end

    % Compute Khatri-Rao product
    function [X] = krp(B)
       Ns = prod(arrayfun(@(k)size(B{k},1),1:length(B)));
       X = zeros(Ns,size(B{1},2));
       for ll = 1 : size(X,2)
           x = B{1}(:,ll);
           for kk = 2 : length(B)
               x = kron(x,B{kk}(:,ll));
           end
           X(:,ll) = x;
       end
    end
end
\end{lstlisting}

\begin{lstlisting}
function [ p, delta ] = isl( p, nabla, n, r )
%ISL Applies the Iterated Scaling algorithm. 
%   The first input argument is the set of vectorized factor matrices.
%   The second input argument is the perturbation vector.
%   The third input argument is an array of integers containing the
%     dimensions of the tensor represented by the vectorized factor
%     matrices in p.
%   The fourth input argument is the number of terms in the rank
%     decomposition of the tensor represented by p.
%   The first output argument is the rescaled set of vectorized factor
%     matrices, representing the same tensor as p.
%   The second output argument is a vector perpendicular to the kernel
%     of Terracini's matrix, whose norm is approximately the square of the
%     norm of nabla.

% Constants
d = length(n);
cn = cumsum(n);
Sigma = cn(d);

% Construct kernel
K = zeros(r*Sigma, r*(d-1));
li = 1;
for i = 1 : r
   pt = p((i-1)*Sigma+1:i*Sigma);
   for k = 2 : d
       off = (i-1)*Sigma;
       K(off+1:off+cn(1),li) = pt(1:cn(1));
       K(off+cn(k-1)+1:off+cn(k),li) = -pt(cn(k-1)+1:cn(k));
       li = li + 1;
   end
end

% Iterated scaling algorithm
p0 = p;
pk = p;
delta = zeros(size(p));
gam = ones(d-1,r);
[U,S,V] = svd(K,0);
nablanorm = inf;
while nablanorm > 10*eps
    v = V*(S\(U'*nabla));
    gam = gam - reshape(v, [d-1 r]);
    zk = pk + K*v;
    off = 0;
    for i = 1 : r
        pk(off+1:off+cn(1)) = p0(off+1:off+cn(1))/prod(gam(:,i));
        off = off + n(1);
        for k = 2 : d
            pk(off+1:off+n(k)) = p0(off+1:off+n(k))*gam(k-1,i);
            off = off + n(k);
        end
    end
    nabla = zk - pk;
    delta = delta + (nabla - U*(U'*nabla));
    nabla = U*(U'*nabla);
    nablanorm = norm(nabla);
end
p = pk;
end
\end{lstlisting}

\begin{lstlisting}
function [ v ] = flatten( A )
%FLATTEN Vectorizes a given set of factor matrices A.

v = [];
for r = 1 : size(A{1},2)
    for k = 1 : length(A);
        v = [v; A{k}(:,r)];
    end
end
end
\end{lstlisting}

\end{document}